\documentclass [10pt,a4paper,reqno]{amsart}

\usepackage[utf8]{inputenc}
\usepackage{amsmath}
\usepackage{lmodern}
\usepackage{amssymb}
\usepackage{xypic}
\usepackage{tikz-cd}
\usepackage{enumerate}
\usepackage{mathtools}

\usepackage{soul}

\usepackage{enumerate,amssymb, xcolor,mathrsfs}
\usepackage[colorlinks]{hyperref}

\usepackage{amsmath,amssymb,amsfonts,xfrac,MnSymbol,graphicx,float}
\usepackage{amsthm}
\usepackage{cite}
\usepackage{hyperref}
\usepackage{todonotes}
\usepackage{enumitem}
\usepackage{graphicx}
\usepackage{color}
\usepackage{import}

\numberwithin{equation}{section}

\newtheorem{theorem}{Theorem}[section]
\newtheorem{claim}[theorem]{Claim}

\newtheorem{lemma}[theorem]{Lemma}
\newtheorem{corollary}[theorem]{Corollary}

\newtheorem*{claim*}{Claim}

\newtheorem{mainthm}{Theorem}

\theoremstyle{definition}
\newtheorem{definition}[theorem]{Definition}

\newtheorem{remark}[theorem]{Remark}
\newtheorem{example}[theorem]{Example}
\newtheorem{question}[theorem]{Question}

\newtheorem*{notation*}{Notation}

\newcommand{\bbN}{\mathbb{N}}

\newcommand{\calB}{\mathcal{B}}
\newcommand{\B}{\mathcal{B}}

\newcommand{\calM}{\mathcal{M}}

\newcommand{\calO}{\mathcal{O}}
\newcommand{\calP}{\mathcal{P}}
\newcommand{\calQ}{\mathcal{Q}}

\newcommand{\calY}{\mathcal{Y}}
\newcommand{\calZ}{\mathcal{Z}}

\DeclareMathOperator{\supp}{supp}
\DeclareMathOperator{\Lab}{Lab}

\DeclareMathOperator{\ord}{ord}
\DeclareMathOperator{\cln}{cl}
\DeclareMathOperator{\code}{code}
\DeclareMathOperator{\pcode}{pcode}

\DeclareMathOperator{\modd}{mod}

\begin{document}

\title{Asymptotic Burnside laws}

\author{Gil Goffer} \address{Department of Mathematics, University of California, San Diego, La Jolla, CA 92093, USA} \email{ggoffer@ucsd.edu}

 \author{Be'eri Greenfeld} \address{Department of Mathematics, University of Washington, Seattle, WA 98195-4350, USA} \email{grnfld@uw.edu}
 
 \author{Alexander Yu. Olshanskii} \address{Department of Mathematics, 1326 Stevenson Center, Vanderbilt University, Nashville, TN 37240, USA} \email{alexander.olshanskiy@Vanderbilt.edu}

\maketitle

\begin{abstract}
We construct novel examples of finitely generated groups that exhibit seemingly-contradicting probabilistic behaviors with respect to Burnside laws. 
We construct a finitely generated group that satisfies a Burnside law, namely a law of the form $x^n=1$, with limit probability 1 with respect to uniform measures on balls in its Cayley graph and under every lazy non-degenerate random walk, while containing a free subgroup. 
We show that the limit probability of satisfying a Burnside law is highly sensitive to the choice of generating set, by providing a group for which this probability is $0$ for one generating set and $1$ for another. 
Furthermore, we construct groups that satisfy Burnside laws of two co-prime exponents with probability 1. Finally, we present a finitely generated group for which every real number in the interval $[0,1]$ appears as a partial limit of the probability sequence of Burnside law satisfaction, both for uniform measures on Cayley balls and for random walks.  

Our results resolve several open questions posed by Amir, Blachar, Gerasimova, and Kozma. The techniques employed in this work draw upon geometric analysis of relations in groups, information-theoretic coding theory on groups, and combinatorial and probabilistic methods.
\end{abstract}

\setcounter{tocdepth}{1}

\section{Introduction}

In group theory, a law is a non-trivial element $w\in F(x_1,x_2,\dots)$ in the free group. Given a group $G$, we say that $G$ satisfies the law $w$, if $w$ vanishes under every substitution of elements from $G$. For instance, the commutator $[x_1,x_2]$ is a group law of any abelian group, the iterated commutator $[x_1,[x_2,[\cdots[x_l,x_{l+1}]]$ is a group law of any $l$-step nilpotent group, and by Lagrange's theorem, every finite group of order $n$ satisfies the law $x^n$.

Suppose that a group $G$ is equipped with a probability measure $\mu$, such as the uniform measure on a finite group, or more generally, the Haar measure on a compact group. 
Given a law $w=w(x_1,\dots,x_d)$, one can now compute the probability that $w$ is satisfied in $G$ under a $\mu$-random substitution $x_1\mapsto g_1,\dots,x_d\mapsto g_d$. 
Namely:
$$\Pr_{G,\mu}(w=1) := \mu \big(\{(g_1,\dots,g_d)\ : \ w(g_1,\dots,g_d)=1\}\big).$$ 
If $G$ or $\mu$ are clear from the context, we omit them from the notation, writing $\Pr_\mu(w=1)$, $\Pr(w=1)$, etc.

In many cases, laws satisfied in $G$ with sufficiently high probability must actually hold in full in $G$, or imply that a similar law holds in $G$. For instance, a classical theorem of Gustafson \cite{Gus} states that if $\Pr([x_1,x_2]=1)>\frac{5}{8}$ in a finite group $G$, then $G$ must be abelian; results in a similar vein have been proven for iterated commutators laws \cite{Eber_Shum}, metabelian laws and Engel laws \cite{DJMN}, and power laws \cite{Laf1,Laf2,Laf3,Mann,Mann_Mart}.

The current article focuses on infinite discrete groups. In the case where $G$ is infinite, 
there is no longer a natural probability measure on $G$ to consider. A common approach \cite{AMV,Tointon,ABGK} is to consider a sequence of probability measures $\vec{\mu}=\{\mu_i\}_{i=1}^{\infty}$ on $G$, whose supports exhaust $G$. 
We say that $G$ satisfies a law $w$ with probability $1$ with respect to $\vec\mu$ if:
$$\lim_{i\rightarrow \infty}
\Pr_{G,\mu_i}(w=1) = 1.$$
Natural sequences $\vec\mu$ to consider (as in \cite{AMV,Tointon,ABGK}) on 
any finitely generated infinite discrete groups include:

\begin{itemize}
    \item \textbf{Uniform measures on balls}. Fix a finite symmetric generating set $S$ of $G$. 
    Denote by $U_S(R)$ the uniform measure on the ball of radius $R$ centered at the identity element in the Cayley graph of $G$ with respect to $S$.

    \item \textbf{Random walks}. Fix a non-degenerate measure $\nu$ on $G$, that is, such that its support $\supp(\nu)$  generates $G$ as a semigroup. Consider a random walk $\{X_i\}_{i=1}^{\infty}$ on $G$ whose step distribution is $\nu$; thus, the distribution of $X_i$ is given by $i$-fold self-convolution of $\nu$, denoted $X_i\sim \nu^{*i}$. In many cases, one also adds the assumption that $X_i$ is \emph{lazy}, that is, $\nu(e)>0$, or \emph{symmetric}, that is, $\nu(g)=\nu(g^{-1})$ for all $g\in G$. 
\end{itemize}
 
It is interesting to notice that while the measure sequences $\{U_S(R)\}_{R=1}^\infty$ and $\{X_i\}_{i=1}^\infty$ do share some similarities, they can behave surprisingly differently from each other. For instance, symmetric, lazy random walks always visit every finite-index subgroup with limit probability reciprocal to its index, while uniform measures on Cayley balls need not (see \cite[Theorem 1.11]{Tointon} and subsequent remark). 
Another interesting notion of probabilistic laws, applicable to residually finite groups, is given by the uniform measures on their finite quotients (for instance, see \cite{LarsenShalev,Shalev} and references therein). See \cite{ABGK} for a thorough discussion of the various notions of probabilistic group laws.

As in the case of finite group, the following broad question naturally arises:
\begin{question}
    Given a group satisfying a law with high probability (or: with probability 1), how close is it to satisfying a genuine group law?
\end{question}
Amir, Blachar, Gerasimova and Kozma showed in \cite[Theorem 6.3]{ABGK} that the wreath product $F_2\wr \mathbb{Z}^5$ admits simple\footnote{A random walk is simple if its step distribution is uniform on some finite set.} non-degenerate random walks for which the metabelian law $[[x,y],[z,w]]$ holds with probability arbitrarily close to $1$, although it evidently contains a free subgroup; however, the constructed random walks change as the probability approaches $1$, which led the authors to ask if a law holding in probability $1$ with respect to a fixed non-degenerate random walk on a given group must be a genuine group law for this group \cite[Question 13.3]{ABGK}. 
Our first main result resolves this question in the negative:

\begin{mainthm}[Asymptotic Burnside group with a free subgroup]\label{thm:Theorem A} 
Let $n\in \mathbb{N}$ be a large enough odd integer. Then there exists a finitely generated group $G=\left<S\right>$ such that: 
$$ \lim_{R\rightarrow \infty} \Pr_{U_S(R)}(x^n=1) = 1 $$
and for \textit{every} non-degenerate random walk 
$\{X_i\}_{i=1}^{\infty}$,  
$$ \lim_{i\rightarrow \infty} \Pr\left((X_i)^n=1\right) = 1; $$ 
yet, $G$ contains a non-abelian free subgroup.
\end{mainthm}
The group $G$ in Theorem \ref{thm:Theorem A} seems to satisfy the law $x^n$ when tested in any reasonable probabilistic sense. However, it does not actually satisfy $x^n$, nor any other group law, as it clearly contains a free subgroup. 
Theorem \ref{thm:Theorem A} provides a negative answer also to \cite[Question 13.1]{ABGK}, showing that large enough Burnside laws do not satisfy `gap phenomena' as established for low powers in \cite{ABGK}. The first statement of Theorem \ref{thm:Theorem A} was independently obtained in \cite{AB,ABM}.

The authors of \cite{ABGK} have further posed the question of whether it is possible for a finitely generated group to satisfy a given law with probability $0$ with respect to one generating set but with positive limit probability with respect to another generating set (\cite[Question 13.4]{ABGK}). It is worth noting that for many significant invariants of random walks, such as entropy and expected speed, it is unknown whether these depend only on the ambient group or can vary across different lazy, symmetric, non-degenerate, finitely supported random walks \cite{ZhengICM}.

The next results answer \cite[Question 13.4]{ABGK} in the positive, in the strongest possible way. For both uniform measures and symmetric random walks, we construct groups which are simultaneously of finite exponent and torsion-free with probability $1$, with respect to different generating sets (resp., different lazy, symmetric, non-degenerate, finitely supported random walks).

\begin{mainthm}[{Contrasting generating sets}] \label{thm:Theorem B} There exists a group $G$, an odd integer $n\in \mathbb{N}$, and two finite symmetric  generating sets $X$ and $Y$ of $G$, such that: $$\lim_{R\to \infty} \Pr_{U_X(R)}(x\ \text{has finite order})=0\ \ \ \text{and}\ \ \ \lim_{R\to \infty} \Pr_{U_Y(R)}(x^n=1)=1.$$
Moreover, every finite order element $x\in G$, satisfies $x^n=1$.
\end{mainthm}

We prove a random walk version of Theorem \ref{thm:Theorem B} in Theorem \ref{thm:Theorem B'}. 
We also prove a similar theorem for two different exponents, constructing a group which simultaneously admits two different, co-prime exponents, each with limit probability $1$:
\begin{mainthm}[Different asymptotic  exponents]\label{thm:two different exponents}
There exists a group $G$, two finite symmetric generating sets $X$ and $Y$ of $G$, and two co-prime odd integers $n_1,n_2\in \mathbb{N}$ such that: $$\lim_{R\to \infty} \Pr_{U_X(R)}(x^{n_1}=1)=1\ \ \ \text{and}\ \ \ \lim_{R\to \infty} \Pr_{U_Y(R)}(x^{n_2}=1)=1.$$
\end{mainthm}

Another natural question is how stable the probability sequence $\{\Pr_{\mu_i}(w=1)\}_{i=1}^{\infty}$ is; in particular, whether its limit always exists, see \cite[Question 13.5]{ABGK}. The following theorem, with which we conclude the paper, answers this question in the negative in the strongest possible sense.
\begin{mainthm}[Highly oscillating probabilities]\label{thm:Theorem D}
There exists a finitely generated group $G$ and an odd number $n$, with a finite generating set $S$ (respectively, with a symmetric, non-degenerate, finitely supported, lazy random walk $\{X_i\}_{i=1}^{\infty}$), such that every real number in $[0,1]$ occurs as a partial limit of the sequence: 
$$\left\{\Pr_{U_S(R)}(x^n=1)\right\}_{R=1}^{\infty}\ \ \ \left(\text{resp.}\ \ \ \Big\{\Pr((X_i)^n=1)\Big\}_{i=1}^{\infty}\right).$$
\end{mainthm}
In fact, we prove a stronger version in Theorem \ref{thm:spectrum strong version}: 
we construct a single group in which, for all odd $N\geq n$, every number in $[0,1]$ is a partial limit of the sequence $\{\Pr_{U_S(R)} (x^N=1)\}_{R=1}^{\infty}$.

\bigskip

The groups we construct, which are finitely generated but not finitely presented, are (Gromov-Hausdorff) limits of hyperbolic groups. Although our
groups are not hyperbolic, and not even lacunary hyperbolic in the sense of \cite{OOS}, this framework allows us to leverage hyperbolic geometry for their local analysis. To achieve this, we use techniques reminiscent of small cancellation theory. However, our groups are not small cancellation bona fide, namely, they do not satisfy effective metric small cancellation conditions.

Our sharp and seemingly paradoxical results (particularly Theorems \ref{thm:two different exponents} and \ref{thm:Theorem D}) depend on highly precise estimates of torsion occurrence. Traditional small cancellation theory, which relies on measuring word length, is inherently limited in handling relations dominated by periodic components. To overcome this limitation, we develop a novel information-theoretic compression method, which we term ``code length." This approach efficiently encodes paths in the Van Kampen diagrams of the constructed groups. 
For instance, we utilize code length to demonstrate that triangles in the resulting (limit) Cayley graphs are generically almost-tripods; that is, statistically, our Cayley graphs locally resemble trees.

By combining this machinery with analyses of random walks, algebraic arguments, and the asymptotic combinatorics of words in free groups, we ultimately achieve the desired results.

\section*{Acknowledgements}
We thank Gidi Amir, Guy Blachar, R\'emi Coulon, Gady Kozma, Tom Meyerovich,  Denis Osin, and Tianyi Zheng for insightful conversations and helpful comments. We thank the referee for a careful reading of the paper and for many corrections and suggestions that improved the readability of the paper. We further thank Anton Klyachko and Igor Lysenok, who brought to our attention an error in the proof of Theorem \ref{thm:Theorem A} in an earlier version of this manuscript.

\section{Burnside-type groups}\label{sec:partial burnside groups}
\subsection{Burnside groups}
The (bounded) Burnside Problem \cite{Burnside1902} asks whether a finitely generated group of finite exponent\footnote{The unbounded version, where every element is required to be torsion, was solved using Golod--Shafarevich algebras, see \cite{Golod}.} must be finite. Equivalently, whether the free Burnside group:
$$ \calB(m,n) = \left<x_1,\dots,x_m\ |\ X^n=1\ \text{for every word}\ X\ \text{in the alphabet}\ x_1,\dots,x_m\right> $$
is finite for every positive integers $n,m\geq 2$. 
This problem has been settled in 1968 by Novikov and Adian \cite{NA} who proved that $\calB(m,n)$ is infinite for sufficiently large odd $n$ and $m$, and for even exponents in \cite{Iv94}; see also \cite{Ad}.
Ol'shanskii \cite{Ol82,OL-book} introduced a geometric approach to this problem; see also the recent improvement by Atkarskaya, Rips and Tent \cite{ART}. For the rest of the paper, we focus on odd exponents.

\subsection{Burnside-type presentations}
The groups constructed in this article belong to a large family of groups called \emph{Burnside-type groups}.

\begin{definition}[Burnside-type presentations and Burnside-type groups]\label{def:minimal partial Burnside presentation}
A \emph{Burnside-type presentation} is any presentation constructed as follows.

Let $S$ be an alphabet and $n$ be a large integer. Let $L_0=R_0 =E_0= \emptyset$ and let $G(0) = F(S)$ be the free group on $S$. Proceed by induction to define for all $i\in \mathbb{N}$ sets of words $L_i,R_i$, set of odd integers $E_i$, and a group $G(i)$. For $i\in \mathbb{N}$, suppose that a group $G(i-1)$ and a set of words $R_{i-1}$ have already been defined. Let $L_i$ be a set of words over $S$ satisfying the following three conditions:
\begin{enumerate}[label=(L\arabic*)] 
    \item Each $A\in L_i$ has length $i$.
    \item Each $A \in L_i$ is not conjugate in $G(i-1)$ to any power of a word shorter than $A$. In particular, $A$ is cyclically reduced and is not a proper power.
    \item If $A,B \in L_i$ with $A \neq B$, then $A$ is not conjugate in $G(i-1)$ to either $B$ or $B^{-1}$.
\end{enumerate}

Let $E_i$ be a set of odd integers greater or equal to $n$, and let $A\mapsto n_A$ be a mapping $L_i\to E_i$.
Let $R_i= \{A^{n_A} \ : \ A \in L_i\}\cup R_{i-1}$ and let
\begin{equation}\label{eq:G(i)}
    G(i)=\langle S \mid W=1 \ :\ W\in R_i \rangle,
\end{equation}
This completes the construction of $G(i),L_i,R_i,E_i$ for every $i\geq 0$. Finally, let $R=\bigcup_{i=0}^{\infty} R_i$ and consider the group $G$ given by the 
presentation
\begin{equation}
     G=G(\infty) = \langle S \mid W=1\ :\ W\in R=\bigcup_{i=0}^{\infty}R_i\rangle.\label{eq:G(infty)}
\end{equation}

A \emph{Burnside-type group} is any group admitting a Burnside-type presentation, namely, a presentation as in (\ref{eq:G(infty)}). 
The group $G(i)$ is called \emph{the rank $i$ group}, and elements of $L_i$ are called \emph{periods of rank $i$}. Note that by construction, for each period $A$, its rank is equal to its length $|A|$. The \emph{rank of $G$} is defined to be the maximal $i\in \mathbb{N}\cup \{0\}$ for which $L_i\neq \emptyset$, or $\infty$ if $L_i\neq \emptyset$ for infinitely many $i$'s.
\end{definition}

Note that every $i\in \mathbb{N}\cup \{0\}$, the identity map $S\to S$ extends to an epimorphism $\alpha_i:G(i)\twoheadrightarrow G(i+1)$, and the group $G$ is the direct limit of the sequence: 
$$G(0)\xrightarrow{\alpha_0}G(1)\xrightarrow{\alpha_1}\cdots.$$

\begin{remark}[Free Burnside groups]
Definition \ref{def:minimal partial Burnside presentation} is similar to the presentation of free Burnside groups given in \cite[Section §18]{OL-book}, but with two relaxations:
\begin{enumerate}
    \item The set of periods (namely, elements in the free group that are made torsion by a relation) chosen at the $i^{\text{th}}$ step is not necessarily maximal.
    \item The imposed relations are of the form $A^{n_A}=1$, where the exponent $n_A\ge n$ may vary depending on the period $A$.
\end{enumerate}
Indeed, if $L_i$ are taken to be maximal sets satisfying (L1)--(L3), and $n_A=n$ for all periods $A$, the resulting group $G$ is the free Burnside group $\calB(m,n)$, $m=|S|$.
\end{remark}

\begin{remark}[Partial-Burnside groups.]
A Burnside-type presentation in which $n_A=n$ for all periods $A$ is called a \emph{partial-Burnside presentation}. A group admitting a partial-Burnside
presentation is called a \emph{partial-Burnside group}, see \cite{boatman2012partial}. Boatman proved that the above definition is equivalent to a non-inductive definition of partial-Burnside groups: a group $G$ is partial-Burnside (of partial-exponent $n$) if and only if every element of finite order $g\in G$ satisfies $g^n=1$, and $G$ admits a presentation where every relation has the form $W^n=1$ for some word $W$ in the generators of $G$.
\end{remark}

More details on Burnside-type Presentations are found in \cite[Chapters 7,8]{OL-book}.

\section{Geometric preliminaries}

\subsection{Words} Let $S=\{x_1,\dots,x_m\}$ be a finite set. We say that $X$ is \emph{a word over $S$} if it is a concatenation of letter from $S^{\pm 1}=\{x_1^{\pm 1},\dots,x_m^{\pm 1}\}$. Let $X$ and $Y$ be words over a set $S$. 
We write $X\equiv Y$ if $X$ and $Y$ consist of the same letters in the same order. 
A word $Y$ is a \emph{cyclic shift} of $X$ if there exist words $U,V$ such that $Y \equiv UV$ and $X \equiv VU$. A word over $S$ is called \emph{reduced} if it contains no subword of the form $x x^{-1}$ or $x^{-1} x$ with $x \in S$, and \emph{cyclically reduced} if all of its cyclic shifts are reduced. 

\subsection{Maps} 
A \emph{spherical map} is a finite cellular partition of a $2$-sphere. Let $\Pi_1,\dots,\Pi_k$ be the cells ($2$-cells) of this partition. Removing the interior of a cell $\Pi_i$, we obtain a \emph{circular map} (geometrically, a disk), whose boundary is the boundary of $\Pi_i$. Removing the interiors of two cells (if $k\ge 2$) we obtain an \emph{annular map} whose boundary consists of two connected components.

For an edge $e$ in a map $\Delta$, we denote by $e_-$ and $e_+$ the initial and terminal vertices of $e$, respectively. By a \emph{path} in $\Delta$, we mean a concatenation $p = e_1e_2 \cdots e_k$ of edges $e_i$ with $(e_i)_+ = (e_{i+1})_{-}$ for $i = 1,\dots,k - 1$. The initial endpoint of $p$ is $p_{-} = (e_1)_{-}$ and its terminal endpoint is $p_+ = (e_k)_+$. 
For a path $p = e_1 \cdots e_k$, any path $e_ie_{i+1} \cdots e_j$ with $i \leq j$ is called a \emph{subpath} of $p$. If a path $p$ has that $p_{-} = p_+$, we can regard $p$ as a so-called \emph{cyclic path} by identifying all of its cyclic shifts, where indices are taken modulo $k$. A \emph{subpath of a cyclic path} is a subpath of any of its cyclic shifts.
We define the \emph{length} $|p|$ of $p$ to be the number of edges in $p$. 

Let $p$ be a path on a map $\Delta$, and suppose that $p_-=p_+$ and that the edges of $p$ form the boundary of a subspace $\Delta' \subseteq \Delta$  that is connected and simply connected. Then the restriction of a cell decomposition $\Delta$ to $\Delta'$ is a decomposition of $\Delta'$ , called \emph{a submap} of the map $\Delta$.
A \emph{contour} of $\Delta$ is a cyclic 
path in the boundary $\partial \Delta$. 
Circular maps have
one contour and annular maps have two; one is called the outer and the other is
the inner. The contour of a circular map (submap) $\Delta$ will be denoted by $\partial \Delta$. 
By convention we agree that the contour $\partial \Pi$ of a cell $\Pi$, and the inner contour $p$ of an annular map, travel counter-clockwise; while the boundary $\partial \Delta$ of a circular map and the outer component $q$ of the boundary of an annular map, go clockwise, see Figure \ref{fig:orientation}. A subpath of a contour will be called a \emph{section}.
\begin{figure}    \centering    \includegraphics[width=0.5\linewidth]{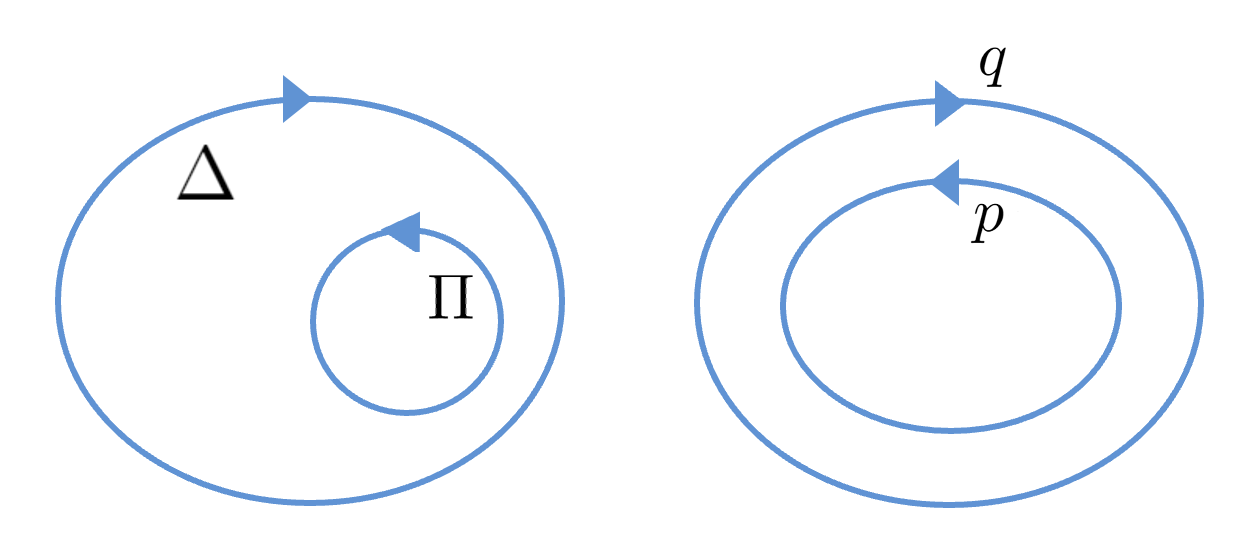}    \caption{Convention on orientation.}   \label{fig:orientation}\end{figure}

\subsection{Diagrams} 
Let $\langle S \mid R\rangle$ be a presentation of a Burnside-type group $G$, as in Definition \ref{def:minimal partial Burnside presentation}. 
Recall that $R = \bigcup_{i=0}^{\infty}R_i$, where $R_i\supseteq R_j$ if $i>j$, and equalities $R_i=R_j$ are possible. No word of $R_i\setminus R_j$ can coincide with a conjugate of a word (or an inverse of a word) from $R_j$ if $i>j$, and words in $R_i$ are non-empty and cyclically reduced, for all $i$. 
A \emph{diagram} over a Burnside-type presentation consists of a (circular or annular) map $\Delta$ 
together with a labeling function:
$$\Lab:\{\text{Edges of }\Delta\}\to S^{\pm 1}$$
that extends to the set of paths in $\Delta$ by declaring, for any path $e_1\cdots e_k$ in $\Delta$,
$$\Lab(e_1e_2 \cdots e_k) := \Lab(e_1)\Lab(e_2)\cdots \Lab(e_k),$$ 
and such that:
\begin{enumerate}
    \item $\Lab(e^{-1}) = \Lab(e)^{-1}$, and
    \item the contour label $\Lab(\partial \Pi)$ of any cell $\Pi$ 
    is a cyclic shift of a word (or an inverse of a word) from $R_i$ for some $i$. 
    The rank of $\Pi$ is the maximal $i$ such that $\Lab(\partial \Pi)$ is a cyclic shift of a word (or an inverse of a word) from $R_i \setminus R_{i-1}$.
\end{enumerate}
The \emph{rank} of a diagram $\Delta$ equipped with a labeling function as above is defined as the maximal $i$ such that $\Delta$ includes a cell of rank $i$. It is denoted by $r(\Delta)$. 


Since all groups in this paper are of Burnside-type, we will refer to diagrams over Burnside-type presentation as simply `diagrams'.  
We typically use the same name, $\Delta$, for both the diagram and the underlying map. 


\begin{remark}
    A circular (resp., annular) diagram $\Delta$ is not necessary homeomorphic to a topological disk (resp., annulus), but one can add to $\Delta$ several cells whose boundary label is trivial in the free group and obtain a topological disk (resp., annulus) $\Delta'$, see the figure below. However, we shall not use such cells (called $0$-cells) in this paper. For more details see \cite[Section §11]{OL-book}.\begin{figure}\label{fig:0 cells}
    \centering
\includegraphics[width=0.8\linewidth]{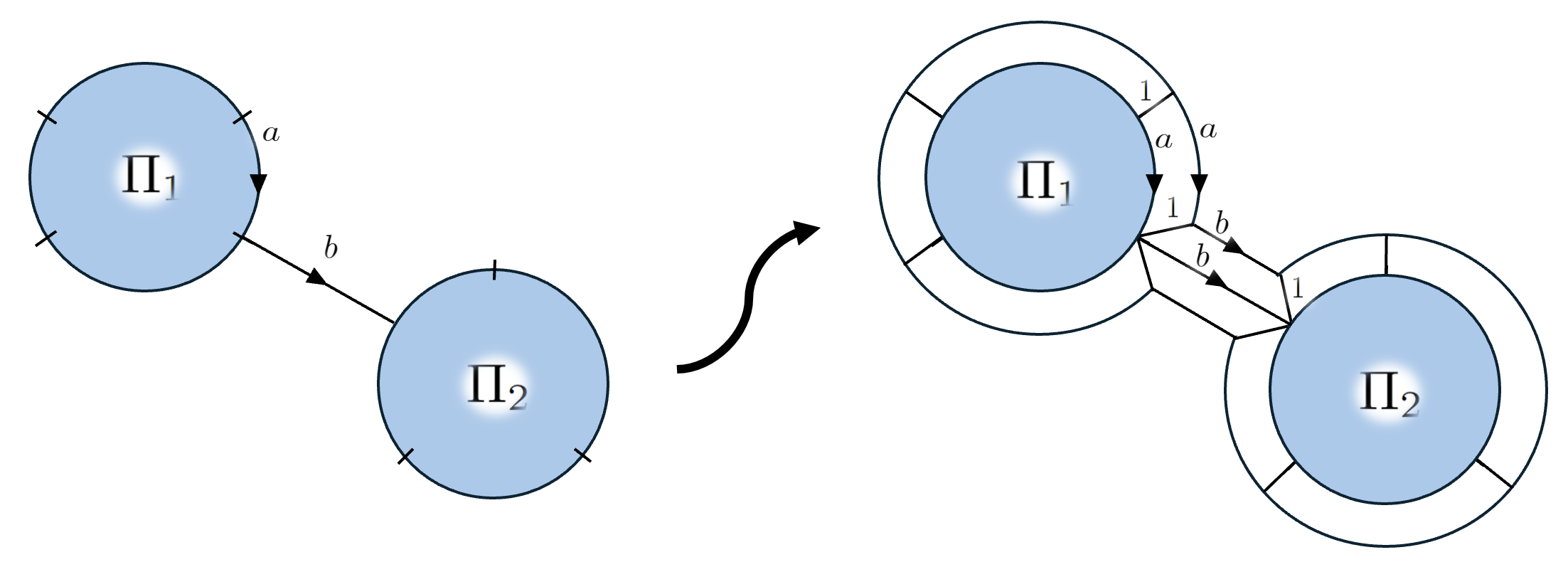}
    \caption{A diagram non-homeomorphic to a disc (left) can be transformed into a diagram homeomorphic to a disc (right) by adding cells whose boundary label has the form $a1a^{-1}1$, where $1$ denotes the identity element. By definition, the edges labeled $1$ have length $0$.}
    \label{fig:placeholder}
\end{figure}
\end{remark}

A diagram (over a Burnside-type presentation) is called \emph{reduced} if it does not contain a subdiagram $\Gamma$, having exactly two cells of some rank $i\ge 1$, such that the boundary label
of $\Gamma$ is trivial in rank $i-1$. In the current paper we only deal with reduced diagrams, which is possible due to the following. 

\begin{lemma}[Theorem 13.1 in \cite{OL-book}]
\label{lem:thm 13.1 van kampen} 
Let $W$ be a non-empty word over $S$. Then $W = 1$ in the group $G = \langle S \mid R=\bigcup_{i=0}^{\infty} R_i\rangle$ (respectively, in $G(i) = \langle S \mid R=\bigcup_{j=0}^{i} R_j\rangle$ for some $i\in \mathbb{N}$) if and only if there is a reduced circular diagram over this presentation of $G$ (resp., $G(i)$), with contour $p$,
such that $\Lab(p)\equiv W$.
\end{lemma}

\begin{lemma}[Theorem 13.2 in \cite{OL-book}]
\label{lem:thm 13.2 annular van kampen} Let $V$ and $W$ be words over $S$ that are non-trivial in the group $G = \langle S \mid R=\bigcup_{j=0}^{\infty} R_j\rangle$ (respectively, in $G(i) = \langle S \mid R=\bigcup_{j=0}^{i} R_j\rangle$ for some $i\in \mathbb{N}$). Then they are conjugate in $G$ (resp., $G(i)$) if and only if there is a reduced annular diagram over this presentation of $G$ (resp., $G(i)$), with contours $p$ and $q$, such that $\Lab(p)\equiv V$ and $\Lab(q)\equiv W^{-1}$.\end{lemma}

The next fundamental properties of Burnside-type presentations are proved in \cite[§18, §19]{OL-book} for free Burnside group, and as observed in \cite[Chapters 7,8]{OL-book} (see also \cite{boatman2012partial}), they hold with identical proofs for Burnside-type presentations. For convenience, we include here the precise references of each lemma in both \cite{OL-book} and \cite{boatman2012partial}. 

Recall that words $X$ and $Y$ are said to be equal (resp., conjugate) \emph{in rank $i$}  if $X$ and $Y$ are equal (resp., conjugate) in the group $G(i)$. 

\begin{definition}[{\cite[$\mathsection 18.1$]{OL-book} and \cite[Definition II.48]{boatman2012partial}}]\label{def:simple in rank i} 
A non-trivial word $X$ is called  \textit{simple in rank $i \geq 0$} if it is not conjugate in rank $i$ to any power of any period of rank $k\leq i$, and not conjugate in rank $i$ to any power of any word shorter than $X$. 
We say that $X$ is \textit{simple in all ranks} (or just `simple') if it is simple in rank $i$ for every $i\geq 0$.
\end{definition}


\begin{lemma}
[{\cite[Lemma 18.1]{OL-book} and \cite[Lemma II.49]{boatman2012partial}}]
\label{lem:every word is conjugate to a power of a period or simple}
     Let $i \geq 0$. Every word is conjugate in rank $i$ to a power of some period of rank $j\leq i$ or to a power of a word simple in rank $i$.
\end{lemma}

Note that Lemma \ref{lem:every word is conjugate to a power of a period or simple} holds also for $i=\infty$. Namely, every word is conjugate in $G$ to a power of some period (of some rank) or to a power of a word simple in $G$. In fact, every word that has finite order in $G$ is conjugate in $G$ to a power of a period of some rank, and every word of infinite order in $G$ is conjugate in $G$ to a power of a simple word.

\begin{lemma}[{\cite[Lemma 18.3]{OL-book} and \cite[Lemma II.51]{boatman2012partial}}]
\label{lem:finite order element is conjugate to a power of a period}
    Let $i \geq 0$. If $X$ is non-trivial in rank $i$ and $X$ has finite order in rank $i$, then it is conjugate in rank $i$ to a power of some period of rank $\leq i$.
\end{lemma}
It is straightforward that Lemma \ref{lem:finite order element is conjugate to a power of a period} holds in rank $i=\infty$ as well. Namely, if a non-trivial word has a finite order in $G$ then it must be conjugate in $G$ to a power of some period of some rank.

\section{Diagrams over Burnside-type groups}\label{sec:geometry of BT groups}

\subsection{Lower Parameter Principle} \label{subsec:LPP}
Throughout the paper, we make use of real constants 
from the interval $(0,1)$, which satisfy certain inequalities. 
As in \cite[Section §15]{OL-book}, the consistency of the system of inequalities is based here on the Lower Parameter Principle (LPP), which we explain below. 

In the current paper, the list of our parameters is: 
\begin{equation}\label{eq:parameters}
    \theta, m^{-1}, \beta,\gamma,\epsilon,\zeta,n^{-1}.
\end{equation}
We choose the values for these parameters from left to right. 
Namely, we start by choosing a parameter $\theta\in (0,1)$ which is `small enough' so that certain inequalities, involving only the parameter $\theta$, are satisfied. After that, we choose $m$ (the number of generators) with $m^{-1}\in (0,1)$ being small enough so that certain inequalities, involving $\theta$ and $m$, are satisfied. We then choose $\beta$, to satisfy certain inequalities involving $\theta$, $m$, and $\beta$, and so on, where each parameter is chosen with dependene on the previous ones. In other words, we arrange the positive parameters according to their `height' where the first one to be chosen is considered the `heighest' and the last one the `lowest'. In particular, we may assume that 
$$\theta>m^{-1}>\beta>\dots>n^{-1}.$$ Moreover, when we meet an inequality involving certain higher parameter as well as lower ones, which could be satisfied for small enough choice of the lower parameters, we can assume that this inequality holds, by the LPP. E.g., one can assume that the inequality $10\gamma<\theta-2\beta$ holds. Indeed, since $\beta$ is chosen after $\theta$, one can ensure that $\theta-2\beta>0$; then, for the same reason, $\gamma$ can be chosen so that $10\gamma <\theta -2\beta$.

We further denote 
$\bar{\beta}=1-\beta$ and $\bar{\gamma}=1-\gamma$, for simplicity of some of our computations. 

For more details on the choice of the constants we refer the reader to \cite[Section §15]{OL-book}. All of the theorems to follow hold for odd $n\in \mathbb{N}$ sufficiently large and for some choice of the rest of the parameters.

\subsection{Auxiliary properties of diagrams}
This subsection presents various properties of diagrams over Burnside-type presentations that are used throughout the paper. Properties whose proofs mirror those in \cite{OL-book}—where they are established for presentations of free Burnside groups—are given without proof.

Some of the following lemmas are formulated in \cite{OL-book} for $A$-maps, but here we reformulate them for reduced diagrams instead. This is possible since by \cite[Lemma 19.4]{OL-book} every reduced diagram over a Burnside-type presentation is an $A$-map. We will not use $A$-maps in the current paper, and the interested reader can learn more about them in \cite[Chapter 5]{OL-book}.

We will use the concept of contiguity subdiagrams. 
\begin{figure*}    \centering \includegraphics[scale=0.25]{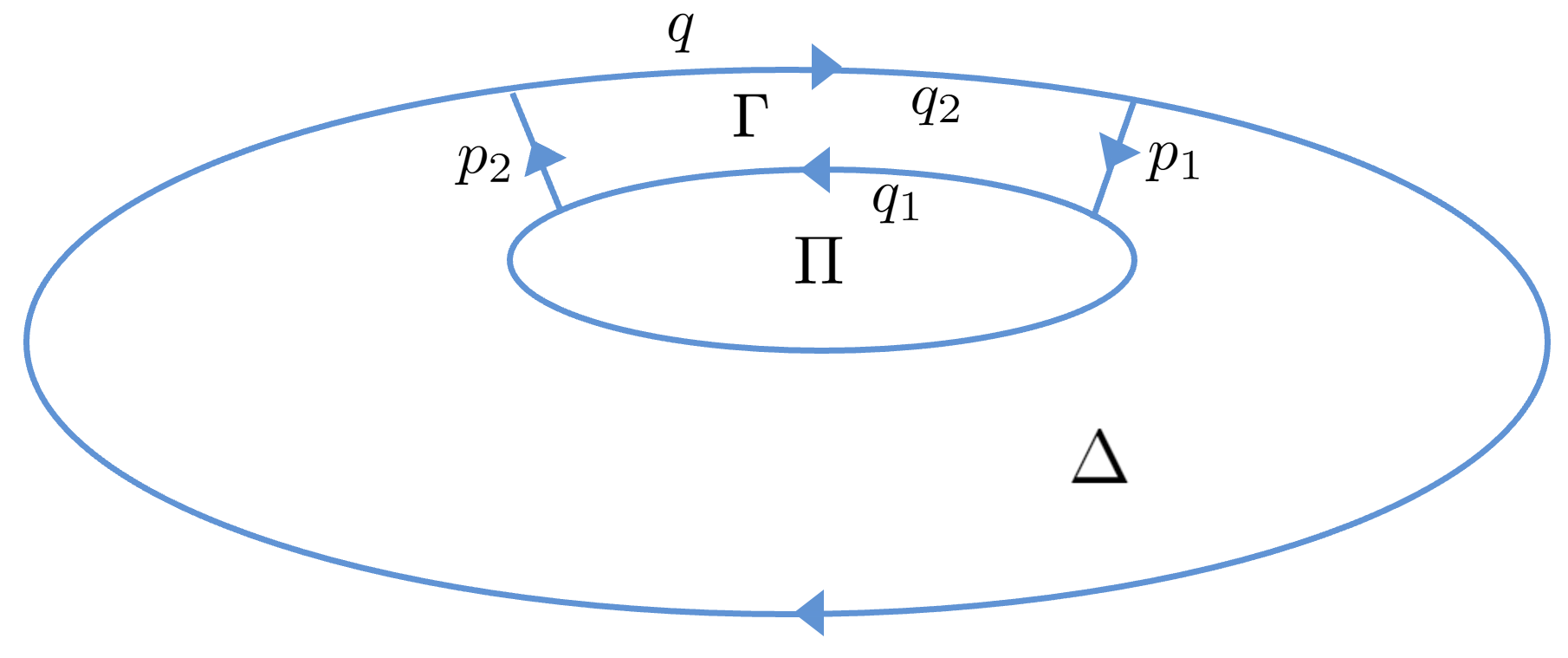}    \caption{A contiguity subdiagram $\Gamma$ between a cell $\Pi$ and a section $q$ of $\partial \Delta$.}   \label{fig:contiguity}\end{figure*}
Let $\Pi$ be a cell in a diagram $\Delta$ and let $q$ be a section of $\partial \Delta$. A \emph{contiguity diagram between $\Pi$ and $q$} is a subdiagram $\Gamma$ of $\Delta$, which does not contain $\Pi$, and whose contour is $\partial\Gamma = p_1q_1p_1q_2$, where $q_1$  is a subpath of $\partial \Pi$, $q_2$ is a subpath of $q$, and the paths $p_1$, $p_2$ are `very short' in comparison with the perimeter of $\Pi$ and the length of $q$, see Figure \ref{fig:contiguity} for illustration. 
The ratio $\frac{|q_1|}{|\partial\Pi|}$ is called the \emph{contiguity degree} of $\Pi$ to $q$ with respect to $\Gamma$, and is
denoted $(\Pi,\Gamma,q)$. 
In particular, 
any common edge of $\Pi$ and $q$ is a contiguity subdiagram with side arcs $p_1$ and $p_2$ of zero length. 
We do not need the more detailed definition
given in \cite[Section §14]{OL-book}, rather we shall use the properties of contiguity
subdiagrams given in the next lemmas.
\begin{lemma}[Lemma 15.3 in \cite{OL-book}]\label{lem:contiguity sides are short}
Let $\Delta$ be a reduced diagram, $\Gamma$ a contiguity subdiagram of a cell $\Pi$ to a section $q$ of $\partial \Delta$. As in Figure \ref{fig:contiguity}, denote the boundary of $\Gamma$ by 
$p_1q_1p_2q_2$. 
Then:
$$\max\{ |p_1|,|p_2| \}<
\zeta |\partial \Pi|.$$ 
\end{lemma}




\begin{lemma}[Lemma 15.4 in \cite{OL-book}]\label{lem:q_2 is large}
Denote by $\Psi$ the degree of $\Gamma$-contiguity of a cell $\Pi$ 
to a section $q$ of the contour in a reduced diagram $\Delta$. As in Figure \ref{fig:contiguity}, denote the boundary of $\Gamma$ by 
$p_1q_1p_2q_2$. 
Then:
$$|q_2|>(\Psi-2\beta)|\partial \Pi|.$$
\end{lemma}


\begin{lemma}[Corollary 16.1 in \cite{OL-book}]
\label{lem:existence of gamma-cell}
    Let $1\leq k\leq 4$. Let $\Delta$ be a circular reduced diagram of non-zero rank whose contour is decomposed into $k$ sections, $q_1,\dots,q_k$. Then there is a cell $\Pi$ and disjoint contiguity subdiagrams $\Gamma_{i_1}\dots,\Gamma_{i_l}$, $1\leq i_1<\dots <i_l\leq k$ of $\Delta$ to $q_{i_1},\dots,q_{i_l}$, respectively, in $\Delta$, such that: $$\Sigma_{j=1}^{l} (\Pi,\Gamma_{i_j},q_{i_j})>\bar{\gamma}.$$
\end{lemma}

\begin{lemma}[[Corollary 16.2 in \cite{OL-book}]\label{lem:existence of gamma-cell annular}
    Let $\Delta$ be an annular reduced diagram of non zero rank with contours $q_1$ and $q_2$ (regarded as cyclic sections). Then $\Delta$ has a cell $\Pi$ and disjoint contiguity subdiagrams $\Gamma_1$ and $\Gamma_2$ of $\Pi$ to $q_1$ and $q_2$ respectively (one of these may be absent), such that: $$(\Pi,\Gamma_1,q_1)+(\Pi,\Gamma_2,q_2)>\bar{\gamma}.$$
\end{lemma}
\noindent A cell $\Pi$ as in Lemma \ref{lem:existence of gamma-cell} and Lemma \ref{lem:existence of gamma-cell annular} is called a $\gamma$-cell.

\begin{lemma}[Theorem 16.2 in \cite{OL-book} and Theorem II.25 in \cite{boatman2012partial}]\label{lem:if Gamma is with minimal num of cells then r(Gamma)=0}
Let $\Delta$ be a reduced diagram of positive rank. 

Assume $\Delta$ is circular and its contour is partitioned in  sections $q_1...q_s, s\le 4$.  Then $\Delta$ has a cell $\Pi$  and a contiguity subdiagram $\Gamma$ of $\Pi$ to a section $q_j, j\le s$, such 
that $r(\Gamma) = 0$ and $(\Pi,\Gamma, q_j)\ge\epsilon$.

Moreover, assume $\Delta$ is circular or annular and let $q$ be a section of the contour $\partial \Delta$, such that there exists a cell $\Pi$ with contiguity degree at least $\epsilon$ to $q$, then $\Delta$ contains a cell $\Pi'$ and a contiguity subdiagram $\Gamma'$ of $\Pi'$ to $q$ with $r(\Gamma') = 0$ and $(\Pi',\Gamma',q) > \epsilon$.
\end{lemma}


\begin{lemma} [Lemma 17.1 in \cite{OL-book}]
\label{lem:short path conneting contours of annual A-map}
Let $\Delta$ be an annular reduced diagram with contours $p$ and $q$ such that 
$\Lab(p)$ is non-trivial in the group $G$.
Then there is a path $t$ connecting vertices $o_1$ and $o_2$ of the paths $p$ and $q$, respectively, such that $|t|<\gamma (|p|+|q|)$.
\end{lemma}

\begin{lemma}[Corollary 17.1 in \cite{OL-book}]\label{lem:boundary of a cell shorter than of the diagram}
If a circular reduced diagram $\Delta$ contains a cell $\Pi$, then $|\partial \Delta|> \overline{\beta}|\partial \Pi|$.
\end{lemma}

A path $q$ in a diagram $\Delta$ is called \emph{$\lambda$-geodesic} for $\lambda\ge 1$ if for every subpath $t$ of $q$ and for every path $s$ homotopic to $t$ in $\Delta$, we have $|t|\le \lambda |s|$. 
A path $q$ is called \emph{geodesic} if it is $\lambda$-geodesic with $\lambda=1$. A contour $p$ of an annular diagram is called \emph{$\lambda$-geodesic} if any subpath of the cyclic path $p$ is $\lambda$-geodesic. We have:

\begin{lemma}\label{lem:contig deg to lambda-geodesic}
If a boundary section $q$ of a reduced diagram $\Delta$ is $\lambda$-geodesic for $1\le \lambda\le 2$, then the degree of contiguity of any cell $\Pi$ of $\Delta$ to $q$ is less than $\frac{\lambda}{\lambda+1}+2\beta$.
\end{lemma}
 
\begin{proof}
Let $\Gamma$ be a contiguity diagram
of $\Pi$ to $q$. Denote by $\psi = (\Pi, \Gamma, q)$ the contiguity degree of $\Pi$ to $q$, by $p_1q_1p_2q_2$ the boundary path of $\Gamma$ (as in Figure \ref{fig:contiguity}), and by $q'q_1$ be the boundary path of $\Pi$. 

By Lemma \ref{lem:contiguity sides are short}, we have that $|p_1|,|p_2|<\zeta |\partial \Pi|$, and so by 
definition of $\lambda$-geodesic path, we have: \[ |q_2|\le \lambda (|p_1(q')^{-1}p_2|)\le \lambda(1-\psi+ 2\zeta )|\partial\Pi|. \]
On the other hand, by 
Lemma \ref{lem:q_2 is large}, $|q_2|>(\psi -2\beta)|\partial\Pi|$. Eliminating $|q_2|$ and
$|\partial\Pi|$ from these inequalities, one gets that: \[\psi < \frac{\lambda}{\lambda + 1} + 2\frac{\beta + \lambda \zeta}{\lambda+1}.\] Since $1\le \lambda \le 2$ and $\zeta<\beta$ (LPP), the desired upper bound for $\psi$ follows.
\end{proof}

Let $A$ be a word. A \emph{periodic word with period $A$} (also called an `$A$-periodic word') is any subword of a power of $A^m$ where $m>0$. 
The following follows from Lemmas 19.5, 13.3 and Theorem 17.1 of \cite{OL-book}.
\begin{lemma}\label{lem:smooth section}
\ \ \
    \begin{enumerate}
        \item 
        Let $p$ be a boundary section in a reduced circular diagram $\Delta$ of rank $i$. If $\Lab(p)$ is a periodic word with a simple in rank $i$ period $A$, then $p$ is a $\bar{\beta}^{-1}$-geodesic path. Moreover, $p$ is geodesic if $|p|\le |A|$.
        \item  Let $p$ and $q$ be two boundary components of a reduced annular diagram $\Delta$ and $\Lab(p)=A^k$ for a period $A$ of some rank $j$, $k \ne 0 \modd n_A$. Then there is a reduced annular diagram $\Delta'$ 
with boundary components $p'$ and $q'$ such that $\Lab(q')\equiv \Lab(q)$,
 $\Lab(p') \equiv A^{k'}$, where $k'=k \modd n_A$, $|k'|<n_A$, and the cyclic 
path $p'$ is $\bar{\beta}^{-1}$-geodesic. Furthermore, $\bar{\beta}|p'|\le |q'|$ and $|A|\le |q|$.

    \end{enumerate}
    
\end{lemma}


Consider a word $W$. By Lemma \ref{lem:every word is conjugate to a power of a period or simple}, $W$ is conjugate in $G$ to $A^k$, where $A$ is either a period, or simple in $G$. The following lemma shows that up to cyclic permutations, this conjugation can be done using a very short word.

\begin{lemma}\label{lem:up to cyclic shift conj is done by short word}
Let $A$ be a period or a word simple in $G$, let $k\in \mathbb{Z}$, and suppose a word $W$ is conjugate in $G$ to $A^k$.
Then there exists a word $T$ of length $3\gamma|W|$ and cyclic shifts $W'$ and $A'$ of $W$ and $A$ respectively, such that $W'=T^{-1}(A')^kT$ in $G$.
\end{lemma}
\begin{proof}
We may assume that $k \ne 0 \modd n_A$ if $A$ is a period.
By Lemma \ref{lem:thm 13.2 annular van kampen}, we can draw a reduced diagram $\Delta$ for the conjugation of $A^k$ and $W$. So the contours of $\Delta$ are $p$ and $q$ and have labels $\Lab(p)\equiv A^k$ and $\Lab(q)\equiv W$. 
By Lemma \ref{lem:smooth section}, we can assume that $|p|\le \bar{\beta}^{-1}|q|$, namely, $|A^k|\le \bar{\beta}^{-1}|W|<2|W|$.

By Lemma \ref{lem:short path conneting contours of annual A-map}, there is a path $t$ connecting $p$ and $q$, with 
$$|t|<\gamma (|p|+|q|)=\gamma (|W|+|A^k|)<3\gamma|W|.$$
Denote $\Lab(t)\equiv T$. So $|T|<3\gamma |W|$.

Cutting $\Delta$ along $t$, we obtain a circular diagram $\Delta'$ whose contour is 
$\partial\Delta' \equiv (q')^{-1}t^{-1}p't$, 
where $p'$ and $q'$ are cyclic shifts of $p$ and $q$ respectively.
So $\Lab(p')\equiv W'$ is a cyclic shift of $W$, and $\Lab(q')\equiv (A')^k$ is the k$^{th}$ power of a cyclic shift $A'$ of $A$. (Indeed, every cyclic shift of $A^k$ is of the form $(A')^k$ for some cyclic shift $A'$ of $A$). Now Lemma \ref{lem:thm 13.1 van kampen} ensures that 
$\Lab(\partial \Delta')\equiv W'^{-1}T^{-1}(A')^kT$ is trivial in $G$, completing the proof.\end{proof}


\section{Growth of Burnside-type groups}
\subsection{Stabilization of balls}
For a finitely generated group $G=\langle S\rangle$ denote by $B_{G,S}(r)$ the ball of radius $r$ in $G$ with respect to $S$, namely, the set of elements in $G$ that can be expressed as a product of at most $r$ elements from  $S^{\pm 1}$. 
In this subsection we show that if $G$ is a Burnside-type group then the $r$-ball in $G$ is already determined in rank $\lceil \frac{3}{n}r\rceil$, and torsion elements in the $r$-ball of $G$ can be safely counted in the $r$-ball of the rank $r$ group, $G(r)$, rather than in $G$ itself.


\begin{lemma}\label{lem:conjugation holds in low rank}
Let $U,V$ be reduced words. If $U,V$ are equal (resp., conjugate) in $G$, then they are equal (resp., conjugate) in rank $j$ for any $j\geq \frac{3}{n}\max\{|U|,|V|\}$.
\end{lemma}
\begin{proof}
Our proof uses ideas from the proof of \cite[Theorem 19.4, Item 3]{OL-book}. 
Suppose $U=V$ in $G$. By Lemma \ref{lem:thm 13.1 van kampen}, there exists a reduced circular diagram $\Delta$ with contour label $\Lab(\partial\Delta)\equiv UV^{-1}$. 
By construction, $|\partial \Delta|\leq 2\max\{|U|,|V|\}$. Let $\Pi$ be a cell in $\Delta$. By Lemma \ref{lem:boundary of a cell shorter than of the diagram}, $|\partial \Pi|<\bar{\beta}^{-1}|\partial \Delta|$. The cell $\Pi$ has contour label $\Lab(\partial\Pi)\equiv B^{n_B}$ for some period $B$, 
and we have:
$$|B|=\frac{|\partial \Pi|}{n_B}\leq \frac{2\bar{\beta}^{-1}\max\{|U|,|V|\}}{n_B}<\frac{3\max\{|U|,|V|\}}{n}.$$ 
Hence, $B$ is a period of rank less than $\frac{3}{n}\max\{|U|,|V|\}$.
Since $\Pi$ is arbitrary, the rank of $\Delta$ is less than $\frac{3}{n}\max\{|U|,|V|\}$. By Lemma \ref{lem:thm 13.1 van kampen}, $UV^{-1}=1$ in $G(j)$ for any $j\ge \frac{3}{n}\max\{|U|,|V|\}$, completing the proof of the statement.

Next, suppose $U$ and $V$ are conjugate in $G$. By Lemma \ref{lem:thm 13.2 annular van kampen}, there exists a reduced annular diagram $\Delta$ with contour labels $U$ and $V$. 
If $U=1$ in $G$, then so is $V$, and the statement follows from the first part of the proof. Otherwise, we can cut $\Delta$ along the path $t$ as in Lemma \ref{lem:short path conneting contours of annual A-map}, obtaining a circular reduced diagram $\Delta'$ with perimeter less than $(|U|+|V|)(1+\gamma)$.
We proceed as in the previous case: Let $\Pi$ be a cell in $\Delta$, which is now viewed as a cell in $\Delta'$. It has contour label $\Lab(\partial\Pi)\equiv B^{n_B}$ for some period $B$. By Lemma \ref{lem:boundary of a cell shorter than of the diagram}, $|\partial \Pi|<\bar{\beta}^{-1}|\partial \Delta'|$. Since $\bar{\beta}^{-1}(1+\gamma)<\frac{3}{2}$ (by the LPP), it follows that:
$$|B|=\frac{|\partial \Pi|}{n_B}
<\frac{\bar{\beta}^{-1}(1+\gamma)(|U|+|V|)}{n_B}<\frac{3}{n}\max\{|U|,|V|\}.$$ 
Hence, the period $B$ has rank less than $\frac{3}{n}\max\{|U|,|V|\}$, and since $\Pi$ was arbitrary, $r(\Delta)=r(\Delta')<\frac{3}{n}\max\{|U|,|V|\}$. By Lemma \ref{lem:thm 13.2 annular van kampen}, $U$ and $V$ are conjugate in $G(j)$ for any $j\ge r(\Delta)$, and the statement follows. 
\end{proof}


For a word $V$, denote by $\ord_G(V)$ (respectively, $\ord_{G(i)}(V)$) its order in $G$ (resp., $G(i)$). We now show that the ball $B_{G,S}(R)$, as well as the orders of all elements in this ball, are determined in low ranks.

\begin{corollary}\label{cor:R-th ball is determined in low rank}
Let $i,r\in \mathbb{N}$ be such that $i\geq \frac{3}{n}r$. Then the identity map on $S$ induces a bijection 
$f_i:B_{G(i),S}(r)\to B_{G,S}(r)$. If $i\geq r$, then $f_i$ preserves the order of elements. That is, for every word $V$ of length $\leq r$, we have  $\ord_{G(i)}(V)=\ord_G(V)$.
\end{corollary}
\begin{proof}
    It is clear that $f_i$ is a well-defined surjective homomorphism, since $G$ is defined as a quotient of $G(i)$. Let us show the injectivity of $f_i$. Let $U,V$ be words of length $\leq r$, representing two elements in $B_{G(i),S}(r)$ and assume that $f_i(U)=f_i(V)$. In particular, $U=V$ in $G$. By Lemma \ref{lem:conjugation holds in low rank}, and since $i\geq \frac{3}{n}r$, we have that $U=V$ in $G(i)$.

    Suppose now that $i\geq r$. Let $V$ be a non empty word of length $\leq r$.
    If $\ord_G(V)=\infty$ then clearly $\ord_{G(i)}(V)=\infty$. Suppose $\ord_G(V)<\infty$. By Lemma \ref{lem:finite order element is conjugate to a power of a period}, $V$ is conjugate in $G$ to $A^k$, where $A$ is a period, and $k\in \mathbb{Z}$, $|k|<n_A$. It follows that: $$\ord_G(V)=\ord_G(A^k)=\frac{n_A}{(n_A,k)}.$$

By Lemma \ref{lem:smooth section}(2), $|A^k|\leq \overline{\beta}^{-1}|V|$, and also $|A|\leq |V|\le r$. In particular, $A$ is a period of rank at most $r$. Since $i\geq r$, it follows that $\ord_{G(i)}(A)=n_A$.
    Since $\max\{|A^k|,|V|\}\le \max\{\bar{\beta}^{-1}|V|,|V|\}\le 2r$, Lemma \ref{lem:conjugation holds in low rank} implies that $V$ and $A^k$ are conjugate in rank $i$. It follows that: $$\ord_{G(i)}(V)=\ord_{G(i)}(A^k)=\frac{n_A}{(n_A,k)}.$$
    The proof is completed.
\end{proof}

\subsection{Growth and uniform measures on Cayley balls}

Recall the discussion on free Burnside groups, $\calB(m,n)$, from Section \ref{sec:partial burnside groups}. Free Burnside groups with sufficiently large exponents are not only infinite, but they are also `as large as possible', by means of growth rates. Let $G=\left<S\right>$ be a finitely generated group. The growth function of $G$ with respect to $S$ is:
\[
\gamma_{G,S}(r) = \# B_{G,S}(r) 
\]
This function depends on $S$, but only up to asymptotic equivalence; for more on growth functions of groups we refer the reader to \cite{delaHarpe}. 
It follows from the works of Adian \cite{Ad} and Ol'shanskii \cite{OL-book} that if $n$ is sufficiently large then the growth function of $\calB(m,n)$ is bounded from below by an exponential function of $r$  (see also \cite[Theorem 1.3]{Coulon_exp}). We will need more precise estimates of the growth for Burnside-type groups, see Lemma \ref{lem:many aperiodic}.


For the rest of this subsection with fix a generating set $S=\{x_1,\dots,x_m\}$, $m\geq 2$, and assume all the groups in the statement are generated by $S$. 
Let $t\geq 2$. A word is called \emph{$t$-aperiodic} if it has no non-empty subwords of the form $Y^t$.

\begin{lemma} \label{lem:many aperiodic}
For every $l<2m-1$ there exists $t$ such that the number of $t$-aperiodic words in $F_m$ of length $r$ is at least $l^r$.
\end{lemma}
\begin{proof}
For $t\ge 2$, let $b(r)$ be the number of $t$-aperiodic (reduced) words $W$ over the alphabet $\{x_1^{\pm 1},\dots,x_m^{\pm 1}\}$ of length $r$. Let us prove by induction on $r$ that $b(r)> l^r$ if $t$ is large enough (to be determined in the sequel). Observe that $b(1)=2m>l$. Let us prove that $b(r)>lb(r-1)$ for $r\ge 2$.

Given a $t$-aperiodic reduced word $W$ of length $r-1$, we can obtain $2m-1$ reduced  words $Wa$ of length $r$ by adding a letter $a$ from the right. The total number of such products is $(2m-1)b(r-1)$. Some of these products are not $t$-aperiodic. Let us estimate the number of such `bad' products. 
A bad product has the form $Vu^t$, where $|u|\ge 1$. 
Let $|u|=i\geq 1$. Then $|V|=r-ti$ and the number of ways to choose $V$ is $b(r-ti)$, and by the inductive hypothesis, $b(r-ti) < l^{1-ti}b(r-1)$. There are $2m$ words of
length $1$. Hence the number of bad products with $|u|=i$ is at most $(2m)^i l^{1-ti} b(r-1)$. Therefore, the number of all bad products does not exceed:
\[
\sum_{i=1}^{\infty} (2m)^i l^{1-ti} b(r-1) = 
lb(r-1)\sum_{i=1}^\infty\left(\frac{2m}{l^t}\right)^i 
=\frac{2ml}{l^t-2m}b(r-1),
\] 
where the second equality holds for $t$ large enough for which $2m<l^t$, and the series is convergent. It follows that:
\begin{eqnarray*}
    b(r) & \ge & \underbrace{(2m-1)b(r-1)}_{\text{words of the form $Wa$}}-\underbrace{\frac{2ml}{l^t-2m} b(r-1)}_{\text{`bad' products}} \\
    & =& \left(2m-1-\underbrace{\frac{2ml}{l^t-2m}}_
    {=:\kappa}\right)b(r-1) >lb(r-1)
\end{eqnarray*}
Where the last "$>$" holds by choosing $t$ large enough so that $\kappa=\frac{2ml}{l^t-2m}<2m-1-l$.

\end{proof}

\begin{lemma} \label{lem:aperiodic are nontrivial}
Let $G$ be a Burnside-type group. Let $t=\lfloor \epsilon n \rfloor$. Then all $t$-aperiodic words in the free group $F_m$ represent distinct elements in $G$.
\end{lemma}

\begin{proof}
Let $U,V$ be two distinct $t$-aperiodic words in $F_m$, and suppose that $U=V$ in $G$. By Lemma \ref{lem:thm 13.1 van kampen}, there exists a reduced circular diagram $\Delta$ with contour $pq^{-1}$ where $\Lab(p)\equiv U$ and $\Lab(q)\equiv V$. Since $U,V$ are reduced and $U\neq V$ in $F_m$, $r(\Delta)>0$. 
By Lemma \ref{lem:if Gamma is with minimal num of cells then r(Gamma)=0}, $\Delta$ contains a cell $\Pi$ and a contiguity diagram $\Gamma$ of $\Pi$ to either $p$ or $q$, without loss of generality suppose to $p$, with rank $r(\Gamma)=0$ and contiguity degree $\ge \epsilon$.

Recall that $\Lab(\partial \Pi)\equiv A^{n_A}$ for some period $A$. It follows that $U$ has a subword equal in rank $0$ to a subword of $A^{n_A}$, of length at least $\epsilon |\partial \Pi|$. Since $\epsilon |\partial \Pi|=\epsilon n_A\ge \epsilon n >t$, $U$ must contain $A^t$ as a subword. But this contradicts the assumption that $U$ is $t$-aperiodic.
\end{proof}

\begin{lemma}[Growth of Burnside-type groups.]\label{lem:growth of BT groups}
Let 
$l<2m-1$. Then if $n$ is large enough (depending on $m,l$) then every Burnside-type group $G=\langle S \rangle$ with exponents at least $n$ has:
$$\# B_{G,S}(r)\ge l^r$$
for all $r \geq 1$.

In particular, if $n$ is large enough then every Burnside-type group $G=\langle S \rangle$ with exponent(s) $\ge n$ satisfies that: $$\# B_{G,S}(r)>(2m-1)^{0.999r}$$
for $r \geq 1$.
\end{lemma}

\begin{proof}
Given $l<2m-1$, Lemma \ref{lem:many aperiodic} provides $t\in \mathbb{N}$ such that the number $b(r)$ of $t$-aperiodic words of length $r$ in the free group $F_m$ is at least $l^r$. By LPP we can choose $\iota=n^{-1}$ small enough (namely, $n$ large enough) such that $\epsilon n>t$.
By Lemma \ref{lem:aperiodic are nontrivial}, all $t$-aperiodic words in $F_m$ represent distinct elements in $G$. It follows that: $$\# B_{G,S}(R)\ge b(r)\ge l^r.$$
The `in particular' follows immediately, by choosing $l$ such that $(2m-1)^{0.999} < l < 2m-1$. 
\end{proof}

We now turn to calculate a few effective bounds on the asymptotic density of torsion elements with respect to the uniform measures on Cayley balls in Burnside-type groups.  

\begin{lemma}\label{lem:bound on conjugates of powers of single period}
Let $G$ be a Burnside-type group. Let $A$ be a period 
or a word simple in $G$. 
Then the number of conjugates of powers of $A$ in the ball of radius $r$ in $G$ is at most $(2m-1)^{0.6r}$.
\end{lemma}

\begin{proof}
Fix $A$ as in the assumptions. 
Suppose $W$ is a word of length $|W|\leq r$ that is conjugate in $G$ to $A^k$ for some power $k\in \mathbb{Z}$. 
By Lemma \ref{lem:smooth section}, we may assume that the cyclic word $A^k$ is $\bar{\beta}^{-1}$-geodesic and so that $|A^k|<\bar{\beta}^{-1}|W|<2|W|\le 2r$, and in particular $|A|<2r$ and $|k|<2r$. 

By Lemma \ref{lem:up to cyclic shift conj is done by short word}, there exist cyclic shifts $W'$ and $A'$ of $W$ and $A$ respectively, such that $W'$ and $(A')^k$ are conjugate in $G$ by a word $T$ of length at most $3\gamma|W|$. 
Since $W'$ is a cyclic permutation of $W$, we may write $W\equiv U_1U_2$ and $W'\equiv U_2U_1$, where $|U_i|\le \frac{1}{2}|W|$ for at least one of the indices $i=1,2$. So $W'$ is conjugate to $W$ by a word $U$ of length at most $\frac{1}{2}|W|$.
It follows that $W$ and $(A')^k$ are conjugate by a word $T'$ of length at most $|T|+ |U|\le (\frac{1}{2} + 3\gamma)|W|
\le 0.51 r$. Therefore, the number of options for $T'$ is bounded by the number of (not necessarily reduced) words in $S^{\pm1}\cup \{1\}$, hence it is at most $(2m+1)^{0.51r}$.

Note further, that since $|A|\leq 2r$, it has at most $2r$ cyclic permutations. 
It follows that the number of words $W$ that are conjugate to powers of $A$ in $G$ is at most 
$$\underbrace{2r}_{\substack{\text{options} \\ \text{for $A'$}}}\underbrace{4r}_{\substack{\text{options} \\ \text{for $k$}}}\underbrace{(2m+1)^{0.51r}}_{\text{options for }T'}.$$ 
Taking $m$ large enough, this number is at most $(2m-1)^{0.6r}$.
\end{proof}

\begin{lemma}\label{cla:finite rank group has asymptotic torsion density 0}
Let $G$ be a Burnside-type group. 
If $G$ has finite rank $j$ (namely, $G=G(j)$), then for every $r>10j$: 
$$\#\{g\in B_G(r) \ : \ g\text{ has finite order} \}< (2m-1)^{0.9r}.$$
In particular, $$\lim_{r\to \infty} \Pr_{G,U_S(r)}(g\text{ has finite order})=0$$
\end{lemma}
\begin{proof}
    By Lemmas \ref{lem:every word is conjugate to a power of a period or simple} and \ref{lem:finite order element is conjugate to a power of a period}, an element $g\in G$ has finite order if and only if it is conjugate to a power of some period. 
    Since all periods of $G$ have rank at most $j$, the total number of periods in $G$ is at most $2m \left(1+(2m-1)+\dots+(2m-1)^{j-1}\right)$. In particular, it does not exceed $2m(2m-1)^j$.
    By Lemma \ref{lem:bound on conjugates of powers of single period}, for every $r$, there are at most $(2m-1)^{0.6r}$ elements in the ball $B_G(r)$ that are conjugate to powers of a single period $A$. It follows that the number of elements in $B_G(r)$ that have finite order is at most: $$ 2m(2m-1)^j (2m-1)^{0.6r}\le (2m-1)^{2+j+0.6r} < (2m-1)^{0.7r+2} \leq  (2m-1)^{0.9r}, $$ since $j<0.1r$ and $r > 10$. By Lemma \ref{lem:growth of BT groups}, the ball $B_G(r)$ have size at least $(2m-1)^{0.999r}$, and so the lemma follows.
\end{proof}

The following serves as a counterpart to Lemma \ref{cla:finite rank group has asymptotic torsion density 0}.
\begin{lemma}\label{cla:maximal period sets gives asymptotic torsion density 1}
    Let $G$ be a Burnside-type group. 
    Suppose that for some $j\ge 0$ and some odd $N>n$, all sets $\{L_i\}_{i>j}$ were taken maximal with respect to conditions (L1)--(L3) (see Definition \ref{def:minimal partial Burnside presentation}), and $n_A=N$ for all $A\in L_i,i>j$. 
    Then for every $r>10j$:
    $$\# \{g\in B_{G(r)}(r) \ : \ g^N\neq 1\}<(2m-1)^{0.9r}.$$
    In particular,
    $$\lim_{r\to \infty} \Pr_{G(r),U_{S}(r)}(g^N=1)=1$$
\end{lemma}
\begin{proof}
    Let $g\in G$ be such that $g^N\neq 1$. 
    By Lemmas \ref{lem:every word is conjugate to a power of a period or simple} and \ref{lem:finite order element is conjugate to a power of a period}, it must be that $g$ is conjugate to a power of a word $B$, where $B$ is either simple of all ranks, or is a period with $n_B\neq N$. 
    Since all $L_i$, $i>j$, were chosen maximal, and $n_A=N$ for all $A\in L_i,i>j$, $B$ as above must have length at most $j$. It follows that the total number of such words $B$ cannot exceed $2m(2m-1)^j$. The proof proceeds as in the proof of Lemma \ref{cla:finite rank group has asymptotic torsion density 0}.
\end{proof}

\subsection{Random walks}
\label{subsec:RWs}

Consider a random walk $\{X_i\}_{i=1}^{\infty}$ on a group $G$ with step distribution $\mu$. In other words, $X_i$ is a product of $i$ independent random variables with distribution $\mu$, equivalently, $X_i \sim \mu^{*i}$. We assume that the random walk is \emph{non-degenerate}, namely, $\supp(\mu)$ generates $G$ as a semigroup (this means that $\bigcup_{i=1}^{\infty}\supp(\mu^{*i})=G$). Recall that a random walk is \emph{lazy} if $\mu(1)>0$, \emph{symmetric} if $\mu(g)=\mu(g^{-1})$ and \emph{finitely supported} if $\#\supp(\mu)<\infty$. Kesten's theorem \cite{Kesten} asserts that if $G$ is a finitely generated group and $\{X_i\}_{i=1}^{\infty}$ is a symmetric non-degenerate random walk on $G$ then $G$ is non-amenable if and only if $\Pr(X_r=1) < \rho^r$ for some $\rho < 1$ for all $r$. Furthermore, for each $g\in G$ we have that $\Pr(X_{2r}=1)\geq \Pr(X_r=g)\Pr(X_r=g^{-1})=\Pr(X_r=g)^2$, and hence, we may further assume that $\Pr(X_r=g)<\rho^r$ for all $g\in G$. Adian \cite{Adian amenable} proved that free Burnside groups (for sufficiently large odd exponents) are non-amenable (in fact, a stronger result holds \cite{O}). Consequently, partial-Burnside groups (with the same exponents) are also non-amenable; notice that $\rho$ above remains true whenever we replace a group $G$ by a group mapping onto it, with the corresponding random walks.

Fix $m$ and let $n\in \mathbb{N}$ be a sufficiently large odd number. Consider the simple random walk $\{X_i\}_{i=1}^{\infty}$ on $\B(m,n)$ whose step distribution is the uniform measure on the standard symmetric generating set $\{1,x_1^{\pm 1},\dots,x_m^{\pm 1}\}$.
For $r\geq 1$, we denote by $w(r)$ the set of strings in the letters $\{1,x_1^{\pm 1},\dots,x_m^{\pm 1}\}$, which we call `formal words', to stress that they are are not necessarily reduced. Evidently, $w(r)$ is in a natural bijection with $\{1,x_1^{\pm 1},\dots,x_m^{\pm 1}\}^{\times r}$ and $\#w(r)=(2m+1)^r$.
By the definition of random walks:
\begin{eqnarray*} \Pr(X_r=1) & = & \frac{\#\{W \in w(r)\ |\ W = 1\ \text{in}\ G\}}{\# w(r)}.
\end{eqnarray*}
Furthermore, by the above discussion (and using Kesten's theorem):
\[
\Pr(X_r=g) < \rho^r
\]
for some $\rho<1$ and for all $g\in G$. 
Hence, for all $g\in G$:
\begin{align}
\#\{W \in w(r)\ |\ W = g\ \text{in}\ G\} <  \left(\rho (2m+1)\right)^r, \nonumber
\end{align}
and, for some constant $c$ (independent of $r$):
\begin{align} \label{(2m+1)/k}
\#\bigcup_{j=0}^r \{W \in w(j)\ |\ W = g\ \text{in}\ G\} \leq \sum_{j=0}^r \left(\rho (2m+1)\right)^j \leq c\left(\rho (2m+1)\right)^r.
\end{align}
Notice that if we replace $n$ by a multiple of itself $n'$ then $\rho$ does not increase, due to the natural surjection $\B(m,n')\twoheadrightarrow \B(m,n)$. At the same time, we can ensure that $\gamma$ (see our list of parameters, (\ref{eq:parameters})) is arbitrarily close to $0$. Therefore, enlarging $n$ enough we can assume that moreover:
\begin{align}\label{eq:rho(2m+1)^8}
\rho^{1+6\gamma} \left(2m+1\right)^{7 \gamma} <1
\end{align}

We now turn to study some bounds regarding symmetric simple lazy random walks on partial-Burnside groups.

\begin{lemma}\label{lem:bound on conjugates of powers of single *short* period}
Let $G$ be a partial-Burnside group.
Then there exists $c>0$, such that for every $A$ that is either a period or a word simple in $G$, the following holds. For every $r\in \mathbb{N}$, the number of formal words over $S^{\pm 1}\cup \{1\}$ of length $r$ that are conjugate in $G$ to powers of $A$ is at most:
$$16cr^5
    (\rho(2m+1))^{(1+6\gamma)r}.$$
\end{lemma}

\begin{proof}
Let $W$ be a formal word of length $r$ over $S^{\pm 1}\cup \{1\}$ that is conjugate in $G$ to $A^k$ for some $k\in \mathbb{Z}$. By Lemma \ref{lem:smooth section} and by the choice of $\beta<\frac{1}{2}$, we may assume that $|A^k|<2|W|$, and in particular, $|A|<2r$ and $|k|<2r$.  
By Lemma \ref{lem:up to cyclic shift conj is done by short word}, there exist a word $T$ of length $|T|<3\gamma|W|$ and cyclic shifts $A'$ and $W'$ of $A$ and $W$ (respectively) such that $TW'T^{-1}=(A')^k$ in $G$. 
Denote $V \equiv T(W')T^{-1}$ (here we consider it as a formal word, so it might be non-reduced). We have: 
$$|V| = |W'|+2|T|\le (1+6\gamma)r.$$ 
Recall that $V=(A')^k$ in $G$. By Inequality (\ref{(2m+1)/k}), the number of possible formal words of length at most $(1+6\gamma)r$ that are equal in $G$ to a specific element is bounded by $c(\rho(2m+1))^{(1+6\gamma)r}$ for some $c$, independent of $r$. Since there are at most $|A|<2r$ options for the cyclic shift $A'$ of $A$, and at most $4r$ options for $k$, $|k|<2r$, the number of possible formal words $V$ such that $V=(A')^k$ in $G$ for some choice of $A',k$ is at most
$$8r^2c(\rho(2m+1))^{(1+6\gamma)r}.$$ 
Note that $W'$ is a subword of $V$. Therefore, after fixing $V$, there are at most $|V|^2$ options for $W'$, 
and therefore at most $|V|^2|W'|\le 2r^3$ options for $W$. 
We conclude that there are at most: 
\begin{eqnarray*}
\underbrace{8r^2c(\rho(2m+1))^{(1+6\gamma)r}}_{\text{options for $V$}}\underbrace{2r^3}_{\substack{\text{options for $W$} \\ \text{for a fixed $V$}}}\\
\end{eqnarray*}
formal words $W$ that are conjugate in $G$ to powers of $A$.
\end{proof}

\begin{lemma}
\label{lem:bound on finite order words when periods are short}
Let $G$ be a partial-Burnside group, let $c>0$ be as in the previous lemma, and let $r,R\in \mathbb{N}$ be such that $R>\frac{2}{\gamma}r$. 
Suppose that no period (resp., simple word in $G$) has length between $r+1$ and $R$.
Then the number of formal words of length $R$ over $S^{\pm 1}\cup \{1\}$ that have finite (resp., infinite) order in $G$ is at most: $$16c R^5 (\rho(2m+1))^{(1+6\gamma)R}(2m+1)^{\gamma R}.$$
In particular, if $\rho^{1+6\gamma} \left(2m+1\right)^{7 \gamma} <1$ (\ref{eq:rho(2m+1)^8}), 
then this quantity is smaller than $R^{-1}(2m+~1)^R$ for $R\gg 1$. 
\end{lemma}
\begin{proof}
Suppose a formal word $W$ of length $R$ has a finite (respectively, infinite) order in $G$. By Lemma \ref{lem:finite order element is conjugate to a power of a period}, $W$ must be conjugate in $G$ to a power of $A$ where $A$ is a period (resp., simple word) in $G$.
If $A$ is simple, then by definition (see \ref{def:simple in rank i}), it cannot be conjugate to a shorter word, and so $|A|\leq |W|=R$. If $A$ is a period, then by Lemma \ref{lem:smooth section}(2), we may assume that $|A|\le |W|=R$; indeed, $W=\Lab(q)$ in the formulation of \ref{lem:smooth section}(2). 
Therefore, by assumption, $|A|\le r$. The total number of such words $A$ cannot exceed: $$ 2m(2m-1)^r<(2m+1)^{r+1}\leq (2m+1)^{\gamma R},$$
the left-most expression bounding the total number of reduced words of length $\leq r$ in $m$ generators (and their inverses). 
By Lemma \ref{lem:bound on conjugates of powers of single *short* period}, for each period (resp., simple word) $A$ there are at most: $$16cR^5 (\rho(2m+1))^{(1+6\gamma)R}$$ formal words of length $R$ conjugate to powers of $A$. The first statement follows.

For the `in particular' part, observe that:
\begin{eqnarray*}
 16c R^5 \left(\rho(2m+1)\right)^{(1+6\gamma)R}(2m+1)^{\gamma R} & \leq &  R^6 \left( \rho^{1+6\gamma}(2m+1)^{1+7\gamma}\right)^R \end{eqnarray*}
which is exponentially negligible (in $R$) with respect to $(2m+1)^R$, hence smaller than $R^{-1}(2m+1)^R$, for $R\gg 1$. Notice that we may indeed assume that (\ref{eq:rho(2m+1)^8}) holds by the discussion preceding it.
\end{proof}

\section{Free subgroups of partial-Burnside groups}
We now construct groups satisfying the properties indicated in Theorem \ref{thm:Theorem A}. 
By the end of this section, we will discuss some related examples and results for a broader context.

Let $n$ be a large enough odd integer, let $m\in \bbN$ be large enough (to be specified in the sequel), and let $S=\{a,b,x_1,\dots,x_m\}$.
We will construct a group $G$ by providing a Burnside-type presentation for it. Toward this end, we recursively define, for every integer $i\geq 0$, finite sets $L_i,R_i$ of words over $S$, as well as a finitely presented group $G(i)$, generated by $S$. As in Definition \ref{def:minimal partial Burnside presentation}, the words of $L_i$ are called \emph{periods of rank $i$}, and the group $G(i)$ is referred to as \emph{rank $i$}.

Let $L_0=R_0=\emptyset$ and let $G(0)$ be the free group $F(S)$. 
Let $i\in\mathbb{N}$ and suppose that $L_j$, $R_j$, and $G(j)$ have already been defined for all $j<i$.  Denote by $C_i$ the set of reduced words over $S$ that represent elements that are not conjugate in rank $i-1$ to any element from the subgroup $\langle a,b\rangle\leq G(i-1)$.
Let $L_i$ be a maximal subset of $C_i$ that satisfies conditions (L1)--(L3) (see Definition \ref{def:minimal partial Burnside presentation}). 
Let $R_i= \{A^n \ : \ A \in L_i\}\cup R_{i-1}$ and let
$$ G(i)=\langle S \mid W=1 \ :\ W\in R_i \rangle. $$ 
This completes the construction of $G(i),L_i,R_i$ for every $i\geq 0$. Finally, let $R=\bigcup_{i=0}^{\infty} R_i$ and consider the presentation:
\begin{equation}\label{eq:graded presentation of the a,b construction}
    G=G(\infty)=\langle S \mid W=1\ :\ W\in R=\bigcup_{i=0}^{\infty}R_i\rangle.
\end{equation}

We turn to study the properties of $G$. For an element $g\in G$, we denote by $\|g\|$ its norm, namely, the minimum length $|U|$ of a word $U$ over $S$ representing $g$ in $G$. 

\begin{lemma} \label{lem:torsion}
Let $g\in G$. Then either $g$ is conjugate in $G$ to an element from the subgroup $\langle a,b\rangle\leq G$ or $g^n=1$.
\end{lemma}

\begin{proof}
This follows directly from the construction, but we prove it here by induction on the norm for the reader's convenience. 
The claim trivially holds for the only element of zero norm, namely, $1$. Let $g\in G$ have norm $\|g\|=i>0$, and let $W$ be a word representing $g\in G$ with $|W|=i$. Suppose that any word of length less than $i$ satisfies the claimed assertion.

Assume first that $W$ is conjugate in $G$ to a power of a shorter word, namely to $U^k$ with $|U|<|W|$, $k\in \mathbb{Z}$. By the inductive assumption, either $U$ is conjugate to a word $V\in \langle a,b\rangle$, in which case $W$ is conjugate to $V^k\in \langle a,b\rangle$, or else $U^n=1$, in which case $W^n$ is conjugate to $U^{nk}$, and so $W^n=1$. In both cases the assertion follows. We can therefore assume that $W$ is not conjugate in $G$ to a power of a shorter word. 

Next, assume that $W$ is conjugate to $L^k$ where $L$ is a period in $G$, and $k\in \mathbb{Z}$. Then $L^n=1$ in $G$, and as before, $W^n$ is conjugate to $L^{nk}=1$, and so $W^n=1$. So the assertion is satisfied in this case as well. We can therefore assume that $W$ is not conjugate to a power of a period.

To conclude, we can assume that $W$ satisfies condition (L2) for $i$. Furthermore, notice that $W$ satisfies (L1).

To complete the proof, suppose that $W$ is not conjugate in $G$ (and therefore, in $G(i)$) to any element from the subgroup $\langle a,b\rangle$. Then by definition, $W$ belongs to $C_i$. Since $L_i$ is chosen maximal, and since $W\in C_i$, and it satisfies (L1) and (L2), we get that either $W\in L_i$, or $W$ is conjugate in rank $i-1$ to $B^{\pm 1}$ for some $B\in L_i$. In the first case, $W^n=1$ is a relation in the presentation of $G$, and in the second, $B^n=1$ is. In both cases, $W^n=1$ in $G$.
\end{proof}

We now turn to study the geometry of the group $G$.
\begin{lemma}\label{lem:word in F(a,b) is not conj to a shorter word}Let $W$ be a cyclically reduced word in $\{a,b\}^{\pm 1}$. Then $W$ is not conjugate in $G$ to any word in $S^{\pm 1}$ shorter than $|W|$.
\end{lemma}
\begin{proof}
    Let $W$ be a cyclically reduced word in $\{a,b\}^{\pm 1}$. Let $V$ be a word conjugate to $W$ in $G$ that has minimal possible length. By Lemma \ref{lem:thm 13.2 annular van kampen}, there exists an annular reduced diagram $\Delta$ with contours $p$ and $q$ and labels $\Lab(p)\equiv W$ and $\Lab(q)\equiv V$. Note that by the choice of $V$, we have that $q$ is cyclically minimal.
    
    Suppose toward contradiction that $|V|<|W|$. It follows that $r(\Delta)>0$, and so by Lemma \ref{lem:existence of gamma-cell annular}, $\Delta$ contains a cell $\Pi$ and contiguity diagrams $\Gamma_1,\Gamma_2$ of $\Pi$ to $p,q$ respectively, such that 
    $$(\Pi,\Gamma_1,p)+(\Pi,\Gamma_2,q)>\bar \gamma.$$
    By the choice of $V$, we have that $q$ (regarded as a cyclic section) is geodesic. 
    Therefore, Lemma \ref{lem:contig deg to lambda-geodesic} used with $\lambda=1$ implies that $(\Pi,\Gamma_2,q)<\frac{1}{2}+2\beta$. 
    It follows that $$(\Pi,\Gamma_1,p)>\bar\gamma -(\frac{1}{2} + 2\beta) =\frac{1}{2}-(\gamma + 2\beta),$$ which is larger than $\epsilon$ by the LPP.
    Therefore, using (the second part of) Lemma \ref{lem:if Gamma is with minimal num of cells then r(Gamma)=0}, $\Delta$ contains a cell $\Pi'$ and a contiguity subdiagram $\Gamma'$ of $\Pi'$ to $p$ with $r(\Gamma')=0$ and $(\Pi',\Gamma',p)>\epsilon$. Recall that by construction, $\Lab(\partial\Pi')$ is equal to $A^n$ for some period $A$. Since $r(\Gamma')=0$, it follows that $W$ has a subword $T$ equal in rank $0$ to a subword of $A^n$, of length at least $\epsilon |\partial \Pi'|=\epsilon n |A|$. Since $\epsilon n >2$, $T$ must contain $A$ as a subword. In particular, $A$ consists of the letters $\{a,b\}^{\pm 1}$ only. But this contradicts the fact that only words not conjugate to the subgroup $\langle a,b\rangle$ can be chosen as periods. 

    We conclude that $|V|\ge |W|$, completing the proof.
\end{proof}

We get the following direct corollary:
\begin{lemma}\label{lem:F(a,b) is free}     The subgroup $\langle a,b \rangle \leq G$ is free. \end{lemma}

Recall that $\langle\langle a,b \rangle\rangle$ denotes the normal closure of the subgroup $\langle a,b\rangle$. Since all of the defining relations of $G$ are of the form $W^n=1$, and since, by Lemma \ref{lem:torsion}, every $g\in G$ that is not in $\langle\langle a,b\rangle\rangle$ satisfies $g^n=1$, we obtain:

\begin{corollary} \label{cor:quotient burnside}
There is an isomorphism $G/\left<\left< a,b \right> \right>\cong \calB(m,n)$ carrying $x_1,\dots,x_m$ to the standard generating set of $\B(m,n)$.
\end{corollary}

\begin{remark}
Simple groups in which every element has order dividing a large odd exponent $n$ except for those conjugate to a free subgroup $F(a,b)$, were constructed by Obraztsov \cite{Obraztsov}. This is done by choosing $G_1 = F(a,b)$ and $G_2$ a finite cyclic group, in the formulation of his theorem that appears in \cite[Theorem 35.1]{OL-book}. Another possible approach is to construct the desired examples as direct limits of hyperbolic groups \cite{Ol91,Coulon_Burnside_Quotients}.\\
However, controlling the torsion probabilities is a quantitative challenge that requires a different tool, one that we develop in the current section.
\end{remark}
Recall that $B_{G,S}(R)$ is the radius-$R$ ball in the Cayley graph of $G$ with respect to $S=\{a,b,x_1,\dots,x_m\}$; we henceforth omit the subscript $S$. We further denote $H=\langle a,b\rangle$, and let $H^G := \bigcup_{x\in G} xHx^{-1}$ denote the union of conjugates of $H$. 

We now turn to bound the density of torsion elements in $G$. By Lemma \ref{lem:torsion}, this amounts to computing the density of $H^G$.

\begin{lemma} \label{lem:uniform}
We have an exponential decay:
$$ \frac{\#\left(B_G(R)\cap H^G\right)}{\#B_G(R)}\xrightarrow{R\rightarrow \infty} 0. $$
\end{lemma}

\begin{proof}
Let $g\in B_G(R)\cap H^G$ be a non-trivial element.
        
        We first claim that $g$ can be written as $CW'C^{-1}$, where $W'\in F(a,b)$ is a word of length at most $R$ and $C$ is a word over $S$ of length at most $(\frac{1}{2}+2\gamma) R$.
        Let $V$ be a shortest word in $S^{\pm 1}$ representing $g$; so $|V|\le R$. Let $W$ be a  
        cyclically reduced word in $\{a,b\}^{\pm 1}$ that is conjugate to $U$ in $G$.
        By Lemma \ref{lem:word in F(a,b) is not conj to a shorter word}, $|W|\le |V|\le R$. 
        By Lemma \ref{lem:thm 13.2 annular van kampen}, there exists a reduced annular  diagram $\Delta$, with contours $p,q$, $\Lab(p)\equiv V$, $\Lab(q) \equiv W$. 
        By Lemma \ref{lem:short path conneting contours of annual A-map}, there is a path $t$ connecting $p$ and $q$, with 
        $$|t|<\gamma (|p|+|q|)=\gamma (|V|+|W|)\leq 2\gamma R.$$
        Cutting $\Delta$ along $t$, we obtain a circular diagram $\Delta'$ whose contour is  $\partial\Delta' \equiv t^{-1}p't(q')^{-1}$, where $p'$ and $q'$, are cyclic shifts of $p$ and $q$ respectively. Denote $\Lab(t)\equiv T$, $\Lab(p')\equiv V'$, and $\Lab(q')\equiv W'$. By Lemma \ref{lem:thm 13.1 van kampen}, this amounts to 
        $$T^{-1}V'TW'^{-1}=1$$ in $G$.
        Since $V'$ is a cyclic shift of $V$, we have that $V'= X^{-1}VX$ for some $X$ of length $|X|\leq \frac{1}{2}|V|$. Denoting $C=XT$, we have that $C^{-1}VCW'^{-1}=1$, and so $$V=CW'C^{-1}$$ in $G$. Finally, since $W'$ is a cyclic shift of $W$, it has the same length as $W$, $|W'|=|W|\leq R$. As for $C$, we have that $|C|\leq |X|+|T|<\frac{1}{2}R+2\gamma R$, completing the proof of the claim.
        
        Next, we will bound the number of possible elements in $B_G(R)\cap H^G$ by bounding the number of possibilities for $W'$ and $C$ as above. Since $g$ is equal to $CW'C^{-1}$, it is determined by $C$ and $W'$, and so we have that $$\# \left(B_G(R)\cap H^G\right)\leq \#\text{Options for }W'\ \cdot \#\text{Options for }C.$$
Observe that
\begin{itemize}
    \item Since $W'$ is a reduced word of length at most $R$ over the alphabet $\{a,b\}^{\pm 1}$, there are at most $5^R$ options for $W'$;
    \item Since $C$ is a word of length at most $\left(\frac{1}{2}+2\gamma\right)R$ over $S^{\pm 1}$, there are at most $(2m+5)^{(\frac{1}{2}+2\gamma)R}$ options for $C$.
\end{itemize}
It follows that $$\# (B_R(G)\cap H^G) \leq 5^R \cdot (2m+5)^{(\frac{1}{2}+2\gamma)R} \leq \left(5 (2m+5)^{\frac{2}{3}}\right)^R<(2m-1)^{0.99R},$$ as we may assume that $\gamma<\frac{1}{12}$ and $m\geq 80$. 
By Lemma \ref{lem:growth of BT groups} we then have,
$$ \frac{\#\left(B_G(R)\cap H^G\right)}{\#B_G(R)} 
\leq  \frac{(2m-1)^{0.99R}}{(2m-1)^{0.999R}}  \xrightarrow{R\rightarrow \infty} 0.$$
\end{proof}

We are ready to prove Theorem \ref{thm:Theorem A}.


\begin{proof}[{Proof of Theorem \ref{thm:Theorem A}}]
As before, we denote by $H$ the subgroup $H=\langle a,b\rangle$. By Lemma \ref{lem:torsion}
and Lemma \ref{lem:uniform}:
$$ \Pr_{U_S(R)} (g^n \neq 1) \leq \frac{\#\left(B_G(R)\cap H^G\right)}{\#B_G(R)} \xrightarrow{R\rightarrow \infty} 0, $$
proving the first assertion.

To prove the second one, consider any lazy random walk $\{X_i\}_{i=1}^{\infty}$ with a non-degenerate step distribution $\mu$. 
Recall that by Corollary \ref{cor:quotient burnside}, $G/\langle\langle H \rangle\rangle\cong \calB(m,n)$; denote the corresponding surjection by $\pi\colon G\twoheadrightarrow \calB(m,n)$ and notice that the random walk on $\calB(m,n)$ whose step-distribution is obtained by pushing forward $\mu$ along $\pi$ (denoted $\pi_{\#}\mu$), say, $\{\overline{X}_i\}_{i=1}^{\infty}$ is again a lazy non-degenerate random walk.
Since by Lemma \ref{lem:torsion}, every element in $G$ which does not have order dividing $n$ is conjugate to an element from $H$, we obtain:
\begin{eqnarray*}
\Pr(X_i^n \neq 1) & = & \Pr\left(X_i\in H^G \setminus \{1\} \right) \\ & \leq & \Pr\left(X_i \in \langle\langle H \rangle\rangle\right) 
 = \Pr(\overline{X}_i=1).
\end{eqnarray*}
Now $\Pr(\overline{X}_i=1)$ is the return probability of the non-degenerate random walk $\{\overline{X}_i\}_{i=1}^{\infty}$ on $\calB(m,n)$ with step-distribution $\pi_{\#} \mu$, and therefore decays to zero 
(even exponentially, since large enough free Burnside groups are non-amenable \cite{Adian amenable}; see e.g. \cite{DS} for the general setting of non-symmetric random walks)
and hence: $$\Pr(X_i^n = 1) \xrightarrow{i\rightarrow \infty} 1,$$ 
as required.

Finally, by Lemma \ref{lem:F(a,b) is free}, $H\leq G$ is a free non-abelian subgroup, completing the proof of the theorem.
\end{proof}

Interestingly, even virtually torsion-free groups can be torsion with high probability. Let $G$ be a finitely generated group and let $H\leq G$ be a torsion-free subgroup of finite index.  Then for every symmetric, lazy, non-degenerate random walk $\{X_i\}_{i=1}^{\infty}$ we have by \cite[Theorem 1.11]{Tointon} that: \[\Pr(X_i\ \text{is torsion})\leq 1-\Pr(X_i\in H)\xrightarrow{i\rightarrow \infty} 1-\frac{1}{[G:H]},\]
so the torsion probability is bounded away from $1$.
By \cite[Proposition 2.4]{BV2002}, for every finitely generated group of subexponential growth, the asymptotic volume of its balls occupied by every finite-index subgroup $H\leq G$ is $\frac{1}{[G:H]}$. Since virtually nilpotent groups are virtually torsion-free, it follows that for such groups, the torsion limit probability, even with respect to the uniform measures on balls in their Cayley graphs, is bounded away from $1$.
\begin{remark}
The authors of \cite{ABGK} constructed virtually abelian semidirect products of the from $\mathbb{Z}^{m-1}\rtimes \mathbb{Z}/m$ which satisfy $\Pr(X_i^m=1)=\Pr(X_i\ \text{is torsion})\geq \varphi(m)/m$ where $\{X_i\}_{i=1}^{\infty}$ is a random walk as above, so the torsion probability in virtually torsion-free groups can be arbitrarily close to $1$ (see \cite[Proposition 5.13]{ABGK}). Moreover, by \cite{BV2002}, a similar result holds for the uniform probability measures on balls in the Cayley graph of $G$. 
\end{remark}

We conclude this section by proving that the torsion probability in every finitely generated virtually \textit{solvable} group is --- unlike the groups constructed in Theorem \ref{thm:Theorem A} --- bounded away from $1$. (Notice that \cite[Proposition 2.4]{BV2002} does not apply in this case, as finitely generated solvable groups have exponential growth unless they are virtually nilpotent.)

\begin{lemma} \label{lem:MN}
Let $G=\left<S\right>$ be a finitely generated group and let $M,N\subseteq G$ be subsets such that $|M|<\infty,|N|=\infty$ and $G=MN=\{mn \ : \ m\in M,~n\in N\}$. Then there exists a constant $c=c(G,S,M)>0$ such that for $r\gg 1$ we have:
\[\frac{|N\cap B_{G,S}(r)|}{B_{G,S}(r)}\geq c.\]
\end{lemma}

\begin{proof}
Let $l=\max\{\|m\|_S\ :\ m\in M\}$ and let $r\geq l$. Every $g\in B_{G,S}(r-l)$ can be decomposed  as $g=hf$ where $h\in M,f\in N$ and therefore $\|f\|_S \leq \|g\|_S+\|h\|_S\leq (r-l)+l=r$. In addition, $|B_{G,S}(r)|\leq |B_{G,S}(r-l)|\cdot |B_{G,S}(l)|$ so:
\[
|B_{G,S}(r)|\leq |M| \cdot |N\cap B_{G,S}(r)| \cdot |B_{G,S}(l)|,
\]
proving the claimed assertion for $c=\frac{1}{|M|\cdot |B_{G,S}(l)|}$.
\end{proof}

\begin{lemma} \label{lem:solvable}
Let $G$ be a finitely generated, infinite, virtually solvable group. Then there exist $M,N\subseteq G$ such that $|M|<\infty$ and $G=MN$, and moreover, $N$ contains no elements of finite order.
\end{lemma}

\begin{proof}
We work by induction on the minimum derived length of a finite-index normal subgroup of $G$. If $G$ is virtually abelian then it contains a finite-index torsion-free subgroup $A\leq G$, and we can take $N=A\setminus \{1\}$ and $M$ a finite set containing at least two distinct representatives from each coset in $G/A$.

In the general case, let $H\leq G$ be a finite-index normal subgroup with minimum possible derived length. In particular, $H/[H,H]$ is infinite (otherwise, replace $H$ by $[H,H]$, which is still normal in $G$) so $G_1:=G/[H,H]$ is an infinite, virtually abelian group and we can find $M_1,N_1$ satisfying the claim for $G_1$. Let $N$ be the pre-image of $N_1$ under the natural projection $G\twoheadrightarrow G_1$; thus every element in $N_1$ has an infinite order.
We let $M$ be a lift 
of $M_1$, so $|M|=|M_1|<\infty$. 
For every element $g\in G$, let $\bar{g}=g[H,H]$ denote the image of $g$ under the natural projection $G\twoheadrightarrow G_1$.
Notice that for every $g\in G$ we have in $G_1$ that $\bar{g}=\bar{x}\cdot \bar{y}$ for suitable $\bar{x}\in M_1,\bar{y}\in N_1$ (and some $x,y\in G$ reducing to them respectively). By the construction of $M,N$ there exists some $x'\in M$ such that $\bar{x'} = \bar{x}$, and $y\in N$, so $\bar{g} = \bar{x'} \cdot \bar{y}$ and therefore $g\in x'y[H,H]$. Notice that by the definition of $N$, since $\bar{y}\in N_1$ then $y[H,H]\subseteq N$ and it follows that $g\in MN$, as claimed.
\end{proof}

\begin{corollary}
Let $G$ be a finitely generated virtually solvable infinite group. Then: \[\limsup_{R\rightarrow \infty} \Pr_{U_{G,S}(R)}(g\ \text{is torsion}) < 1\]
and for every finitely supported, non-degenerate, symmetric, lazy random walk $\{X_i\}_{i=1}^{\infty}$:
\[\limsup_{i\rightarrow \infty} \Pr(X_i\  \text{is torsion}) < 1.\]
\end{corollary}
\begin{proof}
For uniform measures on balls in the Cayley graph, this is immediate from Lemmas \ref{lem:MN} and \ref{lem:solvable}. For random walks, this follows since every finitely generated virtually solvable infinite group surjects onto a finitely generated virtually abelian infinite group, which thus contains a finite-index torsion-free subgroup, and the claim follows by \cite[Theorem 1.11]{Tointon}.
\end{proof}

\section{Contrasting generating sets: random walks}
\label{sec:different generating sets}

In this section, we prove the following:
\begin{theorem}
[{Contrasting generating sets: random walks}] \label{thm:Theorem B'} There exists a group $G$, $n\in \mathbb{N}$, and two finitely supported, symmetric, lazy, non-degenerate random walks $\{X_i\}_{i=1}^{\infty}$ and $\{Y_i\}_{i=1}^{\infty}$ on $G$ such that: 
$$\limsup_{i\to \infty}\Pr((X_i)^n=1)=1\ \ \ \text{and}\ \ \ \lim_{i\to \infty}\Pr((Y_i)^n=1)=0.$$
\end{theorem}
\noindent  
Theorem \ref{thm:Theorem B'} shows that the limit probabilities of Burnside laws can be highly sensitive to the choice of the random walks. (As in previous cases, in fact, the above group satisfies the formally stronger property $\lim_{r\to \infty}\Pr(Y_r\text{ is torsion})=0$.)

In Section \ref{sec:contrasting uniform}, we will prove an (even stronger) variation of Theorem \ref{thm:Theorem B'} for uniform measures, namely, Theorem \ref{thm:Theorem B}. 
However, the cases of random walks and uniform measures require different toolkits. 
The current section introduces an interesting application of oscillating probabilities and weighted random walks on Cartesian squares of groups, which are utilized to solve \cite[Question 13.4]{ABGK} for random walks. Section \ref{sec:contrasting uniform} is more combinatorially involved and solves \cite[Question 13.4]{ABGK} for uniform measures in a strong sense of limit probability (namely, $\lim$ over 
$\limsup$). We proceed to describing the construction of the group $G$, which will be used to prove Theorem \ref{thm:Theorem B'}.

Let $m\ge 2$, let $X=\{x_1,x_2,\dots,x_m\}$, and $n$ be large enough as discussed in Section \ref{subsec:RWs}. Recall that $\gamma$ is close to $0$ and let $K>\frac{2}{\gamma}$, in particular, $K\gg 1$.
Consider a sequence $\{r_i\}_{i=0}^{\infty}$ defined inductively as follows. Let $r_0=0$ and $r_1>0$ and for $i\geq 1$ put:
\[ r_{i+1} = \begin{cases} 
      2Kr_{i} & \text{if}\ i+1\ \text{is even} \\
      8K^3r_{i} & \text{if}\ i+1\ \text{is odd}
   \end{cases}
\]
Call a positive integer $j$ 
a \emph{free radius} if $r_{i}\leq j < r_{i+1}$ for some even $i$, and a \emph{torsion radius} otherwise. 

We now construct by induction, for every integer $i\geq 0$, sets $L_i$, and $R_i$ of words over $X$ and a finitely presented group $G(i)$ generated by $X$. 

For the base case let $L_0=R_0=\emptyset$, and let $G(0)$ be the free group $F(X)$. Let $i\in \mathbb{N}$ and suppose that $L_j$, $R_j$, and $G(j)$ have already been defined for all $j<i$.
If $i$ is a free radius, let $L_i=\emptyset$.
Otherwise, namely, if $i$ is a torsion radius let $L_i$ be a maximal subset of words over $X$ that satisfies conditions (L1)--(L3) (see Definition \ref{def:minimal partial Burnside presentation}).
In either case, let $R_{i}=\{A^n \mid A\in L_{i}\}\cup R_{i-1}$ (so if $L_{i}=\emptyset$, then $R_i=R_{i-1}$) and let:
\begin{equation*}\label{eq:presentation for the oscillation construction}
    G(i)=\langle X \mid W=1\ :\ W\in R_{i}\rangle.
\end{equation*}
This completes the inductive construction of $L_i$, $R_i$, and $G(i)$. 
Finally, let $R=\bigcup_{i=0}^{\infty} R_i$ and consider the group $G$ given by the presentation
\begin{equation}\label{eq:graded presentation of the oscillation construction}
    G=G(\infty)=\langle X \mid W=1\ :\ W\in R=\bigcup_{i=0}^{\infty}R_i\rangle
\end{equation}

Note that $G$ is a partial-Burnside group, and so all torsion elements in $G$ have orders dividing $n$. 

Consider the random walk $\{X_i\}_{i=1}^{\infty}$ on $G$ whose step distribution is the uniform distribution on the set $X^{\pm 1} \cup \{1\}$.
\begin{lemma}\label{lem:torsion density RW S oscillation}
   Suppose $i\gg 1$. If $i\in \mathbb{N}$ is even, then for every $r\in [Kr_{i},\frac{1}{2}r_{i+1}]$:
    \[    \Pr((X_{r})^n=1) \leq \frac{1}{r}. \]
    If $i\in \mathbb{N}$ is odd, then for $r=Kr_i$:
    \[ \Pr((X_{r})^n=1) \geq 1-\frac{1}{r}. \]
\end{lemma}

\begin{proof}
Recall that for any $r>0$, we denote by $w(r)$ the set of formal words of length $r$ over $X^{\pm 1}\cup \{1\}$.

Assume first that $i$ is even, and let $r\in [Kr_i,\frac{1}{2}r_{i+1}]$. Then $r$ must be a free radius, and by construction, there are no periods of length between $r_i$ and $r$. By the choice of $K$, we have that $r\geq \frac{2}{\gamma}r_i$, and by taking $n$ large enough we can further assume that Inequality (\ref{eq:rho(2m+1)^8}) holds. It therefore follows from Lemma \ref{lem:bound on finite order words when periods are short} that there are at most $r^{-1}(2m+1)^r$
formal words of length $r$ over $X^{\pm 1}\cup \{1\}$ that have finite order in $G$ (we may assume that $r\gg 1$).

Since the step distribution of $X_r$ is uniform on $w(r)$, we have:
\begin{eqnarray*}
\Pr(X_r\text{ has finite order in $G$}) & = & \frac{\# \{ W \in w(r) \ : \ W\text{ has finite order in $G$}\}}{\# w(r)}\\
& \le & \frac{r^{-1}(2m+1)^r}{(2m+1)^r} =  \frac{1}{r}
\end{eqnarray*}
This completes the proof in case that $i$ is even.

Suppose next that $i\in \mathbb{N}$ is odd and let $r=Kr_i$; observe that $r$ is a torsion radius. By construction, there are no words of length between $r_i$ and $r$ that are simple in $G$. Again, $r>\frac{2}{\gamma}r_i$ by the choice of $K$, and so Lemma \ref{lem:bound on finite order words when periods are short} implies that there are at most $r^{-1}(2m+1)^r$ formal words of length $r$ over $X^{\pm 1}\cup \{1\}$ that have infinite order in $G$ (again, we may assume that $r\gg 1$). Again, we have:

\begin{eqnarray*}
\Pr(X_r\text{ has infinite order in $G$}) & = & \frac{\# \{ W \in w(r) \ : \ W\text{ has infinite order in $G$}\}}{\# w(r)}\\
& \le & \frac{r^{-1}(2m+1)^r}{(2m+1)^r} =  \frac{1}{r}
\end{eqnarray*}
and so:
$$\Pr(X_r\text{ has finite order in $G$})\ge 1-\frac{1}{r}.$$
This completes the proof.
\end{proof}

We conclude the section with the proof of Theorem \ref{thm:Theorem B'}. We thank Gady Kozma for inspiring the strategy used in the following proof.
\begin{proof}[Proof of Theorem \ref{thm:Theorem B'}]
Consider the direct product $\Gamma = G \times G$. Let $\nu$ be the uniform probability measure on $X^{\pm 1}\cup \{1\}$ and consider a random walk $\{Y_i\}_{i=1}^{\infty}$ on $\Gamma$ given by running two independent random walks as $\{X_i\}_{i=1}^{\infty}$ above (whose step distribution is $\nu$), namely, $Y_i$ has step distribution $\nu \times \nu$. Thus, $\Pr((Y_r)^n=1)=\Pr((X_r)^n=1)^2$ for all $r\ge 0$ and by Lemma \ref{lem:torsion density RW S oscillation}, taking radii along the subsequence $\{Kr_i\}_{i\in \mathbb{N}}$:
%
\[ 
\limsup_{r\rightarrow \infty}  \Pr((Y_r)^n=1)\geq \lim_{i\to \infty}\Pr((X_{Kr_i})^n=1)^2\ge \lim_{i\rightarrow \infty} \left(1-\frac{1}{Kr_i}\right)^2 = 1. \] 
This completes the proof of the first assertion of the theorem.





Consider a random walk $\{Z_i\}_{i=1}^\infty$ on $\Gamma$ defined by the step distribution:
\[\nu^{*4K^2}\times \nu. \]
(As always, $\nu^{*i}$ stands for the $i$-fold self-convolution of $\nu$.) 
Namely, let 
$\{X_i\}_{i=1}^{\infty}$ and $\{X'_i\}_{i=1}^{\infty}$ be two independent random walks on $G$ with step distribution $\nu$, and let:
\[ Z_i=(X_{4K^2 i}, X'_i)\sim \nu^{*4K^2i}\times \nu^{*i}. \]
We have that:
\[
\Pr((Z_r)^n=1) = \Pr((X_{4K^2r})^n=1) \cdot \Pr((X'_r)^n=1).
\]
We claim that for every $r\gg 1$, at least one of the two factors of the right hand side is very small. Indeed, if $r\in [Kr_i,\frac{1}{2} r_{i+1}]$ for some even $i$, then by Lemma \ref{lem:torsion density RW S oscillation}: $$\Pr(X'_r\ \text{is torsion})\leq \frac{1}{r}.$$ Otherwise, $r\in [\frac{1}{2}r_{i-1},Kr_{i}]$ for some even $i$, and so $4K^2r\in [2K^2r_{i-1},4K^3r_{i}]=[Kr_i,\frac{1}{2}r_{i+1}]$, and again by Lemma \ref{lem:torsion density RW S oscillation}:
\[\Pr((X_{4K^2r})^n=1)\leq \frac{1}{4K^2r}.\]
In both cases, we get the bound $ \Pr(Z_r\ \text{is torsion}) \leq \frac{1}{r}$ and therefore:
\[ \lim_{r\rightarrow \infty} \Pr((Z_r)^n=1) = 0,\]
as required, completing the proof of the theorem.
\end{proof}

\section{Efficiently encodable words}

In this section we introduce {\it code length} of a word $W$ over an alphabet $X^{\pm 1}$. This information-theoretic concept, inspired by the idea of Kolmogorov complexity, takes into account periodic patterns in the structure of $W$. This will be used in studying the groups constructed in Sections \ref{sec:contrasting uniform} and \ref{sec:co-prime}, and also might be of independent interest. The definition is not inductive; it is purely combinatorial and not group theoretic, but it takes into account the nature of Burnside-type relations. Similar encoding can be applied to the study of other direct limits of hyperbolic groups provided every defining relation $R = 1$ can be uniquely reconstructed from the length $|R|$ and a short subword of the word $R$.
Much more sophisticated inductive approaches to the study
of the nesting structure of periodic words can be found in \cite{Ad} and \cite{ART}.

Recall that a word $W$ is called $A$-periodic if it is a subword of a power of cyclically reduced word $A$.

\begin{definition}[Code Length]\label{def:code length}
Let $W$ be a non-empty $A$-periodic word over $X^{\pm1}$, such that $|A|<|W|$ and $A$ is cyclically reduced and is not a proper power in the free group $F(X)$.
We may assume that $W$ and $A$ have the same common prefix, equal to $A$. 
The \emph{periodic code of $W$ with respect to $A$}, denoted $\pcode_A(W)$, is the word $0A|W|$, where the length $|W|$ is expressed in binary digits using the alphabet $\{0,1\}$. Thus $\pcode_A(W)$ is a word over the alphabet $X^{\pm1 }\cup \{0,1\}$.
    
Now let $W$ be an arbitrary word over $X^{\pm1}$.
Given a factorization $f$ of $W$ as: $$W\equiv V_0W_1V_1\cdots W_kV_k,$$
where $W_i$ is a periodic word with period $A_i$ for every $i=1,\dots,k$, we define the \emph{code of $W$ with respect to $f$} to be $$\code_f(W)\equiv V_0\pcode_{A_1}(W_1) V_1\cdots\pcode_{A_k}(W_k)V_k.$$
For every word $W$ over $X^{\pm 1}$, we choose a factorization $f$ for which the length of $\code_f(W)$ (as a word over $X^{\pm 1}\cup \{0,1\}$) is minimal among all possible factorizations $f$ of $W$, and define the \emph{code of $W$} to be $\code(W) := \code_f(W)$. We further define the \emph{code length} of $W$, denoted $\cln(W)$, to be the length of $\code(W)$.
\end{definition}

Every word $W$ can be recovered from $\code(W)$ starting with the end, and so the map $W\mapsto \code(W)$ is injective. 
The following claim is straightforward.
\begin{claim}\label{cla:cln(W)} \ \ \
    \begin{enumerate}
        \item For any $W$, $\cln(W)\le |W|$; for any $W_1$ and $W_2$, $\cln(W_1W_2)\le \cln(W_1)+\cln(W_2)$.
        \item Denote $m=|X|$. Then the number of words $W$ with $\cln(W)=k$ is at most $|X^{\pm 1}\cup \{0,1\}|^k=(2m+2)^k$.  
        \item If $W$ is a periodic word with 
        period $A$, then:
$$\cln(W)=1+|A|+\text{Binary length of }|W|.$$ 
If furthermore $|W|\le n|A|$, then $\cln(W)\le \zeta n|A|$.
    \end{enumerate}
\end{claim}

All the diagrams in this section are assumed to be reduced diagrams over Burnside-type presentations.  
For a path $p$ in a diagram $\Delta$, we define $\cln(p)$ as the code length of the label of $p$.

\begin{definition}\label{def:(p,q)}
Let $\Delta$ be a circular (or an annular) diagram containing two paths $p$ and $q$. We denote by $(p,q)$ the number of edges of $\Delta$, which belong to both $p$ and $q$.
\end{definition}

\begin{lemma}
\label{lem:p+ and p- connected with short cln}
Let $\Delta$ be a reduced circular diagram
with boundary $pq^{-1}$. Then the vertices $p_-$ and $p_+$
can be connected with a path $z$ such that
$$\cln(z)\le (p,q) + \beta(|p|+|q|).$$
\end{lemma} 

\begin{proof}
If $\Delta$ has rank $0$, then the edges of the geodesic path $z$ connecting the vertex $p_-=q_-$ to $p_+=q_+$ have to
be among the edges of $p$ and of $q$, and so $\cln(z)\le |z|\le (p,q)$ and the statement holds. 

Therefore $\Delta$ has a positive rank. 
We now induct on the number of 
cells in $\Delta$, assuming it is positive. 
By Lemma 
\ref{lem:existence of gamma-cell}, 
$\Delta$ has a $\gamma$-cell $\Pi$. We consider two cases.

\smallskip

\textbf{Case 1.} There is a maximal contiguity diagram of $\Pi$ to one of the sections $p$ and $q^{-1}$ only, say, to $p$. Denote it by $\Gamma_p$. We have $(\Pi, \Gamma_p,p)>\bar\gamma = 1-\gamma$. Denote the contour of $\Gamma_p$ by $s_1t_1s_2t_2$ (See Figure \ref{fig:lem3 case1}). 
So $t_2$ is a subpath of $p$ and we have a decomposition $p=p_1t_2p_2$. Denote by $v$ the section of $\partial \Pi$ complementing $t_1$, namely, $\partial \Pi = t_1v^{-1}$; it is possible that $|v|=0$. Consider the path $p'=p_1(s_1 vs_2)^{-1}p_2$ and the diagram $\Delta'$ inscribed by the path $p'q^{-1}$. One may apply the inductive hypothesis to $\Delta'$ and obtain a path $z$ connecting $p_-$ to $p_+$ in $\Delta'$ where
\begin{equation}
    \cln(z)\le (p',q) +\beta(|p'|+|q|).\label{eq:cl(z)}
\end{equation}
We conclude that $v$ has no common edges with $q$ since there exists no contiguity diagram of $\Pi$ to $q^{-1}$. Therefore, and by Lemma \ref{lem:contiguity sides are short} applied to $s_1,s_2$, $$(p',q)\le (p,q)+|s_1|+|s_2| <(p,q)+2\zeta|\partial\Pi|.$$
By Lemma \ref{lem:q_2 is large} applied to $t_2$, and again by Lemma \ref{lem:contiguity sides are short},
$$|p'|= |p|-|t_2|+|s_1vs_2|< |p|-(\bar \gamma -2\beta -2\zeta-\gamma)|\partial\Pi|.$$
Substituting the last two bounds in Inequality (\ref{eq:cl(z)}), it follows that:
\begin{eqnarray*}
\cln(z) & \le & (p,q)+ 2\zeta|\partial\Pi| + \beta(|p|+|q|) -\beta
(1-2\gamma-2\beta-2\zeta)|\partial\Pi| \\ & < & (p,q)+\beta(|p|+|q|)    
\end{eqnarray*}
since the coefficients of $|\partial\Pi|$ sum up to a negative number (by the LPP).
\begin{figure*}    \centering    \includegraphics[scale=0.35]{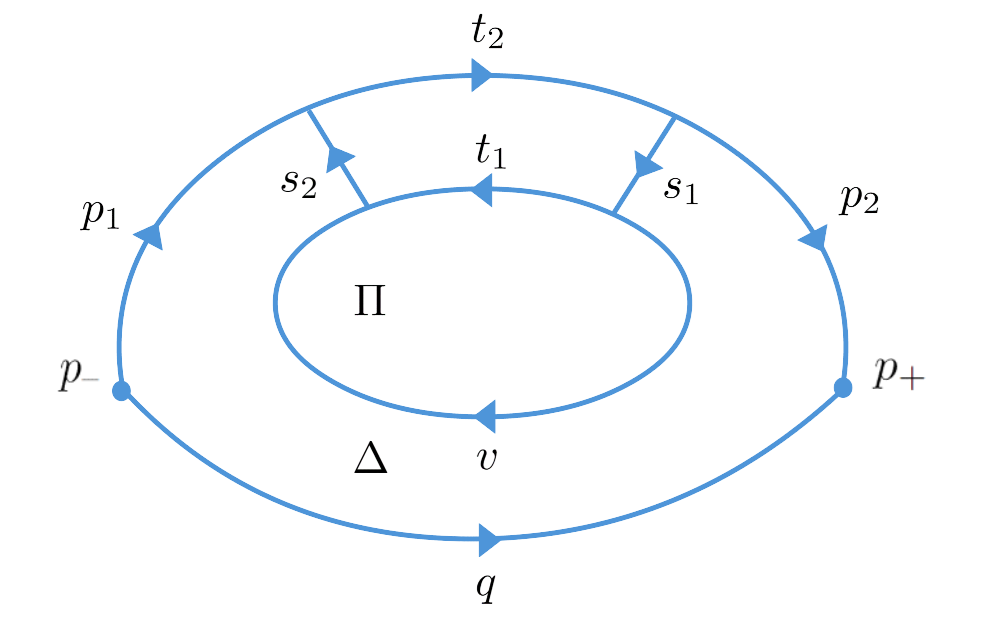}    \caption{Case 1}    \label{fig:lem3 case1}\end{figure*}
\begin{figure*}    \centering    \includegraphics[scale=0.34]{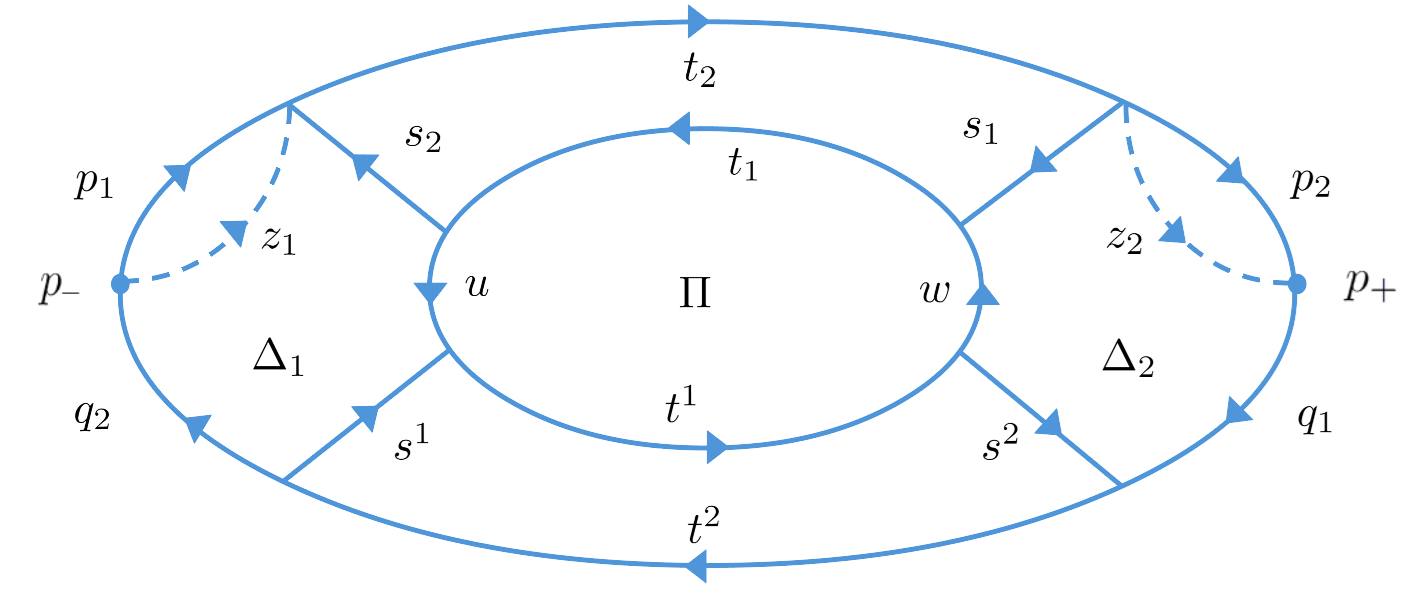}    \caption{Case 2}    \label{fig:lem3 case2} \end{figure*}

\smallskip

\textbf{Case 2.} There exist maximal contiguity subdiagrams $\Gamma_p$ and $\Gamma_q$ of $\Pi$ to $p$ and to $q$, resp. Using the notation
from the previous case for $\Gamma_p$, we denote the contour of $\Gamma_q$ by $s^1t^1s^2t^2$, and so $q^{-1}=q_1t^2q_2$. Without loss of generality, we can assume that $(\Pi,\Gamma_p,p)>\frac 12 \bar\gamma$. Denote by $t_1ut^1w$
the boundary of $\Pi$.

We now apply the inductive hypothesis to the subdiagram $\Delta_1$ inscribed by the path $p_1(q' _1)^{-1}$, where $q'_1=q_2^{-1} s^1u^{-1}s_2$, and get a path $z_1$ from $(p_1)_-$ to $(p_1)_+$ such that:
\begin{eqnarray*}
    \cln(z_1) & \le & (p_1,q'_1)+\beta(|p_1|+|q'_1|) \\ 
    & \le & (p_1,q_2)+|s_2^{-1}us_1^{-1}|
+\beta\big(|p_1|+|q_2|+(\gamma+2\zeta)|\partial\Pi|\big) \\
  & \le & (p_1,q_2)+2\gamma|\partial\Pi|+\beta(|p_1|+|q_2|)+ 2\beta\gamma |\partial\Pi|
\end{eqnarray*}
using again Lemma \ref{lem:contiguity sides are short}, along with the fact that $\Pi$ is a $\gamma$-cell. 
Similarly, there is a path $z_2$ connecting $(p_2)_-$ and $(p_2)_+=p_+$ such that
$$\cln(z_2) \le (p_2,q_1)+2\gamma|\partial\Pi|+\beta(|p_2|+|q_1|)+ 2\beta\gamma |\partial\Pi|.$$ 
The path $t_1$ has a periodic label with period of length $|\partial\Pi|/n$. 
By Claim \ref{cla:cln(W)}(3), we have that $\cln(t_1^{-1})\le \zeta |\partial\Pi|$
and, again by Lemma \ref{lem:contiguity sides are short}, we conclude that: $$\cln((s_1t_1s_2)^{-1}) < 3\zeta|\partial\Pi|.$$
Let now $z=z_1(s_1t_1s_2)^{-1}z_2$. Taking into account that:
\[ |p_1|+|p_2|=|p|-|t_2|<|p|-\left(\frac{\bar\gamma}{2}-2\beta\right)
|\partial\Pi| \] by Lemma \ref{lem:q_2 is large}, we obtain:
\begin{eqnarray*}
    \cln(z) & \le & \cln(z_1)+\cln((s_1t_1s_2)^{-1})+\cln(z_2) \\
    & \le & (p,q)+
\beta(|p|+|q|) - (\frac{1}{2}- \frac{1}{2}\gamma-2\beta) |\partial\Pi| +4\gamma|\partial\Pi|+4\beta\gamma |\partial\Pi|+3\zeta |\partial\Pi| \\ 
& <& (p,q)+ \beta(|p|+|q|),
\end{eqnarray*}
as required.    
\end{proof} 

Note that the path $z$ provided by Lemma \ref{lem:p+ and p- connected with short cln} depends on both $p$ and $q$; in particular, $z=p$ if $p=q$.
Note further, that the lemma  only guarantees that $\cln(z)$ is short, namely, that $z$ is of `low complexity', while $z$ need not be short, or (quasi)-geodesic in any way. Indeed, consider the following example.
\begin{example}
Let $\Delta$ consist of a single cell $\Pi$ whose boundary label is $A^n$, for some word $A$. We can decompose $\partial \Pi$ as $pq^{-1}$, where $\Lab(p)\equiv A$ and $\Lab(q)\equiv A^{1-n}$. One can take $z=p$ or $z=q$, as both choices fulfill the inequality of Lemma \ref{lem:p+ and p- connected with short cln}, though $|p|<<|q|$ and $q$ is highly non-geodesic.
\end{example}

\begin{remark}
Recall that $\beta$ is very small. Thus, if the number of common edged $(p,q)$ is significantly smaller than $|p|,|q|$, we get an encoding of $\Lab(z)$ that is efficient compared to $|p|,|q|$.

For `most' words, the code length is not far from their actual length. Hence the Lemma reads as `$|z|\approx (p,q)$', namely, the diagram $\Delta$ is very thin and $\Lab(p)\approx \Lab(q)$ in the free group. 
Note that $\Lab(z)$ represents the same group element as $\Lab(p),\Lab(q)$.
This will help us to estimate the number of elements in $G$ with some properties, which can be represented by the words of length $\le r$, for large $r$.
\begin{figure}
    \centering
    \includegraphics[width=0.4\linewidth]{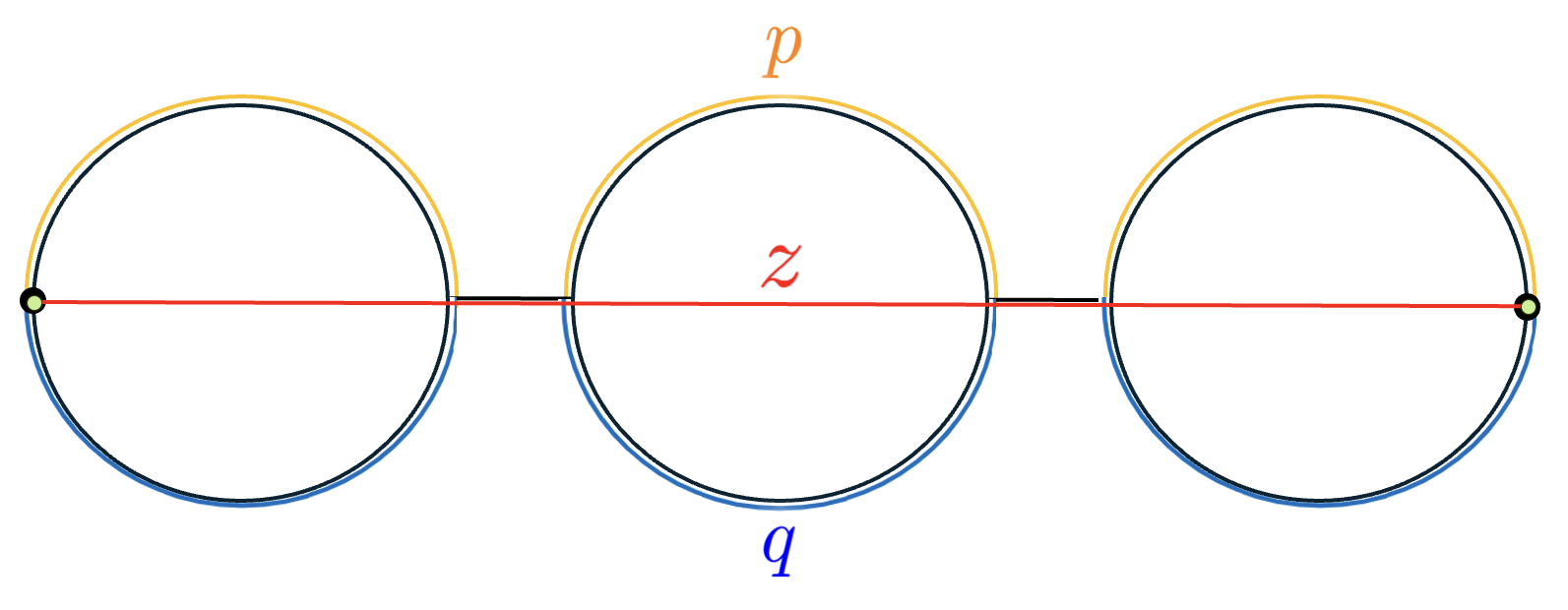}
    \caption{Subpaths of $z$ that do not go along common edges of $p$ and $q$ have small code length.}
    \label{fig:placeholder}
\end{figure}
\end{remark}


\begin{lemma}
\label{lem:connecting O to the triangle floor with short code length}
Let $T$ be a minimal (`triangular') diagram
over $G$ with boundary $A-B-C-A$, where the side $BC$ is $\lambda$-geodesic for some $1\leq \lambda \leq \bar{\beta}^{-1}$. Then there exists a path $t$ connecting the vertex $A$ and the side $BC$
such that: $$\cln(t) \le \frac 12 (|CA|+|AB|-\lambda^{-1}|BC|)+ 6\gamma(|CA|+|AB|).$$
\end{lemma}

\begin{remark}
Since $\lambda^{-1}$ is close to $1$ and $\gamma$ to $0$, the inequality of Lemma \ref{lem:connecting O to the triangle floor with short code length} resembles the inequality of geodesic triangles in hyperbolic spaces. However, one cannot replace $\cln(t)$ by $|t|$ here; consider, for example, a diagram consisting of a single cell, with three vertices on the boundary defining an (almost) equilateral triangle.
\end{remark}

\begin{proof}[Proof of Lemma \ref{lem:connecting O to the triangle floor with short code length}]
If $T$ 
has rank $0$, then it is a tripod, and the path $t$ from $A$ to the branch point has length $\frac{1}{2} (|CA|+|AB|-|BC|)$, and so it satisfies the assertion of the lemma.
Hence one may assume that $T$ has positive rank, 
and we induct on its number of cells. By Lemma \ref{lem:existence of gamma-cell}, $T$ contains a $\gamma$-cell $\Pi$.

Denote by $a$, $b$, and $c$ the lengths of the sides $BC$, $CA$, and $AB$, respectively. There exist contiguity diagrams of $\Pi$ to at least two sides of the triangle, since by Lemma \ref{lem:contig deg to lambda-geodesic}, the contiguity degree to one side is less than $\frac{\lambda}{\lambda+1} +2\beta<\frac{1}{2}+3\beta<\bar \gamma$. (Here we used the LPP as well as the assumption on $\lambda$.) 
Along the proof, we will apply Lemma \ref{lem:contiguity sides are short} multiple times in order to bound any side path of any of the contiguity diagrams of $\Pi$, by $\zeta|\partial \Pi|$. 
We consider three cases.

\smallskip

\textbf{Case 1.} Assume that there are only two maximal contiguity subdiagrams of $\Pi$, say $\Gamma_a$ and $\Gamma_b$ (or $\Gamma_c$) to the sides $BC$ and $CA$ (or $AB$), respectively. Thus, each one of the contiguity degrees is greater than $\bar\gamma - (\frac12 +3\beta)= \frac12 -3\beta-\gamma$.

Denote by $p_1q_1p_2q_2$ 
the boundary path of $\Gamma_a$, by
$s_1t_1s_2t_2$
the boundary path of $\Gamma_b$, and by $q_1wt_1u$ the boundary path of $\Pi$. Denote by $C'$ the vertex $(p_2)_+$ (i.e. the end of $p_2$) and by $D$ the vertex $(t_2)_+$. The length of the boundary subpaths $BC'$, $C'C$, and $CD$ are denoted by $a'$, $f$, and $e$ resp.  

Let $AC'$ be a geodesic connecting the vertices $A$ and $C'$. We have that $|s_1w^{-1}p_2| <(\gamma +2\zeta)|\partial\Pi|<2\gamma|\partial\Pi|$, and it follows that
\begin{equation}\label{eq:AC'}
    |AC'|\le |AD|+|s_1w^{-1}p_2|< (b-e) +2\gamma|\partial\Pi|.
\end{equation}
Moreover, since $BC$ is $\lambda$-geodesic, we have $f\leq \lambda (|(s_1w^{-1}p_2)^{-1}|+e)\leq \lambda(2\gamma|\partial\Pi| + e)$, and therefore $$a'=a-f>a-\lambda e- 2\lambda\gamma|\partial\Pi|.$$
\begin{figure*}    \centering    \includegraphics[scale=0.5]{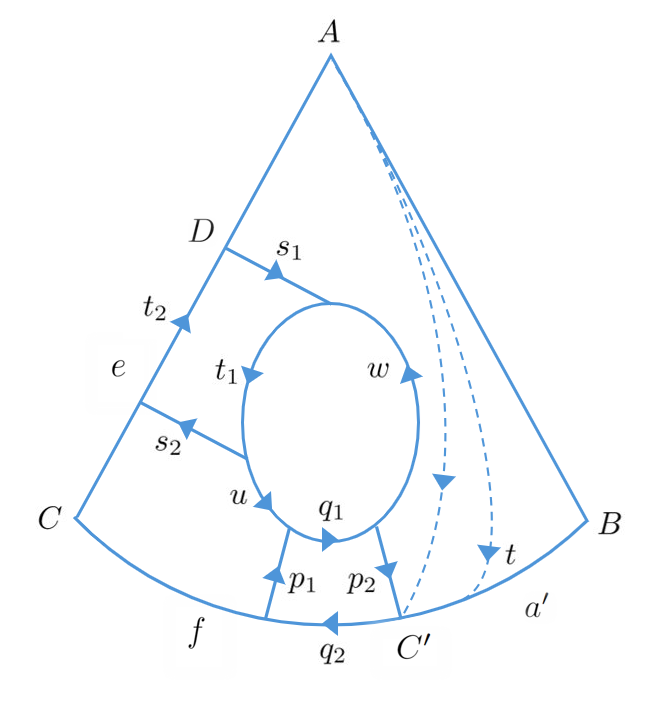}    \caption{Case 1}    \label{fig:enter-label} \end{figure*}

We define $\Delta'$ as the subdiagram with the boundary $(D-A-B-C')p_2^{-1}ws_1^{-1}$. 
Since the triangle $ABC'$ satisfies the conditions of the lemma, and contains less cells 
than $T$ does, the vertex $A$ can be connected to the side $BC'$ with a path $t$, where
\begin{eqnarray*}
    \cln(t) & \le & \frac 12 (c+|AC'|-\lambda^{-1}a') + 6\gamma(c+|AC'|)\\ 
    & \le & \frac 12\big(c+|AC'|-\lambda^{-1}(a-\lambda e-2\lambda\gamma|\partial\Pi|)\big)+6\gamma(c+|AC'|) \\
    & \le & \frac 12 (c+b-e+2\gamma|\partial\Pi|-\lambda^{-1} a  + e +2\gamma|\partial\Pi|) + 6\gamma(c+b-e + 2\gamma|\partial\Pi|) \\
    & = & \frac 12 (b+c-\lambda^{-1}a)+6\gamma(b+c)+2\gamma|\partial \Pi|-6\gamma e +12\gamma^2|\partial \Pi|
\end{eqnarray*}
using the bounds on $a'$ and $|AC'|$ given above. Now this value is less than $$\frac 12 (b+c-\lambda^{-1}a) + 6\gamma (b+c),$$
because by Lemma \ref{lem:q_2 is large}, and since the contiguity degree of $\Gamma_b$ is at least $\frac{1}{2} - 3\beta -\gamma$,
\[ e\ge |t_2|> |t_1| - 2\beta|\partial\Pi| \ge (\frac 12 -3\beta - \gamma -2\beta)|\partial\Pi|> \frac 25|\partial\Pi|,\] and so
$6\gamma e > (2\gamma + 12\gamma^2)|\partial\Pi|$.
Thus, $t$ satisfies the assertion of the lemma.
\medskip

\smallskip

{\bf Case 2.}  Assume that there are only two maximal
contiguity subdiagrams of $\Pi$, say $\Gamma_b$ and $\Gamma_c$  to the sides $CA$ and $AB$, resp. So as in Case 1, each of the contiguity degrees is greater than
$\frac 12 -3\beta-\gamma$.

Denote by $p_1q_1p_2q_2$ the boundary path of $\Gamma_c$, 
by $s_1t_1s_2t_2$ the boundary path of $\Gamma_b$, 
and by $q_1wt_1u$ the boundary path of $\Pi$.
Denote further by $A'$ (respectively, by $D$) the vertex $(p_1)_-$ (resp., the vertex $(s_2)_+$) and by $\Delta'$ the subdiagram bounded by the circle $(A'-B -C - D) s_2^{-1}up_1^{-1}$. Denote by $A'C$ a geodesic line in $\Delta'$ connecting the vertices $A'$ and $C$.
We also use the following notation: $|BA'|=x,\  |A'A|=x',\ |CD|=y,\ |DA|=y'$, so $x+x'=c, y+y'=b$.
Without loss of generality, assume that $x'\le y'$.

\begin{figure*}    \centering    \includegraphics[scale=0.5]{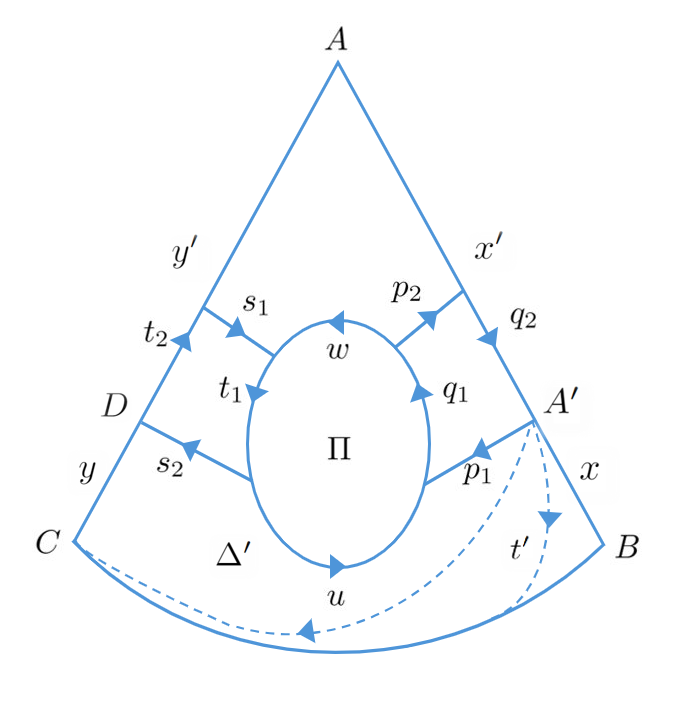}    \caption{Case 2}    \label{fig:enter-label} \end{figure*}

By the inductive hypothesis, the vertex $A'$ of the triangle $T'$ with boundary path $A'-B-C-A'$ can be connected to the side $BC$ with a path $t'$ such that
$$\cln(t')\le \frac 12(x+|CA'|-\lambda^{-1}a)+ 6\gamma(x+|CA'|).$$ 
Since $|u|<\gamma|\partial \Pi|$, we have: $$|CA'|\le y + |s_2^{-1}up_1^{-1}| \le y +(\gamma +2\zeta)|\partial\Pi|),$$ therefore:
$$\cln(t')\le \frac 12\left(x+y+(\gamma +2\zeta)|\partial\Pi|-\lambda^{-1}a\right)+ 6\gamma\left(x+y+(\gamma +2\zeta)|\partial\Pi|\right).$$ 
The path $t=(A-A')t'$ connects the vertex $A$ and the side $BC$, and since $x'\leq y'$, we have: 
\begin{eqnarray*}
    \cln(t) & \le & \cln(t')+x' \\
    & \le & \frac 12(x+y+(\gamma +2\zeta)|\partial\Pi|-\lambda^{-1}a)+ 6\gamma(x+y+(\gamma +2\zeta)|\partial\Pi|)+\frac12 (x'+y') \\ 
    & = & \frac 12(b+c-\lambda^{-1}a)
+ 6\gamma(b+c)+(\frac{1}{2}+6\gamma)(\gamma+2\zeta)|\partial\Pi| - 6\gamma(x'+y')
\end{eqnarray*}
where the last equality uses that $c=x+x'$ and $b=y+y'$.

Next, we note that $(\frac{1}{2}+6\gamma)(\gamma+2\zeta)|\partial\Pi|<\gamma|\partial\Pi|$ is less than $6\gamma(x'+y')$, since $x'>q_2> (1-6\beta)|\partial\Pi|$ by Lemma \ref{lem:q_2 is large}, keeping in mind that the contiguity degree of $\Gamma_c$ is at least $1-3\beta-\gamma$. It follows
that $\cln(t) < \frac 12(b+c-\lambda^{-1}a)
+ 6\gamma(b+c)$, as required.
\medskip

\smallskip

{\bf Case 3.} There are three maximal
contiguity subdiagrams of $\Pi$: $\Gamma_a$, $\Gamma_b$, and $\Gamma_c$  to the sides $BC$, $CA$ and $AB$, respectively. We align the notations for the boundaries of $\Gamma_a$ and $\Gamma_b$ with those from Case 2. We now denote $B'=(p_2)_+, C'=(s_1)_-, v=|BB'|, v'=|B'A|, z=|CC'|, z'=|C'A|$, whence $v+v'=c$ and $z+z'=b$. Assume without loss of generality that $v'\le z'$.

Since $BC$ is $\lambda$-geodesic, $\Pi$ is a $\gamma$-cell, and using Lemma \ref{lem:contiguity sides are short}, we have:
\begin{eqnarray*}
    a & \le & \lambda(v +|B'C'|+z) \\
    & < & \lambda (v+z+(\gamma+2\zeta)|\partial\Pi|) \\
    & < & \lambda (b+c-v'-z' +2\gamma|\partial\Pi|) \\ 
    &\le & \lambda (b+c-2v'+2\gamma|\partial\Pi|)
\end{eqnarray*}
from which we conclude: $$v'\le\frac 12(b+c-\lambda^{-1}a)+\gamma|\partial\Pi|.$$

\begin{figure*}    \centering    \includegraphics[scale=0.5]{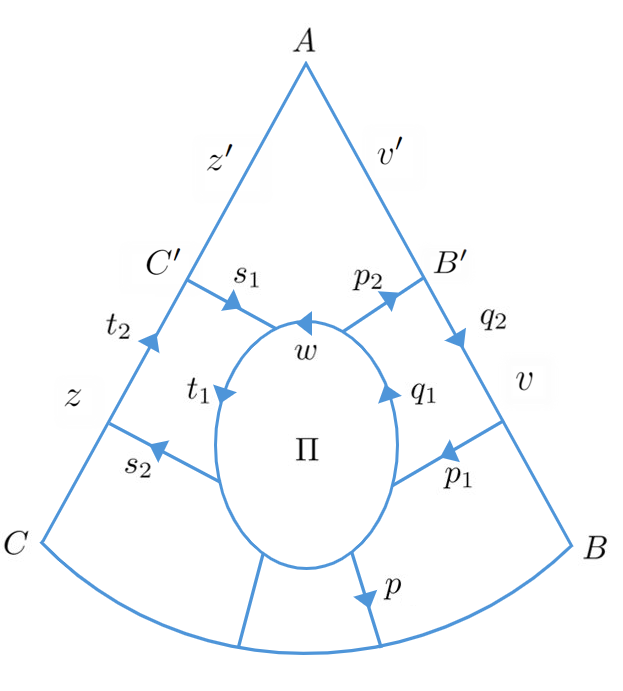}    \caption{Case 3}    \label{fig:enter-label}\end{figure*}

The vertex $B'$ can be connected to $BC$ by a path $r$ as follows. The first segment of $r$ is $p_2^{-1}$, the third segment is a side arc $p$ of the contiguity diagram $\Gamma_a$, and the second segment is a part $p'$ of the boundary $\partial\Pi$.
Since $|p'|< |\partial\Pi|$ and the label of $p'$ is a periodic word, we have by Claim \ref{cla:cln(W)} that $\cln(p')<\zeta |\partial\Pi|$.
It follows that:
\begin{eqnarray*}
    \cln(r) & \le & \cln(p_2)+\cln(p')+\cln(p) \\ & \le & |p_2|+|p|+ \cln(p') \\ & < & 3\zeta|\partial\Pi|.
\end{eqnarray*}
The path $t=(A-B')r$ connects $A$ and the side $BC$, and we have that:
\begin{eqnarray*}
    \cln(t) & \le & v'+\cln(r) \\
    & \le & \frac 12 (b+c-\lambda^{-1}a) + \gamma| \partial\Pi| +3\zeta|\partial\Pi| \\
    & < & \frac 12(b+c-\lambda^{-1}a) + 2\gamma|\partial\Pi|.
\end{eqnarray*}
Since the sum of the contiguity degrees of $\Gamma_b$ and $\Gamma_c$ is greater than $\frac 12-\gamma-3\beta$, by Lemma \ref{lem:q_2 is large}, we obtain that $|q_2|+|t_2|> \frac 13|\partial\Pi|$.
Hence $b+c> \frac 13|\partial\Pi|$, and therefore one can replace the term $2\gamma|\partial\Pi|$ in the last bound by $6\gamma(b+c)$. This completes the proof of the lemma.
\end{proof}

Recall that a word $U$ is called $\lambda$-geodesic for some $\lambda\ge 1$ if for every word $U'$ which is equal to $U$ in $G$, we have $|U|\le \lambda |U'|$.
A word $U$ is called \emph{geodesic} (or \emph{minimal}) if it is $\lambda$-geodesic with $\lambda=1$. Finally, a word $U$ is called \emph{cyclically $\lambda$-geodesic} if every cyclic permutation of $U$ is $\lambda$-geodesic.

\begin{lemma}
\label{lem:conj to cyclic minimal by a short cl}
Let $\Delta$ be a reduced annular diagram over $G$ with contours $p$ and $q$. Assume that the label of $q$ is a word $V$ and the label of $p$ is a word $U$ which is cyclically $\lambda$-geodesic in $G$, $1\le \lambda \le \bar{\beta}^{-1}$. Then the vertex $O=q_-=q_+$ can be connected
to $p$ with a path $t$, where: 
$$\cln(t)\le \frac 12 (|V|-\lambda^{-1}|U|)+9\gamma|V|.$$
In particular, if $U$ is cyclically minimal in $G$, namely, it is not conjugate to any shorter word, then: $$\cln(t)\le \frac 12 (|V|-|U|)+9\gamma|V|.$$
\end{lemma} 

\begin{proof}
By Lemma \ref{lem:short path conneting contours of annual A-map}, the path $p$ can be connected to $q$ by a path $s$, where $|s|<\gamma(|p|+|q|)\le 2\gamma|V|$. The cut along $s$ transforms $\Delta$ into a circular diagram $\Delta'$ with boundary $q_1s^{-1}p'sq_2$, where $q=q_1q_2$ and $p'$ is a cyclic permutation of $p^{-1}$. 
\begin{figure}
\centering    \includegraphics[scale=0.3]{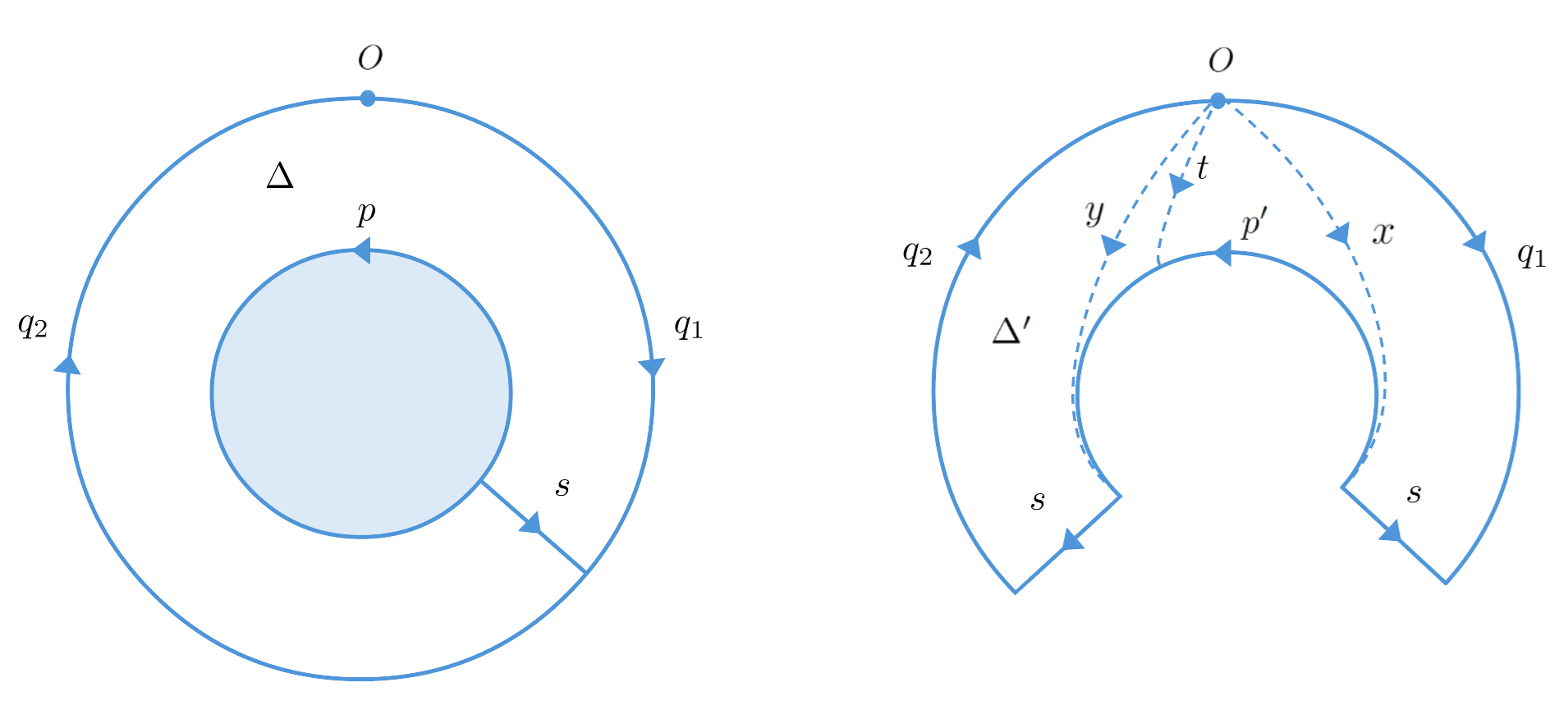}
\caption{The diagrams $\Delta$ (left) and $\Delta'$ (right).} 
\label{fig:enter-label}
\end{figure}

Let $x$ and $y$ be two geodesic paths in $\Delta'$ connecting the vertex $O$ to the vertices $p'_-$ and $p'_+$, respectively. Note that $p'$ is $\lambda$-geodesic by assumption. Therefore, Lemma \ref{lem:connecting O to the triangle floor with short code length} gives a path $t$ in the `triangle' (subdiagram) $xp'y^{-1}$, which connects $O$ to $p'$, and thus connects $O$ to $p$ in
$\Delta$, such that: $$\cln(t)\le \frac 12 (|x|+|y|-\lambda^{-1}|p'|) +6\gamma(|x|+|y|).$$ 
Since the geodesic path $x$ is homotopic to $q_1s^{-1}$,
we have $|x|\le |q_1|+|s|$; similarly, $|y|\le |q_2|+|s|$.
Therefore we obtain:
\begin{eqnarray*}
    \cln(t) & \le & \frac 12 (|q_1|+|q_2|-\lambda^{-1}|p'|+2|s|)+6\gamma(|q_1|+|q_2|+2|s|)\\
    & = & \frac 12 (|V|-\lambda^{-1}|U|) +6\gamma |V| +(1+12\gamma)|s|.
\end{eqnarray*}
It  remains to observe that $(1+12\gamma)|s|\le (1+12\gamma)2\gamma |V|\le 3\gamma|V|$.
\end{proof}

\section{Contrasting generating sets: uniform measures} \label{sec:contrasting uniform}
The current Section is dedicated to proving Theorem \ref{thm:Theorem B}. 
Let $\theta$ be a number between $0$ and $1/2$. Let $X=\{x_1,\dots,x_m\}$ and assume  that $m> 2^{\frac {3}{\theta}}$. 
We call a reduced word $W$ over $X^{\pm 1}$ a \emph{$\theta$-word} if the number of disjoint occurrences of squares $x^{\pm 2}$ ($x\in X$) is at least
$\theta|W|$, where by \emph{disjoint} occurrences we mean occurrences with no overlaps. For the rest of this section, we fix $\theta=0.03$ and $m> 2^{100}=2^{\frac{3}{\theta}}$.

Let $n$ be a large enough odd integer. 
We now construct a partial-Burnside group $G$ satisfying the conditions of Theorem \ref{thm:Theorem B}. 
Let $L_0=R_0=\emptyset$ and let $G(0)$ be the free group $F(X)$. Let $C$ be the set of all $\theta$-words over $X^{\pm 1}$. 
Let $i\in\mathbb{N}$ and suppose that $L_j$, $R_j$, and $G(j)$ have already been defined for all $j<i$.
Let $L_i$ be a maximal subset of $C$ that satisfies conditions (L1)--(L3) (see Definition \ref{def:minimal partial Burnside presentation}). 
Let $R_i= \{A^n \ : \ A \in L_i\}\cup R_{i-1}$ and let:
$$ G(i)=\langle X \mid W=1 \ :\ W\in R_i \rangle. $$ 
This completes the construction of $G(i),L_i,R_i$ for every $i\geq 0$. Finally, let $R=\bigcup_{i=0}^{\infty} R_i$ and consider the group $G$ given by the presentation: 
\begin{equation*}\label{eq:presentation constracting gen sets}
    G=\langle X \mid W=1\ :\ W\in R=\bigcup_{i=0}^{\infty}R_i\rangle.
\end{equation*}
As in previous sections, this is a Burnside-type group. 
We turn to study further properties of $G$. We start with the following observation.
\begin{lemma}\label{lem:not too many theta-words}
The number of $\theta$-words of length $r$ 
is less than~$(2m-~1)^{(1-\frac{\theta}{2})r}$ for $r\gg 1$. 
\end{lemma} 

\begin{proof}
Every $\theta$-word of length $r$ can be obtained from a reduced word $V$ of length $k=\lfloor r-\theta r \rfloor$ as follows.
We mark $r-k$ occurrences of letters in $V$ and replace every marked letter by its square. There are $2m(2m-1)^{k-1}$ reduced words $V$, yielding at most: $$ {k\choose r-k}2m(2m-1)^{k-1}\le 2^r 2m(2m-1)^{k-1} $$ $\theta$-words of length $r$. It is now easy to check that $$2^r2m(2m-1)^{(1-\theta) r-1}<(2m-1)^{(1-\frac{\theta}{2})r}$$ for sufficiently large $r$, since $m> 2^{\frac {3}{\theta}}$. 
\end{proof}

\begin{lemma}\label{lem:bounding the length r torsion (Lemma 6)}
The number of elements in $G$ having finite order and length $r$ with respect to the generating set $X$ is less than  $(2m-1)^{0.995r}$ for $r\gg 1$.
\end{lemma}

\begin{proof}
Let $\calM_r$ denote the set of words of length $r$ that have finite order in $G$. Recall that a word is said to be \emph{cyclically minimal}, if it is not conjugate in $G$ to any shorter word.
We subdivide $\calM_r$ into a union of subsets $M_i$, where a word $V\in M_i$ is conjugate to a cyclically minimal word of length $i$ ($i\le r$). 
If  $V\in M_i$, then by Lemma \ref{lem:finite order element is conjugate to a power of a period}, $V$ is conjugate in $G$ to a power $A^k$ of a period $A$ for some $|k|\leq \frac{n}{2}$. We may assume $|A|\le i$ by Lemma \ref{lem:smooth section}(2). 

By Lemma \ref{lem:not too many theta-words}, there exists some $j_0$ such that, for all $j\geq j_0$, the cardinality of the set of periods $A$ of rank $j$ is less than $(2m-1)^{(1-\theta/2)j}$. (This follows from the definition of the set of periods.) Since $j_0$ is constant, there exists a large constant $C'>0$ such that the cardinality of periods of rank $j$, for $j<j_0$, is smaller than $C'(2m-1)^{(1-\theta/2)j}$. It follows that for some constant $C>0$, the set $\calY_i$ of powers $A^k$ which are conjugate to words from $M_i$ has cardinality $$ \# \calY_i \leq \underbrace{n}_{\text{options for }k} \underbrace{C'\sum_{j=0}^{i} (2m-1)^{(0.985)j}}_{\text{options for }A}\leq C(2m-1)^{0.985i}.$$

If $A^k\in \calY_i$ is a conjugate of a word of length $r$ but not cyclically minimal in $G$,  we replace it with a conjugate word $U$, which is cyclically minimal and $|U|\le i$. Thus, we have a set
$\calZ_i$  of cyclically minimal words such that every word in $\calM_r$ is a conjugate of a word $U\in \calZ_i$ for some $i\le r$, and the cardinality of $\calZ_i$ is at most the cardinality of $\calY_i$. In particular, this is less than $C(2m-1)^{(0.985)i}$. 
Lastly, since words in $\calZ_i$ have length at most $r$, the set $\calZ_i'$ of all cyclic permutations of words from $\calZ_i$ is of size at most $rC(2m-1)^{(0.985)i}$.

By Lemma \ref{lem:conj to cyclic minimal by a short cl}, 
for every word $V\in M_i$, we have a cyclic permutation $U'$ of a word $U\in {\calZ}_i$ (so $U'\in \calZ_i'$) 
and a word $T$ such that $V=T^{-1}U'T$ in $G$ and
$\cln(T)\le\frac 12 (r-i)+9\gamma r=f(r,i)$. For any given $j$, the number of $T$'s with $\cln(T)=j$ is bounded from above by $(2m+2)^j$ (see Claim \ref{cla:cln(W)}). So for any fixed $i$, we have at most:
\[ \sum_{l=0}^{f(r,i)} (2m+2)^l \leq (2m+2)^{f(r,i)+1} \] 
words $T$ which are distinct in $G$, with $\cln(T)\leq f(r,i)$.

Hence, the number of elements in $G$ corresponding to words from $M_i$ is less than:

These are the inequalities that appeared in the previous version:
\begin{eqnarray} \label{eqn:Lemma 6}
    \underbrace{C r (2m-1)^{0.985i}}_{\text{options for }U'}\underbrace{(2m+2)^{f(r,i)+1}}_{\text{options for }T} & = & Cr(2m-1)^{0.985i} (2m+2)^{(9\gamma+1/2)r-i/2+1}   \\ & \leq & 
    Cr(2m-1)^{0.485i} (2m+2)^{(9\gamma+1/2)r+1}  \nonumber \\
    & = & C(2m+2)r\underbrace{(2m-1)^{(0.985+9\gamma)r}}_{\text{(I)}}\underbrace{\left(\frac{2m+2}{2m-1}\right)^{(9\gamma+1/2)r}}_{\text{(II)}}.  \nonumber
\end{eqnarray}
Now exponent $\text{(I)}$ is less than $(2m-1)^{(0.99)r}$ since  one can take
$\gamma < 10^{-4}$ and exponent $\text{(II)}$ is less than $(2m-1)^{0.004r}$, if $m$ is large enough.
Thus, the total product in (\ref{eqn:Lemma 6}) is less than $C(2m+2)r(2m-1)^{(0.994)r}$. 
The sum of such numbers over $i$ is less than $(2m-1)^{0.995r}$, for sufficiently large $r$. The lemma is proved.
\end{proof}

Assuming that the number of generators $m$ and exponent $n$ are large enough, we get the following:

\begin{lemma}[Item (1) of Theorem \ref{thm:Theorem B}.]\label{lem:asym density wrt X is 0}
     $$\lim_{R\to \infty}\Pr_{U_X(R)}(x^n=1)=0.$$
\end{lemma}

\begin{proof}
From Lemma \ref{lem:growth of BT groups}, $\#B_{G,X}(R)>(2m-1)^{0.999R}$ for all $R$ large enough. On the other hand, by Lemma \ref{lem:bounding the length r torsion (Lemma 6)}, there exists some $R_0$ such that for all $r\geq R_0$, the number of elements $x$ of length $r$ satisfying $x^n=1$ is at most $(2m-1)^{0.995r}$. Hence
$$\#\{x\in B_{G,X}(R)\setminus B_{G,X}(R_0)|\ x^n=1\}\leq \sum_{r=R_0+1}^{R} (2m-1)^{0.995r}\leq (2m-1)^{0.996R},$$
if $R$ is sufficiently large.
Notice that $R_0$ is constant, and so $\#B_{G,X}(R_0)\leq C$ for some constant $C$. Therefore
$$\Pr_{U_X(R)}(x^n=1)=\frac{\#\{x\in B_{G,X}(R)|x^n=1\}}{\#B_{G,X}(R)} \leq \frac{(2m-1)^{0.996R}+C}{(2m-1)^{0.999R}}\xrightarrow{R\rightarrow \infty} 0,$$
as claimed.
\end{proof}

\begin{remark}
Since $G$ is a partial-Burnside group, for every $g\in G$, either $g^n=1$, or $g$ has infinite order. It therefore follows from Lemma \ref{lem:asym density wrt X is 0} that also
$$\lim_{R\to \infty}\Pr_{U_X(R)}(x\text{ has finite order in }G)=0.$$
\end{remark}

Define the generating set $Y$ of $G$ by $Y=\{y_1,\dots y_{2m}\}=\{x_1,x_1^2,\dots,x_m,x_m^2\}$. The rest of this section is dedicated to proving Item (2) of Theorem \ref{thm:Theorem B}, namely, that $\lim_{R\to \infty}\Pr_{U_Y(R)}(x^n=1)=1$.

A word $W$ of length $r$ over $Y^{\pm 1}$  is called \emph{$Y$-minimal} if it is not equal in $G$ to any $Y$-word of length $<r$. The substitution $y_{2j-1}\mapsto x_{j},\ y_{2j}\mapsto x_{j}^2$ in a word $W$ provides us with a word $\overrightarrow{W}$
over $X^{\pm 1}$. We call it \emph{the $X$-form} of $W$. We have: 
\begin{lemma}\label{lem:X form of minimal Y-word is reduced}
    The $X$-form $\overrightarrow{W}$ of a $Y$-minimal word $W$ is a reduced $X$-word. 
\end{lemma}

\begin{proof}
One could obtain a cancellation in $\overrightarrow{W}$ if $W$ has a subword $y_{2k-1}y_{2k}^{-1}$. But in this case one could
replace the subword with $y_{2k-1}^{-1}$ and
obtain a shorter word, contrary to the $Y$-minimality
of $W$. Other types of cancellations are also
impossible for similar reasons.
\end{proof} 

\begin{lemma}\label{lem:growth of t-aperiodic words}
Let $m\ge 3$. For every positive real constant $l<4m-4$, there is a positive integer $t$ such that the number of
$t$-aperiodic words $\overrightarrow W$, where $W$ is a reduced word over $Y^{\pm1}$ that has $Y$-length $r$, is at least $l^r$.
\end{lemma} 

\begin{proof}
The proof is very similar to that of Lemma \ref{lem:many aperiodic}, with only the necessary adaptations to dealing with $\overrightarrow W$ instead of $W$. 

For $t\ge 2$, let $b(r)$ be the number of $t$-aperiodic words $W$ in the alphabet $Y$ of length $r$ such that $W$ does not contain
2-letter subwords $yy'$ where both subscripts of the letters $y$ and $y'$
belong to some set $\{2j-1,2j\}$ for $j\le m$.

We first prove that $b(r)> l^r$ if $t$ is large enough. Since $b(1)=4m>l$, it suffices, in turn, to prove by induction on $r$ that $b(r)>lb(r-1)$ for $r\ge 2$.

Given a $t$-aperiodic word $W$ of length $r-1$ as above, we can obtain $4m-4$ reduced  words $Wa$ of length $r$ by adding a letter $a$ from the right such that the subscripts of $a$ and the last letter of $W$ do not belong to the same set $\{2j-1,2j\}$. The total number of such products is
$(4m-4)b(r-1)$. However, we should take into account that some products are not $t$-aperiodic. Let us estimate the number of such `bad' products.

A bad product has the form $Vu^t$, where $|u|\ge 1$. 
Let $|u|=i\geq 1$. Then $|V|=r-ti$ and the number of ways to choose $V$ is $b(r-ti)$, and by the inductive hypothesis, $b(r-ti)<l^{1-ti}b(r-1)$. There are $4m$ words of length $1$. Hence the number of bad products with $|u|=i$ is at most $(4m)^i l^{1-it} b(r-1)$. Therefore the number of all bad products does not exceed:
\[
\sum_{i=1}^{\infty} (4m)^i l^{1-ti} b(r-1) = b(r-1)\frac{4ml}{l^t-4m},
\] which can be arbitrarily small if $t$ is chosen large enough. 
Hence the number of bad words can
be made less than $\kappa b(r-1)$ for an arbitrarily small positive $\kappa$.

It follows that the number of $t$-aperiodic words of length $r$ in $Y$ without forbidden $2$-letter subwords (let us denote this set by ${\calZ}_r$) is greater than $b(r-1)(4m-4-\kappa)$  which is greater than $lb(r-1)$ if one chooses
$\kappa < 4m-4 - l$.

It is clear that for $W\in {\calZ}_r$ the word $\overrightarrow W$ is
$(t+1)$-aperiodic, because every subword $V^k$
of $\overrightarrow W$ contains a subword $\overrightarrow{U}^{k-1}$
where $U^{k-1}$ is a subword of $W$.
So the lemma is proved for $t'=t+1$.
\end{proof} 

We need a couple of standard combinatorial inequalities, both of which are well-known and follow from Stirling's formulae.

\begin{lemma}\label{lem:Stirling's bounds}
\ \ \    \begin{enumerate}
        \item For any constant $0 < \lambda<1/2$, there are positive constants $c$ and $d<2$ such that $\sum_{0\le j\le \lambda r} {r\choose j}\le c d^r$ for every $r>0$.
        \item For any $\mu>1$, there is $\lambda>0$ such that $\sum_{0\le j\le \lambda r} {r\choose j}\le \mu^r$ for every $r>0$
    \end{enumerate} 
\end{lemma} 

The next lemma bounds the density of unbalanced words in the $r$-th sphere in $G$ with compare to $Y$. 

\begin{lemma}\label{lem:very few unbalanced words}
Let ${\calO}_r$ be the set of $Y$-minimal words $W$ of length $r$ where the number of occurrences of even (respectively, odd) $Y$-letters is at
most $0.499 r$.
Then there are constants $K< 4$ and $c>0$ such that $|{\calO}_r| < c(Km)^r$.
\end{lemma}

\begin{proof}
Every word from ${\calO}_r$ is determined  by a choice of $k\le 0.499r$ positions for letters with even subscripts and by the choice of a word of length $k$ in even generators and the choice of a word of length $r-k$ in odd generators. Therefore the cardinality of ${\calO}_r$ is at most
$$\sum_{k\le 0.499 r} {r\choose k} (2m)^k(2m)^{r-k}=(2m^r)\sum_{k\le 0.499R} {r\choose k}.$$
Since $0.499<1/2$ the sum is bounded by $c d^r$, where $d<2$, by Lemma \ref{lem:Stirling's bounds}(1). So the lemma is proved with $K=2d$.
\end{proof}

\begin{definition}[Exponentially generic property]
Let $\{S_i\}_{i\in \mathbb{N}}$ be an infinite sequence of finite sets. We will say that some property is \emph{(exponentially) generic for the series $\{S_i\}_i$}  if there is a
constant $\lambda<1$ such that for all big enough $i$-s, the number of elements in $S_i$ enjoying  this property is greater than $(1-\lambda^i)|S_i|$.  
\end{definition}

Clearly, if two properties $\calP$ and $\calQ$ are generic, then so is the property $\calP \wedge \calQ$. 
We now construct a sequence of finite sets, as follows. 
Let $K_r$ be the set of elements $g\in G$ having $Y$-length $r$. To every $g\in K_r$, we correspond a $Y$-minimal word $W$ representing $g$. We then denote $V=\overrightarrow{W}$, which is a word in $X$, and finally, we choose a cyclically minimal word $U$ in $X$ which is a conjugate of $V$ in $G$. We denote by $S_r$ the set of quadruples $(g, W, V, U)$ for $g\in K_r$; $|S_r|=|K_r|$.

We have:
\begin{lemma}\label{lem:1.499 r < |V|_X <1.501r is generic}
    We have that $|S_r|\geq (4m-5)^r$.  
    Moreover, for a generic quadruple in the sequence $\{S_r\}_{r\in \mathbb{N}}$, $1.499 r < |V|_X <1.501r$. 
\end{lemma} 

\begin{proof} 
The first part follows from Lemma \ref{lem:growth of t-aperiodic words} since for $n>6t$, different $t$-aperiodic elements are not equal in $G$, and so the number of elements in $K_r$ is greater than
$(4m-5)^r$. 

Using this bound and Lemma \ref{lem:very few unbalanced words}, it now follows that for a generic quadruple in $\{S_r\}_{r\in \mathbb{N}}$, the number of occurrences of even (respectively, odd) letter in $W$ is greater than $0.499r$. Recall that every even letter in $W$ corresponds to two letters in $\overrightarrow{W}$, while every odd letter in $W$ to one letter in $\overrightarrow{W}$.
It follows that the length of $V=\overrightarrow{W}$ (in the alphabet $X$) is greater than $2\times  0.499r+0.501r =1.499r$. 
The inequality $|V|_X<1.501r$ follows similarly , and we get:
$1.499 r < |V|_X <1.501r$. 
\end{proof}

From the previous lemma it follows that generically, $V$ is a $0.332$-word. (Indeed, this is since $0.499/1.499>0.332$.) We now show a similar result for the cyclically minimal conjugate, $U$.

\begin{lemma}\label{lem:U 0.03-word is generic}
For a generic quadruple $(g,W,V,U)$ in $\{S_r\}_{r}$, the word
$U$ is a $0.03$-word.
\end{lemma} 

\begin{proof}
Observe that the word $U$ is reduced by assumption, and the word $V=\overrightarrow{W}$ is reduced by Lemma \ref{lem:X form of minimal Y-word is reduced}. Let $\Delta$ be a reduced annular diagram $\Delta$ for the conjugacy of $V$ and $U$ in $G$.  
Denote by $p,q$ the boundaries of $\Delta$, so $\Lab(p)\equiv U$, $\Lab(q)\equiv V$. Let $q_1,\dots,q_k$ be the maximal common subpaths of $p$ and $q^{-1}$ in $\Delta$. So $\sum_{i=1}^k |q_i|=(p,q)$ (see Definition \ref{def:(p,q)}). Denote by $Q_1,\dots, Q_k$ the labels of $q_1,\dots, q_k$ respectively. Recall that $V=\overrightarrow{W}$. So for each $j$ there is a subword $Q'_j$ in $Q_j$ of length $\ge |Q_j|-2$ such that $Q_j'= \overrightarrow{P_j}$ for a subword $P_j$ of $W$.

The proof of the lemma consists of four steps, in each we conclude that the corresponding property from the following list is generic for the sequence $\{S_r\}_r$:
\begin{enumerate}
    \item[Step 1.] $|U|>0.498 r$
    \item[Step 2.] $(p,q)>0.49r$ 
    \item[Step 3.] $\sum |P_j|_Y\geq 0.24r$
    \item[Step 4.] $U$ contains at least $0.05r$ disjoint squares $x_j^{\pm 2}$.
\end{enumerate}
From Step 4., we then conclude that $U$ is a $0.03$-word.

We now begin the proof.

{\bf{Step 1.}} 
By Lemma \ref{lem:conj to cyclic minimal by a short cl}, every element $g\in K_r$ is determined by the pair $(U',T)$, where $U'$ is a cyclic permutation of the word $U$ from the quadruple $(g,W,V,U)$, and $$\cln(T)\le \frac 12 (|V|-|U|)+9\gamma |V|.$$

By Lemma \ref{lem:1.499 r < |V|_X <1.501r is generic}, a generic quadruple has $|V|_X\le 1.501r$, and so for a fixed $i=|U|\le |V|\le 2r$, the number of such pairs $(U',T)$ does not exceed $$\underbrace{(2m-1)^i}_{\text{options for  }U'}\underbrace{(2m+2)^{\frac 12 (1.501r-i)+9\gamma 1.501r}}_{\text{options for }T}.$$

However, if $i\le 0.498r$, this product is exponentially less than $(4m-5)^r$ (assuming $\gamma <10^{-5}$ and $m$ to be sufficiently large). The same is true for the sum over all $i\le 0.0498r$; it is bounded from above by $r(4m-5-\xi)^r$ for some positive $\xi>0$. Since $|S_r|\ge (4m-5)^r$ by Lemma \ref{lem:1.499 r < |V|_X <1.501r is generic}, it follows that a generic quadruple must have $|U|=i>0.498 r$. This completes the proof of step 1.

{\bf{Step 2.} } 
Let $\Delta'$ be the circular diagram obtained by cutting $\Delta$ along $s$, as in done in Lemma \ref{lem:conj to cyclic minimal by a short cl}.
We may assume that  $p'=p$, by choosing $U$ so that the distance in $\Delta$ between the vertex $p_-$ and the path $q$ is minimal.

The paths $p$ and $sq_1^{-1}q_2^{-1}s^{-1}$
connect vertices $p_-$ and $p_+$, and
$$(p, sq_1^{-1}q_2^{-1}s^{-1})\leq (p,q)+2|s|<(p,q)+2\gamma(|p|+|q|).$$ By Lemma \ref{lem:p+ and p- connected with short cln}, there is a path $z$
connecting $p_-$ and $p_+$ with 
\begin{eqnarray*}
    \cln(z)& \le& (p,sq_1^{-1}q_2^{-1}s^{-1})+\beta(|p|+|sq_1^{-1}q_2^{-1}s^{-1}|) \\
    & \le & (p,q)+ 2\gamma(|p|+|q|) + \beta(|p|+|q|)(1+2\gamma) \\
    & < & (p,q)+2\beta(|U|+|V|) \\
    & \le & (p,q)+4\beta|V|
\end{eqnarray*}
Since by Lemma \ref{lem:1.499 r < |V|_X <1.501r is generic},  for a generic quadruple, $|V|>1.501r$, we have that for a generic quadruple, 
$$\cln(U)=\cln(z)< (p,q)+7\beta r.$$

As in step 1, we recall that $g\in K_r$ is determined by the pair $(U,T)$ (here $U'=U$), and using step 1 (namely, that $|U|=i> 0.498r$), the number of such pairs is at most
$$\underbrace{(2m+2)^{\frac 12 (1.501-0.498)r+9\gamma  1.501 r}}_{\text{options for }T}\underbrace{(2m+2)^{(p,q)+7\beta r}}_{\text{options for }U}.$$ 
Now, if $(p,q)\le 0.49r$, then the total exponent of $(2m+2)$ in this expression is less than $0.999r$. 
Using again that $|S_r|\ge (4m-5)^r$ by Lemma \ref{lem:1.499 r < |V|_X <1.501r is generic}, we get that $(p,q)>0.49 r$ is generic. This completes the proof of step 2.

{\bf{Step 3.}} 
Let $s=1,\dots, k-1$. Then $q_s$ and $q_{s+1}$ are separated in $q$, i.e. there is a subpath $q_sv_sq_{s+1}$ in $q$ and a circular subdiagram $\Delta_s$ of $\Delta'$ with boundary $v_sv'_s$, where $v'_s$ is a subpath of $p$, see Figure \ref{fig:(p,q)}.
\begin{figure}
    \centering
    \includegraphics[width=0.5\linewidth]{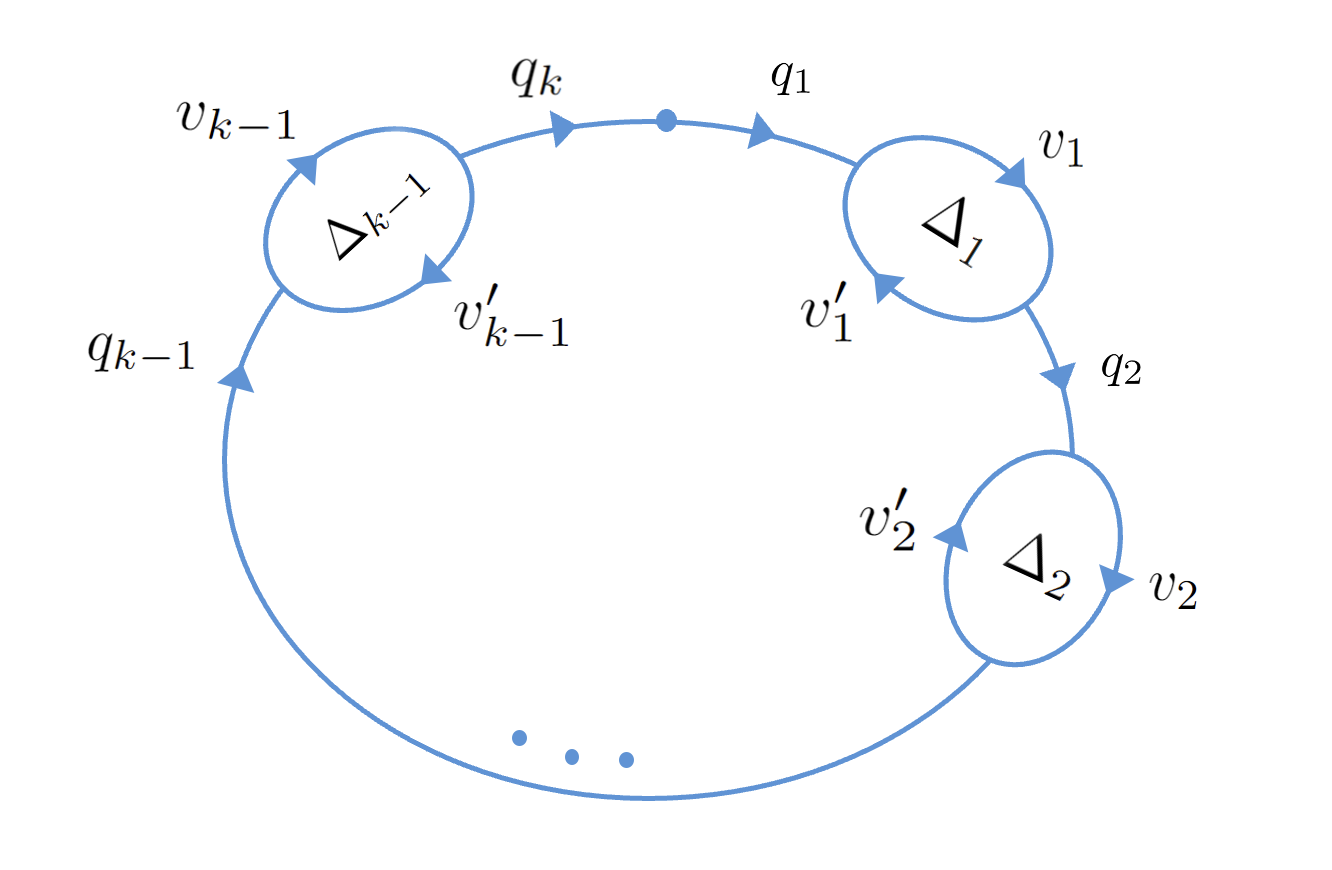}
    \caption{Step 3.}
    \label{fig:(p,q)}
\end{figure}
Since $p$ and $q$ are reduced paths, the diagram $\Delta_s$ has to contain at least one cell, and so the perimeter $|v_s v'_s|$ of $\Delta_s$ is greater than $\bar\beta n$ by Lemma \ref{lem:boundary of a cell shorter than of the diagram}. 
Thus, $$(k-1)\bar \beta n = \sum_{s=1}^{k-1}|v_sv_s'|  \le |p|+|q|\le 2|V|< 3.002 r,$$ and therefore $$k< 3.002\bar\beta^{-1}n^{-1}r+1<4r/n$$ for $r\ge 2$.

By step 2, generically, $\sum |Q_j|=(p,q)>0.49 r$. Together with the bound on $k$, this gives: $$\sum |P_j|_Y\ge \sum |Q'_j|_X/2 \ge
\sum |Q_j|_X/2-2k > 0.49r/2 - 8r/n> 0.24 r.$$ This completes the proof of step 3.

{\bf{Step 4.}} It is enough to bound the number of quadruples in which the total number of disjoint squares $x_j^{\pm 2}$ in the labels $Q_1,\dots, Q_k$ is less than $0.05r$.

Let $(g,W,V,U)$ be a quadruple for which the total number of disjoint squares in $Q_1,\dots, Q_k$ is less than $0.05r$. So the total number of `even' letters in the labels $P_1,\dots,P_k$ is less than $0.05r$. Since $\sum |P_j|_Y\ge 0.24r$ (step 3), this would mean that the word $P=P_1\dots P_k$ is strongly unbalanced. By Lemma \ref{lem:very few unbalanced words}, the number of possible products $P=P_1\dots P_k$ with this property, for a fixed $s=|P|_Y$, does not exceed $c(4m\rho)^s$, for some $\rho<1$.

Now $W$ can be constructed as follows.
We fix the positions of the subwords $P_j$ in the word of $Y$-length $r$ by inserting $k\le 4r/n$ pairs of commas. Let $n_s\le{n+2k\choose 2k}$ be the number of all possible distributions of commas. We fill in the intervals between neighbor commas with odd and even numbers with the parts $P_j$ of $P$. Finally, we put arbitrary letters from $Y^{\pm 1}$ to the remaining positions. The number of possible words $W$ that we can obtain this way, does not exceed $$\underbrace{n_s}_{\substack{\text{comma} \\ \text{displacements}}
}\underbrace{c\rho^s(4m)^s}_{\substack{\text{options} \\ \text{for $P$}}
}\underbrace{(4m)^{r-s}}_{\substack{\text{remaining} \\ \text{positions}}}.$$
Since $s>0.24r$ (step 3) we have a number $\rho'<1$, such that this value is at most $n_sc(\rho'4m)^r$.
Since the choice of $n$ makes the ratio $k/r$  as small as we want, it follows from Lemma \ref{lem:Stirling's bounds}(2) that $\sum_{s\le r} n_s$ can be assumed bounded by a function $f(r)=d \lambda^r$ with $\lambda-1$ as close to $0$ as desired, and it follows that the number of words $W$ under consideration is bounded by $$cd (\rho''4m)^r,$$ where $\rho''<1$.
If one chooses the constant $m$ large enough, this value becomes  exponentially negligible in comparison with $(4m-5)^r$.
This completes the proof of step 4.

Finally, since generically $|U|\le |V|<1.501 r$ (Lemma \ref{lem:1.499 r < |V|_X <1.501r is generic}), the word $U$ is a $0.03$-word.
\end{proof}

\begin{lemma}\label{lem:generically, k=pm 1}
    For every quadruple $(g,W,V,U)\in S_r$ let $A$ be a period, or a simple element, such that $U$ is conjugate to a power $A^k,k\in \mathbb{Z}$. Then generically, $k=\pm 1$.
\end{lemma}
\begin{proof}
We may assume that $A$ is a simple word and and $|k|\ge 2$. Let $\Delta$ be a reduced diagram for conjugacy
of $V$ and $A^k$. By Lemma \ref{lem:smooth section}(2), we may assume that $A^k$ is $\bar{\beta}^{-1}$-geodesic. 
Therefore we can use Lemma \ref{lem:conj to cyclic minimal by a short cl}, to get that $V$ is equal in $G$ to $T (A')^kT^{-1}$, where $A'$ is a cyclic permutation of $A$ and
$$\cln(T)\le \frac 12
(|V|-\bar{\beta}|A^k|)+9\gamma|V| \le \frac 12
(|V|-2\bar\beta |A|)+9\gamma|V|.$$

There are at most
$(2m+2)^{\frac 12 |V|-2\bar\beta|A|+9\gamma|V|}$ different such $T$-s and at most $(2m)^{|A|}$ choices for $A'$ (at first we fix $|A|=j$).
Therefore, the number of such quadruples is at most $$(2m+2)^{\frac 12|V| +9\gamma|V|+\beta|A|}\le (2m+2)^{(\frac 12 +3\beta)|V|}.$$
Using LPP to bound $\beta$, and since by Lemma \ref{lem:1.499 r < |V|_X <1.501r is generic}, $|V|>1.501r$ generically, this is at most $$ (2m+2)^{0.51\times 1.501r}\le (2m+2)^{0.77r}$$  Then we sum over all possible $k,j$, $2\le |k|<n$, and $1\le j<0.51r$, and yet get an exponentially negligible number in comparison with the cardinality of $S_r$, which is at least $(4m-5)^r$ (see Lemma \ref{lem:1.499 r < |V|_X <1.501r is generic}). This completes the proof.
\end{proof}

Finally, we show
\begin{lemma}\label{lem:asym density torsion wrt Y is 1}
A generic element $g$ of $Y$-length at most $r$ in $G$ has finite order.
\end{lemma} 
\begin{proof}
Consider a generic element $g\in G$, and its corresponding quadruple, $(g,W,V,U)$. By Lemma \ref{lem:generically, k=pm 1}, we may assume that $U$ is not conjugate to a power $A^k$, $k\ne \pm 1$ of a shorter word $A$.
By the definition of periods and by Lemma \ref{lem:U 0.03-word is generic}, the word $U$ is therefore a conjugate of a period or the inverse of a period. Therefore, the element $g$ satisfy $g^n=1$. 
\end{proof}
Theorem \ref{thm:Theorem B} now follows from Lemma \ref{lem:asym density wrt X is 0} and Lemma \ref{lem:asym density torsion wrt Y is 1}.

\section{Contrasting generating sets: co-prime exponents}\label{sec:co-prime}
This section is dedicated to proving Theorem \ref{thm:two different exponents}.

The construction is similar to that of the previous section, where here finite order and infinite order elements are now replaced by elements of order dividing $n_1$ and elements of order dividing $n_2$, respectively.

We will make use of a few definitions and notations from the previous section. In particular, we will use the generating sets $X=\{x_1,\dots,x_m\}$ and $Y=\{x_1,x_1^{2},x_2,x_2^2,\dots,x_m,x_m^{2}\}$, and call a reduced word $W$ a $\theta$-word if the number of disjoint occurrences of squares $x^{\pm 2}$ ($x\in X$) is at least $\theta|W|$. As in Section \ref{sec:contrasting uniform}, we fix $\theta=0.03$ and $m>2^{100}$.

Let $C$ be the set of all $0.03$-words over $X^{\pm 1}$, and denote by $C^c$ the complement of $C$ (namely, all reduced words over $X^{\pm 1}$ that are not $0.03$-words). We now construct by induction, for each $i\ge 0$, sets $L_i$, $R_i$, $T_i$ and a group $G(i)$.

Let $L_0=R_0=T_0=\emptyset$ and let $G(0)=F(X)$. For $i\in \mathbb{N}$, suppose that $G(i-1)$ has been defined. 
Let $L_i$ be a maximal set of words satisfying (L1)--(L3) with respect to $G(i-1)$.

Let $R_i=\{A^{n_1} \ : \ A\in L_i\cap C\}\cup R_{i-1}$ and $T_i=\{A^{n_2} \ : \ A\in L_i\cap C^c\}\cup T_{i-1}$. 
Let $G(i)=\langle X \ | \ W=1 \ : \ W\in R_i\cup T_i\rangle$.
This completes the inductive construction.

Let now $R=\bigcup_{i=0}^\infty R_i$ and $T=\bigcup_{i=0}^\infty T_i$, and consider the group $G$ given by 
$$G=G(\infty)=\langle X \ | \ W=1 \ : \ W\in R\cup T\rangle.$$
We will refer to words in $L_i\cap C$ as rank $i$ $n_1$-periods, and to words in $L_i\cap C^c$ as rank $i$ $n_2$-periods.

By construction, $G$ is a Burnside-Type group, and so all lemmas from Section \ref{sec:geometry of BT groups} hold for $G$. 

The next lemma follows from \cite[Lemma 25.2]{OL-book}.

\begin{lemma}\label{lem:every word is conj to a period of some type}
    Every word is conjugate in $G$ to either a power of an $n_1$-period or to a power of an $n_2$-period. It has order dividing $n_1$ in $G$ in the former case, and order dividing $n_2$ in $G$ in the latter. In particular, $g^{n_1n_2}=1$ for every $g\in G$.
\end{lemma}

In analogy to Lemma \ref{lem:bounding the length r torsion (Lemma 6)}, we now have:
\begin{lemma}\label{lem:bounding the length r type 1 torsion (Lemma 6)}
    The number of elements in $G$ having order dividing $n_1$ and length $r$ with respect to the generating set $X$ is less than $(2m-1)^{0.995r}$ for sufficiently large $r$-s.
\end{lemma}

\begin{proof}
    The proof is identical to that of Lemma \ref{lem:bounding the length r torsion (Lemma 6)} were here `periods' is replaced by `$n_1$-periods' and `$n$' by `$n_1$'.
\end{proof}

Suppose that $m,n$ are large enough, we therefore get: 
\begin{lemma}\label{lem:asym density of type 1 torsion wrt X is 0}
    The asymptotic density (with respect to the generating set $X$) of elements with order dividing $n_2$ is $1$. Namely:
    $$\lim _{R\to \infty}\Pr_{U_X(R)}(x^{n_2}=1)=1.$$\end{lemma}
\begin{proof}
    We first show that $\lim _{R\to \infty}\Pr_{U_X(R)}(x^{n_1}=1)=0$.
    Indeed, this follows directly from the previous lemma, Lemma \ref{lem:bounding the length r type 1 torsion (Lemma 6)}, keeping in mind that by Lemma \ref{lem:growth of BT groups}, $\# B_{G,X}(R)$ is at least $(2m-1)^{0.999 R}$. 

    It now follows by the `in particular' part of Lemma \ref{lem:every word is conj to a period of some type}, and since $n_1$ and $n_2$ are co-prime, that
    $\lim _{R\to \infty}\Pr_{U_X(R)}(x^{n_2}=1)=1$.
\end{proof}

As in the previous section, we use the substitution $y_{2j-1}\mapsto x_j$, $y_{2j}\mapsto x_j^2$ we can transform every word $W$ over $Y^{\pm 1}$ to a word $\overrightarrow{W}$ over $X^{\pm 1}$, called the $X$-form of $W$. It is reduced, assuming that $W$ was $Y$-minimal. Also here, $K_r$ will be the set of elements $g\in G$ having $Y$-length $r$. To every $g\in K_r$, $W$ is a minimal word in $Y$ representing $g$. For every word $V=\overrightarrow{W}$ in $X$, we choose a cyclically minimal word $U$ in $X$ which is a conjugate of $V$ in $G$. We use $S_r$ to denote the set of quadruples $(g, W, V, U)$ for $g\in K_r$. We note that Lemma \ref{lem:1.499 r < |V|_X <1.501r is generic}, Lemma \ref{lem:U 0.03-word is generic}, and Lemma \ref{lem:generically, k=pm 1}, regarding properties that are generic for the sequence $\{S_r\}$, work here with identical proofs. 

We conclude the section by showing:
\begin{lemma}
    A generic element $g$ of $Y$-length $r$ in $G$ has order dividing $n_1$. In particular: $$\lim _{R\to \infty}\Pr_{U_Y}(x^{n_1}=1)=1.$$
    
\end{lemma}
\begin{proof}
    Consider a generic element $g\in G$, and its corresponding quadruple, $(g,W,V,U)$. By Lemma \ref{lem:generically, k=pm 1}, we may assume that $U$ is not conjugate to a power $A^k$, $k\ne \pm 1$ of a shorter word $A$. By Lemma \ref{lem:U 0.03-word is generic}, we may assume that $U$ is a $0.03$-word. By construction, it then follows that $U^{n_1}=1$, and therefore that $g^{n_1}=1$.
\end{proof}

\begin{remark}
    In fact, the assumption that $n_1,n_2$ are co-prime is not necessary for the construction. Indeed, repeating the above construction merely assuming that $n_1, n_2>n$ are odd, the resulting group satisfies the statement of Theorem \ref{thm:two different exponents}, even replacing $x^{n_j}=1$ by $\ord(x)=n_j$ for $j=1,2$. 
\end{remark}

\section{The spectrum of torsion densities}
In this section we prove Theorem \ref{thm:Theorem D}. The proof for the case of random walks (see Corollary \ref{cor:spectrum[0,1]}) is based on the construction given in Section \ref{sec:different generating sets}, and on the fact that the set of limits of the sequence $\{\Pr((X_i)^n=1)\}_{i=1}^{\infty}$ is connected (Lemma \ref{lem:connected}). The case of uniform measures on Cayley balls is treated separately in Subsection \ref{subsec:spectrum uniform measures}.

\subsection{Random walks}
For two probability measures $\mu,\nu$ on a measurable space $(X,\mathcal{F})$ (here $\mathcal{F}$ is a $\sigma$-algebra) we let their total variation distance be $d_{\text{TV}}(\mu,\nu) = \sup_{A\in \mathcal{F}} |\mu(A) - \nu(A)|$.
The following is standard, and we include it here with a proof for the reader's convenience. For two real numbers $a,b$ we write $a \approx_\varepsilon b$ if $|a-b| \leq \varepsilon$. (Notice that in the following lemma, $n$ is a running index independent of the exponent $n$ that we use in the rest of the paper.)

\begin{lemma} \label{lem:TV0}
Let $\mu$ be a probability measure on a group $G$ with $\mu(1)>0$. Then $d_{\text{TV}}(\mu^{*(n+1)},\mu^{*n}) \xrightarrow{n\rightarrow \infty} 0$. 
\end{lemma}

\begin{proof}
Let $\mu$ be a measure on $G$ and let $p=\mu(1)>0,\ q=1-p$.
Let $X_1,X_2,\dots$ be independent identically distributed (i.i.d) random variables on $G$ with distribution $X_i\sim \mu$ for every $i\in \mathbb{N}$. Fix a measurable set $A\subseteq G$. Let $a_n = \Pr(X_1\cdots X_n\in A)=\mu^{*n}(A)$. Let $a'_n = \Pr(X_1\cdots X_n\in A|X_1,\dots,X_n\neq 1)$. By the law of total probability:
\[
a_n = \sum_{k=0}^{n} {n \choose k} p^k q^{n-k} a'_{n-k},\ \ \   
a_{n+1} =  \sum_{k=0}^{n+1} {n+1 \choose k} p^k q^{n+1-k} a'_{n+1-k}.
\]
Furthermore, let $X\sim Bin(n,p)$ be a binomial random variable, counting successes (with probability $p$) among $n$ independent trials. Then, by the Chernoff bound (see e.g. \cite[Corollary 4.6]{Chernoff}): \[ \sum_{|k-pn|\geq n^{2/3}} {n \choose k} p^k q^{n-k} a'_{n-k} \leq \sum_{|k-pn|\geq n^{2/3}} {n \choose k} p^k q^{n-k} = \Pr(|X-pn|\geq n^{2/3}) \leq \underbrace{2e^{-cn^{1/3}}}_{=:\varepsilon} \] for some $c>0$ independent of $n$. Therefore:
\[
a_n \approx_\varepsilon \sum_{|k-pn|\leq n^{2/3}} {n \choose k} p^k q^{n-k} a'_{n-k},\ \ \   
a_{n+1} \approx_\varepsilon
\sum_{|k-pn|\leq n^{2/3}} {n+1 \choose k} p^k q^{n+1-k} a'_{n+1-k}
\]
so (with $j=n-k,n+1-k$ in the left and right sums, respectively\footnote{Notice that in the right sum, $|k-pn|=|j-qn+1|$ rather than $|j-qn|$, but this can be ignored by updating $c$ (still independently of $n$).}):
\begin{eqnarray}
a_{n+1} - a_n & \approx_{2\varepsilon} & \sum_{|j-qn|\leq n^{2/3}} a'_j \left( {n+1 \choose j} p^{n+1-j}q^j - {n \choose j} p^{n-j} q^j \right) \nonumber \\ & = & \sum_{|j-qn|\leq n^{2/3}} {n \choose j} p^{n-j}q^j a'_j \left(\frac{n+1}{n+1-j} p - 1\right). \label{eqn:TV1}
\end{eqnarray}
Notice that if $qn-n^{2/3}\leq j\leq qn+n^{2/3}$ then:
\[
\frac{1}{p} - \frac{K}{n^{1/3}} \leq \frac{n}{pn+2n^{2/3}} \leq \frac{n+1}{n+1-j}\leq \frac{n}{pn - n^{2/3}} \leq \frac{1}{p} + \frac{K}{n^{1/3}}
\]
for some constant $K$ (depending on $p,q$ but not on $n$). By (\ref{eqn:TV1}):
\[
a_{n+1} - a_n \approx_{2\varepsilon} \sum_{|j-qn|\leq n^{2/3}} {n \choose j} p^{n-j}q^j a'_j \left(\frac{n+1}{n+1-j} p - 1\right) \approx_\frac{2K}{n^{1/3}} 0,
\]
since $\sum_{j=0}^n {n\choose j}p^{n-j}q^{j}a'_j \leq 1$. It follows that:
\[
|\mu^{*(n+1)}(A) - \mu^{*n}(A)|=|a_{n+1}-a_n|\leq 2\varepsilon + \frac{2K}{n^{1/3}},
\]
a bound which decays to zero as $n\rightarrow \infty$ and which does not depend on $A$. The result follows.
\end{proof}

\begin{lemma} \label{lem:connected}
Let $\{X_i\}_{i=1}^{\infty}$ be a lazy random walk on a group $G$. Let $A\subseteq G$ be an event. Then the set of limit points of $\{\Pr(X_i\in A)\}_{i=1}^{\infty}$ is connected.
\end{lemma}
\begin{proof}
Let $\{a_i\}_{i=1}^{\infty}$ be a bounded sequence of real numbers such that $a_{i+1}-a_i\xrightarrow{i\rightarrow \infty} 0$. Let $C$ be the set of limit points of $\{a_i\}_{i=1}^{\infty}$. This is a closed subset of $\mathbb{R}$. We claim that $C$ is connected. Indeed, let $w<x<y<z$ be real numbers such that $w,z\in C$. Let $\delta = \frac{y-x}{2}$. Let $m\gg 1$ be such that $|a_{i+1}-a_i|<\delta$ for all $i\geq m$. Let $n_2 > n_1\geq m$ be such that $|w-a_{n_1}|,|z-a_{n_2}|<\delta$. Then there exists some $n_1 \leq n' \leq n_2$ such that $x<a_{n'}<y$.
The claim now follows using Lemma \ref{lem:TV0}.
\end{proof}

\begin{corollary} \label{cor:spectrum[0,1]}
There exists a finitely generated group $G$, $n\in \mathbb{N}$, and a finitely supported, non-degenerate, symmetric, lazy random walk $\{X_i\}_{i=1}^{\infty}$ on $G$ such that every real number in $[0,1]$ occurs as a partial limit of the sequence $\{\Pr((X_i)^n=1)\}_{i=1}^{\infty}$.
\end{corollary}
\begin{proof}
By Lemma \ref{lem:torsion density RW S oscillation} there exists a group $G$ as desired with $\limsup_{i\rightarrow \infty} \Pr((X_i)^n=1) = 1$ and $\liminf_{i\rightarrow \infty} \Pr((X_i)^n=1) = 0$. The claim now follows using Lemma \ref{lem:connected}.
\end{proof}

\begin{remark}\label{rem:conn not for cayley}
Lemma \ref{lem:TV0} does not extend to the sequence of uniform measures on balls in Cayley graphs. For instance, consider the free group $F_m$ on $m>1$ generators. Let $\mu_n$ be the uniform measure on the ball of radius $n$ in the Cayley graph of $F_m$ with respect to the standard generating set. Let $A=\bigcup_{i=1}^{\infty} \mathbb{S}(2i)$ be the union of all spheres in the Cayley graph of even radii. Then $\mu_{2n}(A)-\mu_{2n-1}(A)\geq c>0$ for some constant $c$, since the radius-$2n$ sphere occupies a positive (bounded away from zero) fraction of the radius-$2n$ ball. Thus, for uniform measures, we cannot apply the argument from Corollary \ref{cor:spectrum[0,1]}. This case requires a more delicate analysis, which will be done in the next subsection.
\end{remark}

\subsection{Uniform measures}\label{subsec:spectrum uniform measures}
We now turn to construct a Burnside-type group satisfying the conditions of Theorem \ref{thm:Theorem D} for the case of uniform measures on Cayley balls. 
In fact, the group we construct below satisfies a stronger property, as in the following statement:

\begin{theorem}\label{thm:spectrum strong version}
    Let $n\in \mathbb{N}$ be large enough. Then there exists a finitely generated, torsion group $G=\langle S\rangle$, such that for every odd integer $N\ge n$, every real number $\alpha\in [0,1]$ is a partial limit of the sequence $$\{\Pr_{U_S(R)}(x^N=1)\}^{\infty}_{R=1}.$$
\end{theorem}

In other words, all sufficiently large odd integers appear simultaneously as probabilistic exponents for $G$, while for each $N\geq n$, the density of elements of order $N$ in $G$ oscillates, approaching infinitely often any value in $[0,1]$.

The proof of Theorem \ref{thm:spectrum strong version} is brought after explaining the construction.

Fix an alphabet $S=\{x_1,\dots,x_m\}$, and $n \gg 1$ odd. Denote by $\mathbb{N}^{odd}_{\geq n}$ the set of odd integers $\geq n$.
Let $(\alpha_0,N_0),(\alpha_1,N_1),(\alpha_2,N_2),\dots$ be an enumeration of all pairs in $\left(\mathbb{Q}\cap [0,1]\right)\times \mathbb{N}^{odd}_{\geq n}$.

We will now construct a Burnside-type group $G$ by induction, where at the $k$-th inductive step ($k=0,1,\dots$), the group $G(r_k)$ will be defined so that 
\begin{eqnarray}\label{eq:spectrum}
    \Big|\Pr_{U_{G(r_k),S}(r_k)}(g^{N_k}=1)- \alpha_k\Big|\le 2^{-k}.
\end{eqnarray}
The sequence $0=r_0<r_1<\dots $ will be determined throughout the construction. 

Let $r_0=0$ and $G(0)=F(S)$. Since $2^{-0}=1$, (\ref{eq:spectrum}) is satisfied for $k=0$. 
For $k>0$ suppose that $r_{k-1}$ and $G(r_{k-1})$ have already been constructed such that (\ref{eq:spectrum}) holds for $k-1$.

By Lemma \ref{cla:finite rank group has asymptotic torsion density 0}, $G(r_{k-1})$ has asymptotic torsion density approaching to $0$. Therefore, there exists $r>0$ large enough such that the ball of radius $r$ in $G(r_{k-1})$ has torsion density less than $\alpha_k$.

We now turn to construct a Burnside-type group to be called $H$. Set $H(r)=G(r_{k-1})$. 
By induction on $i$, we choose for every $i>r$, a maximal set of periods $L_i$ satisfying conditions (L1)--(L3) with respect to the group $H(i-1)$, and set $n_{A}=N_k$ for all periods $A\in L_i$. By Lemma \ref{cla:maximal period sets gives asymptotic torsion density 1}, we have that:
$$\lim_{i\to \infty}\Pr_{U_{H(i),S}(i)}(g^{N_k}=1)=1.$$
Therefore, there exists $r'>r$ such that $\Pr_{U_{H(r'),S}(r')}(g^{N_k}=1)$ is larger than $\alpha_k$.  
Let $H=H(r')$. So $H$ is a partial-Burnside group of rank $r'$.

We now modify $H$, dropping out several periods starting with periods of maximal length, $r'$, and then, if needed, dropping periods of length $r'-1$ and so on. This is still a Burnside-type group. By Lemma \ref{lem:bound on conjugates of powers of single period}, and since the size of the ball $B_{H}(r')$ is at least $(2m-1)^{0.9r'}$ (by Lemma \ref{lem:growth of BT groups}), dropping a single period $A$, with exponent $n_A=N_k$, decreases the  
density
$\Pr_{U_{H,S}(r')}(g^{N_k}=1)$
by at most: $$\frac{(2m-1)^{0.6r'}}{(2m-1)^{0.9r'}}=(2m-1)^{-0.3r'}$$ 
This is less than $2^{-k}$ if $r'$ is large enough. 
It follows that one can drop periods one by one,  until the torsion density of the ball of radius $R'$ in the resulting group is close to $\alpha_k$ with error up to $2^{-k}$. 

We let $r_k=r'$ and let $G(r_k)$ to be the group obtained from $H$ after dropping out the right number of periods. It is clear that (\ref{eq:spectrum}) is now satisfied for $k$.

This completes the inductive construction of the sequences $\{r_k\}_k$ and $\{G(r_k)\}_k$. We let $G=G(\infty)$.


\begin{proof}[Proof of Theorem \ref{thm:spectrum strong version}]
Consider the group $G$ constructed above. By Corollary \ref{cor:R-th ball is determined in low rank}, for every $k\in \mathbb{N}$, we have $$\Pr_{U_{(G,S)}(r_k)}(g^{N_k}=1)=\Pr_{U_{(G(r_k),S)}(r_k)}(g^{N_k}=1).$$ The right hand side is equal to $\alpha_k$ up to an error of at most $2^{-k}$, by (\ref{eq:spectrum}). This completes the proof, since for every real $\alpha\in [0,1]$ and every integer $r$, there exists $k>r$ such that $|\alpha-\alpha_k|< 2^{-r}$.
\end{proof}

We conclude the section with the proof of Theorem \ref{thm:Theorem D}.
\begin{proof}[Proof of Theorem \ref{thm:Theorem D}]
The case of random walks was completed in Corollary \ref{cor:spectrum[0,1]}. The case of uniform measures on balls in Cayley graphs follows from Theorem \ref{thm:spectrum strong version}, by considering $N=n$. 
\end{proof}

\end{document}